\documentclass[11pt,a4paper]{article}

\usepackage{fullpage}

\usepackage{amsmath,amssymb}

\usepackage{booktabs} % For better table lines

\usepackage{adjustbox} % To constrain height

\usepackage{mathtools}

\usepackage{graphicx,adjustbox,booktabs}
\newsavebox\ordertable
\newlength\panelheight

\usepackage{bbm}

\usepackage{amsopn,amsthm}

\usepackage{subcaption}

\usepackage{graphicx} 

\usepackage{color}

\usepackage{psfrag}

\usepackage{mathrsfs} 

\usepackage{bm}
\usepackage{cases}
\usepackage{hyperref}
\hypersetup{colorlinks=true, urlcolor=blue, citecolor=blue, linkcolor=blue}

\newtheorem{thm}{Theorem}[section]
\newtheorem{lem}[thm]{Lemma}
\newtheorem{proposition}[thm]{Proposition}
\newtheorem{cor}[thm]{Corollary}

\theoremstyle{remark}
\newtheorem{rem}[thm]{Remark}

\newcommand{\Lbfg}{\widehat{\bfg}}
\newcommand{\Lbfu}{\widehat{\bfu}}
\newcommand{\Lbfv}{\widehat{\bfv}}
\newcommand{\Lbfz}{\widehat{\bfz}}
\newcommand{\Lbfw}{\widehat{\bfw}}
\newcommand{\field}[1]{{\mathbb{#1}}}
\newcommand{\C}{\field{C}}
\newcommand{\N}{\field{N}}
\newcommand{\R}{\field{R}}

\DeclareMathOperator{\Fcal}{\mathcal{F}}
\newcommand{\Ocal}{\mathcal{O}}

\newcommand{\bs}{\boldsymbol} 
\newcommand{\bfa}{{\bs a}}
\newcommand{\bfd}{{\bs d}}
\newcommand{\bfe}{{\bs e}}
\newcommand{\bff}{{\bs f}}
\newcommand{\bfg}{{\bs g}}
\newcommand{\bfh}{{\bs h}}
\newcommand{\bfj}{{\bs j}}
\newcommand{\bfu}{{\bs u}}
\newcommand{\bfv}{{\bs v}} 
\newcommand{\bfw}{{\bs w}}
\newcommand{\bfx}{{\bs x}}
\newcommand{\bfy}{{\bs y}} 
\newcommand{\bfz}{{\bs z}}
\newcommand{\bfA}{{\bs A}}
\newcommand{\bfE}{{\bs E}}
\newcommand{\bfH}{{\bs H}}
\newcommand{\bfJ}{{\bs J}}
\newcommand{\bfm}{{\bs m}}
\newcommand{\bfnu}{{\bs\nu}}
\newcommand{\bfzeta}{{\bs\zeta}}
\newcommand{\bfphi}{{\bs\varphi}}
\newcommand{\Lbfphi}{\widehat{\bs\varphi}}
\newcommand{\dx}{\, \dif \bfx}
\newcommand{\ds}{\, \dif s}
\newcommand{\dt}{\, \dif t}
\newcommand{\xhat}{\widehat{\bfx}}
\newcommand{\rmi}{\mathrm{i}} 
\newcommand{\rme}{\mathrm{e}}

\DeclareMathOperator{\curl}{{\mathbf{curl}}}
\DeclareMathOperator{\curlx}{{\mathbf{curl}}_{\bfx}}
\DeclareMathOperator{\scurl}{\mathbf{Curl}}
\DeclareMathOperator{\sCurl}{Curl}
\DeclareMathOperator*{\argmin}{arg\,min}
\DeclareMathOperator{\sgrad}{\mathbf{Grad}}
\DeclareMathOperator{\divop}{\mathrm{div}}
\DeclareMathOperator{\sdiv}{\mathrm{Div}}
\newcommand{\Ei}{\bfE^i}
\newcommand{\Hi}{\bfH^i}
\newcommand{\Es}{\bfE^s}
\newcommand{\Hs}{\bfH^s}
\newcommand{\Et}{\bfE}
\newcommand{\Ht}{\bfH}
\newcommand{\LEs}{\widehat{\Et}^s}
\newcommand{\LEp}{\widehat{\bfE}'}
\newcommand{\LEt}{\widehat{\Et}}
\newcommand{\LEi}{\widehat{\Et}^i}
\newcommand{\LHs}{\widehat{\Ht}^s}
\newcommand{\HcurlOmega}{H(\curl, \Omega)}
\newcommand{\HdivO}{H^{-{1}/{2}}(\sdiv, \Gamma)}
\newcommand{\HcurlO}{H^{-{1}/{2}}(\sCurl, \Gamma)} 

\numberwithin{equation}{section}
\DeclareMathOperator{\dif}{d\!}  
\DeclareMathOperator{\real}{Re}
\DeclareMathOperator{\imag}{Im}
\DeclareMathOperator{\Sop}{\mathcal{S}}
\DeclareMathOperator{\Dop}{\mathcal{D}}
\DeclareMathOperator{\Vop}{\mathrm{V}}

\begin{document}

\title{A Domain Derivative for Electromagnetic Scattering by Perfect Conductors in the Time Domain} 
\author{Marvin
  Kn\"oller\footnote{Department of Mathematics and Statistics, University of Helsinki, Pietari Kalmin katu 5, FI-00014, Finland
    {\tt marvin.knoller@helsinki.fi}}}

\maketitle
\begin{abstract}
A domain derivative for time-dependent electromagnetic scattering from perfect conductors is established. 
By proceeding through the Laplace domain, frequency-dependent bounds on solutions to Maxwell's equations are derived.
These bounds are used both for establishing time regularity properties of the domain derivative and for proving convergence of the proposed Runge--Kutta convolution quadrature semi-discretization in time. 
A full convergence analysis is also carried out for pointwise evaluations of the domain derivative, when this time discretization is combined with a Galerkin method in space.
Numerical experiments validate the convergence orders in time and space.
Eventually, the domain derivative is applied in an iterative shape reconstruction algorithm, in which the electric near field is measured at some receiver positions away from the perfect conductor.
Numerical examples show the feasibility of this algorithm and in particular highlight its robustness, when additional noise is applied to the data.
 \end{abstract}

{
\small\noindent
  Mathematics subject classifications (MSC2020): 
  35R30, % Inverse problems for PDEs
  78A46, % Inverse problems (including inverse scattering) in optics and electromagnetic theory
  65M32 % Numerical methods for inverse problems for initial value and initial-boundary value problems involving PDEs
% 65N21 % Numerical methods for IP for boundary value problems
  \\\noindent
Short title: A Domain Derivative for Perfect Conductors in the Time Domain
  % \\[1em]\noindent Last modified: \today
}
\section{Introduction}

Electromagnetic scattering plays a major role in optics and photonics, sensing and imaging and medical applications.
Within these disciplines the optimal design of scattering objects is essential.
Shape optimization in electromagnetism can be employed to tailor how light travels, to produce remarkable phenomena or to identify quantities that are inaccessible through direct measurements.
Many realistic electromagnetic sources are polychromatic, i.e., their spectrum spans across a whole set of frequencies.
Accurate modeling and simulation of the interaction between scattering objects and polychromatic light are essential for representing and predicting physically relevant scenarios. 
The corresponding shape derivatives that describe the effect of infinitesimal geometry deformations on the electromagnetic fields need to be tailored to this time-dependent setting.

This work is about the temporal domain derivative for time-dependent free space scattering from perfectly conducting obstacles.
Our emphasis lies on (i) the derivation of the domain derivative by proceeding through the Laplace domain, (ii) an error analysis for its numerical discretization using Runge--Kutta convolution quadrature in time and a Galerkin method in space and (iii) a practical use of it in an iterative shape reconstruction algorithm.

The perfect conductor model is an idealized, free space model problem for Maxwell's equations that arises, when
 linear material properties are assumed and the relation between the current density $\bfJ$ and the electric field $\Et$ is determined by Ohm's law via $\bfJ = \sigma \Et$, where $\sigma$ is the conductivity.
  As $\sigma \to \infty$ inside the scatterer $D$, the total wave in $D$ vanishes and Maxwell's equations for the total electric field $\Et = \Es + \Ei$, i.e., the superposition of a
given incident field $\boldsymbol{E}^i$ and the unknown scattered field
$\boldsymbol{E}^s$, can be written outside of the scatterer as
\begin{align*}
c^{-2} \partial_t^2 \Et + \curl^2 \Et \, &= \, 0 \, \quad \text{in } (\mathbb{R}^3\setminus \overline{D}) \times (0,\infty)\, , 
\end{align*}
where $c$ is the speed of light in vacuum.
  This model does not allow wave transmission inside the medium. It is an obstacle scattering problem completed by the perfectly conducting boundary condition implied by the continuity of the tangential trace of $\Et$ across surfaces reading
 \begin{align*}
 \bfnu \times \Et \, = \, 0 \quad \text{on } \partial D \times (0,\infty) \, .
 \end{align*}
The incident field $\Ei$ solves Maxwell's equations in all of $\R^3$ and is supported away from $D$ at $t = 0$, while the scattered field is
causal, i.e., $\Es(\cdot,0) = \partial_t \Es(\cdot,0) = 0$.

The study of electromagnetic scattering problems has been predominantly conducted in the time-harmonic case (see, e.g., \cite{ColKre19, KiHe15, Mon03, Ned01}).
In this setting, shape derivatives have been derived for various settings and have proven to be highly effective, achieving remarkable results in applications like inverse scattering, uncertainty quantification and optimal design.
Studies on the domain derivative emerged from \cite{Kirsch93} within the framework of sound-soft obstacle scattering and have been used ever since in many works on iterative shape reconstruction.
Here, we highlight the recent work \cite{SRHW25}, in which a highly sophisticated geometry representation together with the tangent-point energy as a regularizer has been used. 
In electromagnetic scattering the domain derivative for penetrable media was established in \cite{Het12}. 
In \cite{Hagetal19} the numerical implementation in the context of an inverse scattering problem followed. 
For perfect conductors, domain derivatives were derived in \cite{Hag19, HagHet20}, where also higher order shape derivatives were studied.
Recently, a systematic study of higher order shape derivatives was conducted in \cite{Baoetal26}.
In a slightly different approach, shape derivatives in the frequency domain were derived through the derivation of boundary integral equations.
For perfect conductors this was done in \cite{Pott96b}, for penetrable media in \cite{CosLou12}.
Moreover, in \cite{HipLi18} shape derivatives were derived through differential forms.

In the time-dependent setting, investigations on scattering from perfect conductors via proceeding through the Laplace domain are well-understood (see \cite{BalBanSauVeit13, BanSay22, LiMonWei15}).
Shape derivatives for time-dependent Maxwell's systems have been obtained in \cite{CagEl10, CagEl11, Zol09} in each case for interior problems on bounded domains rather than in a scattering setting.
Shape optimization based on shape derivatives has been carried out for perfectly conducting obstacles, e.g., in \cite{SchSchWa18} and \cite{Taka25} by using a time-domain discontinuous Galerkin method and a time-domain boundary element method, respectively.
Recently, the use of convolution quadrature methods as time discretization in inverse scattering problems has been considered in \cite{DonSuZha26, ZhaDonChe26, ZhaDonMa21, ZhaDonMa22}.
In these works, convolution quadrature is used to transform the time-dependent data fit formulation into a set of Laplace domain problems, on which shape derivatives are then computed.

In \cite{KN24} and \cite{GriKnoKum26} domain derivatives have been established for the wave equation in presence of a sound soft and penetrable object, respectively.
Moreover, inverse scattering problems have been solved numerically with the latter work focusing in particular on reconstructing three-dimensional scattering objects from backscattered far field data.

Finally we mention that sampling methods for time-dependent Maxwell's equations have been studied in \cite{GeSoWaWa25, LaMonSel22}.

In this work we focus on the derivation and the error analysis of the temporal domain derivative for perfect conductors and its use in an inverse scattering problem.
In particular, it is the aim to establish convergence results for the numerical discretization of the temporal domain derivative, when discretization is carried out by a Runge--Kutta convolution quadrature in time and a Galerkin method in space. The related study for the direct scattering problem has been conducted in \cite{BalBanSauVeit13}.

The work is structured as follows.
In Section \ref{sec:setting} we recall basic notations and results on scattering from perfect conductors, which are required throughout the whole work.
We study pointwise estimates, needed for the evaluation of single and double layer potentials pointwise in space and establish a higher order regularity result for scattered electric fields in the Laplace domain.
Section \ref{sec:ddL} is about the derivation of the domain derivative in the Laplace domain. 
Furthermore, we develop a boundary integral representation and prove frequency-dependent bounds for it.
In Section \ref{sec:ddT} we return to the time domain using the previously derived bounds in the Laplace domain.
Afterwards we establish error bounds for the discretization of the domain derivative, starting with a semi-discrete error in time and eventually, proceeding towards a full discretization error.
Section \ref{sec:InvScat} is about numerical experiments on the validation of the claimed convergence orders of the domain derivative. Moreover, shape reconstructions using time-dependent near field data and the previously derived fully discretized temporal domain derivative are carried out. 

\section{Scattering from a perfectly conducting obstacle}\label{sec:setting}
\subsection{Preliminaries}
Denote by $\varepsilon_0$ and $\mu_0$ the electric permittivity and the magnetic permeability of vacuum.
Moreover, let $D$ be a bounded domain of class $C^{2}$ with boundary $\Gamma$ representing the scattering object. The vector field $\bfnu$ denotes the unit normal that points into the exterior of $D$.
Time-dependent Maxwell's equations for the electric field $\bfE$ and the magnetic field $\bfH$ in absence of any electric and magnetic charges and currents but in presence of the perfectly conducting obstacle $D$ in $\Omega:= \R^3 \setminus \overline{D}$ read
\begin{subequations}\label{eq:MWtotal}
\begin{align}
\varepsilon_0 \partial_t \bfE - \curl \bfH \, &= \, 0\, \quad \text{in } \Omega \times (0,\infty)\, , \\ 
\mu_0 \partial_t \bfH + \curl \bfE \,  &= \, 0\, \quad \text{in } \Omega \times (0,\infty)\, , \\
\bfnu \times \bfE \, &= \, 0 \, \quad \text{on } \Gamma \times (0,\infty) \, .
\end{align}
\end{subequations}
Let the pair $(\Ei, \Hi)$ denote a pair of incident fields satisfying
\begin{align*}
\varepsilon_0 \partial_t \Ei - \curl \Hi \, &= \, 0\, \quad \text{in } \R^3 \times (0,\infty)\, , \\ 
\mu_0 \partial_t \Hi + \curl \Ei \,  &= \, 0\, \quad \text{in } \R^3 \times (0,\infty)\, .
\end{align*}
Moreover, at the initial time $t=0$, let $\mathrm{supp}(\Ei(\cdot, 0))$ and $\mathrm{supp}(\Hi(\cdot, 0))$ be disjoint from $D$. As time proceeds $(\Ei, \Hi)$ propagate towards $D$. On $\Gamma$ they induce a current, which generates the scattered fields $(\Es, \Hs)$. The total fields $(\bfE, \bfH) = (\Ei, \Hi) + (\Es, \Hs)$ solve \eqref{eq:MWtotal}.
Using that $(\bfE, \bfH)$ are a superposition of $(\Ei, \Hi)$ and $(\Es, \Hs)$, we can write \eqref{eq:MWtotal} in terms of the causal scattered fields alone, which is
\begin{align*}
\varepsilon_0 \partial_t \Es - \curl \Hs \, &= \, 0\, \quad \text{in } \Omega \times (0,\infty)\, , \\ 
\mu_0 \partial_t \Hs + \curl \Es \,  &= \, 0\, \quad \text{in } \Omega \times (0,\infty)\, , \\
\bfnu \times \Es \, &= \, -\bfnu \times \Ei \, \quad \text{on } \Gamma \times (0,\infty) \, ,
\end{align*}
or, one can even write it in terms of the scattered electric field alone, via
\begin{subequations}\label{eq:MWEs}
\begin{align}
c^{-2} \partial_t^2 \Es + \curl^2 \Es \, &= \, 0 \, \quad \text{in } \Omega \times (0,\infty)\, , \\
\bfnu \times \Es \, &= \, - \bfnu \times \Ei \quad \text{on } \Gamma \times (0,\infty)\, ,
\end{align}
\end{subequations}
where $c := (\varepsilon_0 \mu_0 )^{-1/2}$ denotes the speed of light.
Causality here means that $\Es(\cdot,0) = \partial_t\Es(\cdot,0) = 0$ in $\Omega$.
By using a rescaling in time, one can always assume that $c = 1$. Therefore, in our analysis we always consider $c=1$.

We now cite some well-posedness results for the system \eqref{eq:MWEs}, which are obtained by proceeding through the Laplace domain. More details on Laplace domain bounds, which are the basis for these results, are also given in the next sections below.
First, we introduce spaces, traces and operators related to Maxwell's equations.
Those can be found, e.g., in \cite{BufCoShe02, KiHe15, Mon03}.

Let $\sgrad$ denote the weak surface gradient and let $\divop$ and $\curl$ denote the divergence and rotation operator, respectively.
For spaces and traces related to Maxwell's equations we consider the space of tangential $L^2$ fields $L_t^2(\Gamma)$. For $\Ocal \in \{D, \Omega\}$ we introduce the standard Sobolev space $H^1(\Ocal)$, the trace space $H^{1/2}(\Gamma)$ and its dual $H^{-1/2}(\Gamma)$ as well as 
\begin{align*}
H(\curl,\mathcal{O}) \, := \, \left\{ \bfu \in  L^2(\Ocal)^3 \, : \, \curl \bfu \in L^2(\Ocal)^3 \right\} 
\end{align*}
with the norm $\Vert \cdot \Vert_{|s|,\Ocal}$ that is induced by the scalar product
\begin{align}\label{eq:skp}
\langle \bfu, \bfv\rangle_{|s|,\Ocal} \, := \, \langle \curl\bfu, \curl \bfv\rangle_{L^2(\Ocal)^3} + |s|^2\langle \bfu, \bfv\rangle_{L^2(\Ocal)^3}\, \quad \text{for } s \in \C_+\, ,
\end{align}
where $\C_+ :=  \{s \in \C \, : \, \real s > 0\}$.
For $|s|=1$ the scalar product defined in \eqref{eq:skp} becomes the usual scalar product on $H(\curl,\mathcal{O})$.
The induced norms $\Vert \cdot \Vert_{H(\curl,\mathcal{O})} := \Vert \cdot \Vert_{1, \Ocal}$ and $\Vert \cdot \Vert_{|s|, \Ocal}$ are equivalent, due to the inequalities
\begin{align*}
\min\{1,|s|\}\Vert \bfu \Vert_{H(\curl,\mathcal{O})} \, \leq \, \Vert \bfu \Vert_{|s|,\Ocal}\, \leq \,
\max\{1,|s|\}\Vert \bfu \Vert_{H(\curl,\mathcal{O})} \, .
\end{align*}
For smooth functions, the tangential trace $\gamma_t$ as well as the projection onto the tangential plane $\gamma_T$ is given by 
\begin{align}\label{eq:traceop}
\gamma_t \bfu \, := \, \bfnu \times \bfu|_\Gamma\, , \qquad \gamma_T\bfu \, := \, (\bfnu \times \bfu|_\Gamma ) \times \bfnu \, .
\end{align}
Let $V_t = \gamma_t(H^1(\Ocal)^3)$ and $V_T = \gamma_T(H^1(\Ocal)^3)$ and denote by $V_t^*$ and $V_T^*$ their dual spaces, respectively. 
The surface divergence operator $\sdiv$ and the surface rotational operator $\sCurl$ may be defined for functions in $V_t^*$ and $V_T^*$, respectively, through an appropriate extension (for details see \cite{BufCoShe02}).
The spaces $\HdivO$ and $\HcurlO$ are now defined via
\begin{align*}
\HdivO \, &:= \, \left\{ \bfu \in V_t^* \, : \, \sdiv \bfu \in H^{-1/2}(\Gamma) \right\}\, , \\
\HcurlO \, &:= \, \left\{ \bfu \in V_T^* \, : \, \sCurl \bfu \in H^{-1/2}(\Gamma) \right\} \, .
\end{align*}
For $s' \geq 0$  and (potentially) more regular $\Gamma$ we also define Sobolev spaces of degree $s' \geq 0$ in $\Omega$ and on $\Gamma$ by $H^{s'}(\Omega)$ and $H^{s'}(\Gamma)$, respectively. We also let $H_t^{s'}(\Gamma) = H^{s'}(\Gamma)^3 \cap L_t^2(\Gamma)$ and
\begin{align*}
H^{s'}(\sdiv, \Gamma) \, := \, \{ \bfu \in H_t^{s'}(\Gamma) \, : \, \sdiv \bfu \in H^{s'}(\Gamma)  \} \, .
\end{align*}
For $s'=0$ we write $H(\sdiv, \Gamma)$.
The trace operators $\gamma_t$ and $\gamma_T$ from \eqref{eq:traceop} can be extended to linear, continuous and surjective operators 
\begin{align*}
\gamma_t : H(\curl,\mathcal{O}) \to \HdivO\, , \qquad \gamma_T : H(\curl,\mathcal{O}) \to \HcurlO\, ,
\end{align*}
respectively.
Both operators $\gamma_t$ and $\gamma_T$ have a bounded right inverse that we denote by $\eta_t$ and $\eta_T$.
Moreover, by $H_0(\curl,\Ocal)$ we denote the null space of $\gamma_t$ and $\gamma_T$.
For smooth functions, the vectorial surface rotation $\scurl$ is defined by
\begin{align*}
\scurl u \, := \, \bfnu \times \sgrad u\, .
\end{align*}
The vectorial surface rotation can be extended to
$\scurl : H^{1/2}(\Gamma) \to \HdivO$. 
Moreover, it can be written as $\scurl u = \gamma_t \nabla(\eta u)$, where $\eta$ is a right inverse to the trace $\gamma:H^1(\Ocal) \to H^{1/2}(\Gamma)$.

For relating the Laplace domain estimates with the time domain we stick to the seminal work \cite{L94}. A recent introduction about this procedure is also found in the books \cite{BanSay22, Say16}.
Let $\mathcal{L}$ denote the Laplace transform given by
\begin{align*}
\mathcal{L}\{u\}(s) \, := \, \int_0^\infty \rme^{-st} u(t) \dt \quad \text{for any } s \in \C_+ \, .
\end{align*}
Let $X$ be a Hilbert space and let $H^r(\R,X)$ be the Sobolev space of order $r\in \R$ of $X$-valued functions on $\mathbb R$. Moreover, on finite intervals $(0, T )$ for some $0<T< \infty$,
	we write
\begin{align*}
H_0^r(0,T;X) \, := \, \{g|_{(0,T)} \,:\, g \in H^r(\R,X)\ \text{ with }\ g = 0 \ \text{ on }\ (-\infty,0)\} \, ,
\end{align*}
i.e., the
subscript 0 in $H_0^r$ only refers to the left end-point of the time interval.
The norm on $H_0^r(0,T;X)$ is equivalent to the norm $\Vert \partial_t^r \cdot \Vert_{L^2(0,T;X)}$.
The striking connection between $s$-dependent bounds in the Laplace domain and time-domain mapping properties was made in \cite[Lem.\@ 1]{L94} and reads as follows.
Let $K(s)$ be an analytic family of bounded linear operators  $K(s):X\to Y$, $\real s \geq \sigma>0$,
which are defined between Hilbert spaces $X$ and $Y$.
Let $K$ be polynomially bounded, i.e.\@ let there exist a real $\kappa \in \R$  and $\nu\ge 0$, and for every $\sigma >0$ let there exist $M_\sigma <\infty$, such that
\begin{equation}\label{eq:pol_bound}
	\| K(s) \|_{Y\leftarrow X} \, \le \,  M_\sigma \frac{|s|^\kappa}{(\real s)^\nu}, \qquad \real s \, \geq \, \sigma \, >\,  0 \, .
\end{equation}
Then, the convolution-type operator 
\begin{equation} \label{Heaviside}
		K(\partial_t)g \, := \, \mathcal{L}^{-1}\{K\} * g 
\end{equation}
extends by density to a bounded linear operator with the mapping properties 
\begin{align*}
K(\partial_t) : H_0^{r+\kappa}(0,T;X) \to H_0^{r}(0,T;Y)\qquad \text{for any } r \in \R\, .
\end{align*}
This result is used as follows for Maxwell's equations in presence of a perfect conductor \eqref{eq:MWEs}.
Applying the Laplace transform to \eqref{eq:MWEs} yields Maxwell's equations in the Laplace domain (see also \eqref{eq:LMWEs} below). For $\LEs$ and $\LEi$, the Laplace transformed scattered and incident electric fields, respectively, it can be shown (see, e.g.\@ \cite[Prop.\@ 6.5]{BanSay22}) that there is an operator denoted by $\Sop(s)\Vop^{-1}(s): \HdivO \to H(\curl, \Omega)$ satisfying
\begin{equation*}
-\Sop(s)\Vop^{-1}(s)(\bfnu \times \LEi) = \LEs
\end{equation*}
and 
\begin{align*}
\Vert \Sop(s)\Vop^{-1}(s)\Lbfu \Vert_{\HcurlOmega} \, \leq \, M_\sigma \frac{|s|^2}{\real s} \Vert \Lbfu\Vert_{\HdivO}\quad \text{for any } \Lbfu \in \HdivO\, .
\end{align*}
Using the composition rule
\begin{equation*}
K(\partial_t)L(\partial_t) \, = \, (KL)(\partial_t)\, ,
\end{equation*} 
which holds for compatible Laplace domain operators $K(s)$ and $L(s)$,
this implies that the unique solution to \eqref{eq:MWEs} is given by $\Es \, = \, -(\Sop\Vop^{-1})(\partial_t)\gamma_t \Ei$, where
\begin{align*}
(\Sop\Vop^{-1})(\partial_t) : H_0^{r+2}(0,T ; \HdivO) \to H_0^{r}(0,T ; \HcurlOmega)
\end{align*}
for any $r \in \R$.
\subsection{Maxwell's equations in the Laplace domain}
We study the partial differential equation \eqref{eq:MWEs} in the Laplace domain and thus,
apply the Laplace transform to it.
With the Laplace transformed scattered electric field\footnote{The upper index $s$ in $\LEs$ must not be confused with the variable $s$ in the Laplace domain.} $\LEs := \mathcal{L}\{\Es\}$ Maxwell's equations for a perfectly conducting obstacle $D$ in the Laplace domain read
\begin{subequations}\label{eq:LMWEs}
\begin{align}
\curl^2 \LEs + s^2 \LEs \, &= \, 0 \quad \text{in } \Omega\, ,\label{eq:LMWEs1}\\ 
\bfnu \times \LEs \, &= \, - \bfnu \times \LEi \quad \text{on } \Gamma \, .
\end{align}
\end{subequations}
The field $\LEi = \mathcal{L}\{\Ei \}$ is the Laplace transform of the incoming field $\Ei$, which is a solution to the Laplace domain Maxwell's equations in any bounded domain $D_0$ with $ D \subset \subset D_0$, when $\Ei(\cdot,0)$ is supported away from $D_0$.
\begin{rem}
In this work we will mainly work with electric fields. However, we note that the corresponding magnetic fields can be easily obtained in the Laplace domain since, due to Maxwell's equations, for $\LHs = \mathcal{L}\{\Hs \}$ it holds that $\LHs = -1/(s\mu_0) \curl \LEs$.
\end{rem}
We consider the partial differential equation \eqref{eq:LMWEs} in a slightly more general context.
For this purpose, we let $\Lbfg \in \HdivO$, $\mathcal{O} \in \{D, \Omega\}$ and study the problem to find $\Lbfv \in H(\curl,\mathcal{O})$ such that
\begin{subequations}\label{eq:MWV}
\begin{align}
\curl^2 \Lbfv + s^2 \Lbfv \, &= \, 0 \quad \text{in } \mathcal{O} \, ,\label{eq:MWV1}\\
\bfnu \times \Lbfv \, &= \, \Lbfg \quad \text{on } \Gamma \label{eq:MWV2}
\end{align}
\end{subequations}
for $s \in \C_+$. The scattering problem in the exterior \eqref{eq:LMWEs} is retrieved for $\mathcal{O}= \Omega$ and $\Lbfg = -\bfnu \times \LEi$.

We introduce the fundamental solution of the three-dimensional Helmholtz equation
\begin{align*}
\Phi_s(\bfx, \bfy) \, := \, \frac{\rme^{-s |\bfx-\bfy|}}{4\pi |\bfx-\bfy|} \, \quad \text{for } \real s >0 \, .
\end{align*}
The Maxwell single layer potential is given by $\Sop(s): \HdivO \to H(\curl, \R^3\setminus \Gamma)$ with
\begin{align}
\begin{split} \label{eq:Spot}
\left(\Sop(s) \widehat{\bfa}\right)(\bfx) \, &:= \, -s \int_\Gamma \Phi_s(\bfx, \bfy) \widehat{\bfa}(\bfy) \ds(\bfy) +
\frac{1}{s} \nabla \int_\Gamma \Phi_s(\bfx, \bfy) \sdiv \widehat{\bfa}(\bfy) \ds(\bfy) \\
&= \, \frac{1}{s}\curl^2 \int_{\Gamma} \Phi_s(\bfx, \bfy) \widehat{\bfa}(\bfy) \ds(\bfy)
 \quad \text{for } \bfx \in \R^3\setminus \Gamma\, ,
 \end{split}
\end{align}
where the equality holds by \cite[Thm.\@ 5.52]{KiHe15}.
Furthermore, the Maxwell double layer potential is defined by $\Dop(s) : \HdivO \to H(\curl, \R^3\setminus \Gamma)$ with
\begin{align}\label{eq:Dpot}
(\Dop(s) \widehat{\bfa})(\bfx) \, := \, \curl \int_{\Gamma} \Phi_s(\bfx, \bfy) \widehat{\bfa}(\bfy) \ds(\bfy) \quad \text{for } \bfx \in \R^3\setminus \Gamma\, .
\end{align}
We also introduce the Maxwell single layer operator $\Vop(s) : \HdivO \to \HdivO$ defined by $\Vop(s)= \gamma_t \Sop(s)$. In \cite[Thm.\@ 6.4]{BanSay22} a coercivity property for $\Vop(s)$ is shown. Moreover, the $s$-dependent estimate for $\Vop^{-1}(s)$
\begin{align}\label{eq:Vcoerc}
\left\Vert \Vop^{-1}(s) \right\Vert_{\HdivO \leftarrow \HdivO} \, \leq \, C \frac{|s|^2}{\real s} \max\{1,|s|^{-1} \}
\end{align}
is proven. 
The fact that $\Vop$ is boundedly invertible leads the way to unique solvability of the problem \eqref{eq:MWV} (see \cite[Prop.\@ 6.5]{BanSay22}).
For $\mathcal{O}\in \{D,\Omega\}$, the unique solution $\Lbfv \in H(\curl, \mathcal{O})$ to \eqref{eq:MWV} can be written as $\Lbfv = \Sop(s) \Vop^{-1}(s) \Lbfg$. Moreover, $\Lbfv$ satisfies the bound
\begin{align}\label{eq:Vest}
\Vert \Lbfv \Vert_{H(\curl,\mathcal{O})} \, \leq \, C \Vert \Lbfv \Vert_{|s|,\mathcal{O}} \, \leq \,  C \frac{|s|^2}{\real s} \max\{1, |s|^{-2}\} \Vert \Lbfg \Vert_{\HdivO} \, \quad \text{for any } s\in \C_+ \, .
\end{align}
When $\Lbfg$ is replaced by $\bfnu \times \LEi$, a useful representation for $\Vop^{-1}(s)\Lbfg$ can be derived. To do so, we first note that
solutions to the Laplace domain Maxwell's equations \eqref{eq:MWV1} can be represented by a combination of a single and a double layer potential. 
This is the well-known Stratton--Chu formula: 
Any function $\Lbfv \in H(\curl, \mathcal{O})$ that satisfies \eqref{eq:MWV1} fulfills the formula
\begin{align}\label{eq:Stratton-Chu}
a_{\mathcal{O}}\left(- (\Dop(s) \gamma_t \Lbfv )(\bfx) + \frac{1}{s} (\Sop(s)\gamma_t \curl \Lbfv )(\bfx)\right) \, = \, 
\begin{cases}
\Lbfv (\bfx)\, , \quad &\bfx \in \mathcal{O}\\
 0\, , \quad &\bfx \notin \overline{\mathcal{O}}\\
\end{cases} \, ,
\end{align}
where $a_{D} = 1$ and $a_{\Omega} = -1$.
Replacing the role of $\Lbfv$ in \eqref{eq:Stratton-Chu} by $\LEi$ for $\mathcal{O}=D$ and by $\LEs$ for $\mathcal{O}= \Omega$, respectively, and subtracting the formulas in \eqref{eq:Stratton-Chu} shows that the solution to \eqref{eq:LMWEs} may be written as
\begin{align*}
\LEs(\bfx) \, = \, -\frac{1}{s} \left(\Sop(s) \gamma_t \curl \LEt\right)(\bfx) \quad \text{for } \bfx \in \Omega \, .
\end{align*}
Applying the trace $\gamma_t$ to both sides of this formula and using the boundary condition in \eqref{eq:LMWEs} together with the invertibility of $\Vop(s)$ yields that
\begin{align*}
\frac{1}{s}\Vop(s) \gamma_t \curl \LEt \, = \, \gamma_t \LEi \quad \text{or equivalently }\quad \frac{1}{s}\gamma_t \curl \LEt =   \Vop^{-1}(s) \gamma_t \LEi \, .
\end{align*}

We derive pointwise, $s$-dependent bounds for the operators $\Sop(s)$ and $\Dop(s)$ in the following lemma. For the single layer potential $\Sop(s)$, this has been done in \cite[Thm.\@ 4.4]{BalBanSauVeit13} already.
Similar computations for the Helmholtz single and double layer potentials were made in \cite[Lem.\@ 7]{BLM11}
and \cite{GriKnoKum26}.
\begin{lem}\label{lem:pw}
Let $\bfx \in \Omega$ be fixed and let $d_\bfx=\mathrm{dist}(\bfx,\Gamma)$. Then, the operators $(\Sop \cdot)(\bfx), (\Dop \cdot)(\bfx) : \HdivO \to \C^3$ from \eqref{eq:Spot}, \eqref{eq:Dpot} are both linear and continuous. Moreover, it holds that
\begin{align}\label{eq:pwbounds}
\left| (\Sop \cdot)(\bfx) \right| + \left| (\Dop \cdot)(\bfx) \right| \, \leq \, C(\sigma, d_\bfx)|s|^2 \rme^{-d_\bfx \real s} \quad \text{for any } s \in \C\text{ with } \real s > \sigma > 0 \, .
\end{align}
The constant $C(\sigma, d_\bfx)$ does not depend on $s$, but on $\sigma^{-1}$ and on $d_\bfx^{-1}$.
\end{lem}
\begin{proof}
By straightforward computations we see that for any $\bfy \in \Gamma$, it holds that
\begin{align}\label{eq:Phiest}
|\partial_{x_i}\partial_{x_j} \Phi_s(\bfx,\bfy)|\, &\leq \, C(\sigma, d_\bfx) |s|^2 \rme^{-d_\bfx \real s} \, , \quad
|\partial_{x_i}\partial_{x_j}\partial_{x_k}  \Phi_s(\bfx,\bfy)|\, \leq \, C(\sigma, d_\bfx) |s|^3 \rme^{-d_\bfx \real s} 
\end{align}
for any $1\leq i,j,k \leq 3$. We understand the $\curl$ operator to act on matrices columnwise and find that $\curlx(\Phi_s(\bfx,\bfy) \widehat{\bfa}(\bfy)) = \curlx(\Phi_s(\bfx,\bfy)I_3) \widehat{\bfa}(\bfy)$. We apply the Cauchy-Schwarz inequality componentwise to the $H^{1/2} \times H^{-1/2}$ dual pairing, use the boundedness of the trace $\gamma:H^1(D) \to H^{1/2}(\Gamma)$ together with the estimates in \eqref{eq:Phiest} and find that for any $\widehat{\bfa} \in \HdivO$
\begin{align*}
\left| (\Sop \widehat{\bfa})(\bfx) \right| \, &= \, |s|^{-1}\left| \int_\Gamma \curl_\bfx^2 \left( \Phi_s(\bfx,\bfy) I_3\right) \widehat{\bfa}(\bfy) \ds(\bfy) \right|
\, \leq \, C(\sigma, d_\bfx) |s|^2 \rme^{-d_\bfx \real s} \Vert \widehat{\bfa} \Vert_{\HdivO} \, .
\end{align*}
Applying the same steps to $\left| (\Dop \widehat{\bfa})(\bfx) \right|$ yields that
\begin{align*}
\left| (\Dop \widehat{\bfa})(\bfx) \right| \, &= \, \left| \int_\Gamma \curl_\bfx \left( \Phi_s(\bfx,\bfy) I_3\right) \widehat{\bfa}(\bfy) \ds(\bfy) \right| \, \leq \, C(\sigma, d_\bfx) |s|^2 \rme^{-d_\bfx \real s} \Vert \widehat{\bfa} \Vert_{\HdivO}\, ,
\end{align*}
which finishes the proof.
\end{proof}
We end this section with a lemma on the regularity of the scatterer $D$ and how it influences the regularity of solutions to the Laplace domain Maxwell's equations.
\begin{lem}\label{lem:reg}
Let $D$ be of class $C^{2}$, $\Omega = \R^3\setminus \overline{D}$, $\mathcal{O} \in \{D,\Omega\}$ and let $\Lbfg \in H_t^{1/2}(\Gamma)$. Any weak solution $\Lbfv \in H(\curl, \mathcal{O})$ of \eqref{eq:MWV}
for $s \in \C_+$ satisfies $\Lbfv \in H^1(\mathcal{O})^3$ and 
\begin{align}\label{eq:RegBound}
\Vert \Lbfv \Vert_{H^1(\mathcal{O})^3} \, \leq \, C \left( \frac{|s|^2}{\real s} \Vert \Lbfg \Vert_{\HdivO} + \Vert \Lbfg \Vert_{H^{1/2}(\Gamma)^3} \right) \, .
\end{align}
If additionally $\Lbfg \in H^{1/2}(\sdiv,\Gamma)$, then $\curl \Lbfv \in H^1(\mathcal{O})^3$ and 
\begin{align}\label{eq:curlRegBound}
\Vert \curl \Lbfv \Vert_{H^1(\mathcal{O})^3} \, \leq \, C\left( \frac{|s|^3}{\real s} \Vert \Lbfg \Vert_{\HdivO} + \Vert \sdiv \Lbfg \Vert_{H^{1/2}(\Gamma)} \right) \, . 
\end{align}
\end{lem}
\begin{proof}
By \cite[Thm.\@ 2.12, Rmk.\@ 2.14]{Ametal98}, the space
\begin{align*}
X_N(D) \, := \, \left\{ \Lbfu \in H(\curl,D) \cap H(\divop, D) \, : \, \bfnu \times \Lbfu \in H^{1/2}(\Gamma)^3 \right\}
\end{align*}
with the norm
\begin{align*}
\Vert \Lbfu \Vert_{X_N(D)}^2 \, := \, \Vert \Lbfu \Vert_{L^2(D)^3}^2 + \Vert \curl \Lbfu \Vert_{L^2(D)^3}^2 + \Vert \divop\Lbfu \Vert_{L^2(D)}^2 + \Vert \bfnu \times \Lbfu \Vert_{H^{1/2}(\Gamma)^3}^2
\end{align*}
is continuously embedded in $H^1(D)^3$. The weak formulation of \eqref{eq:MWV} implies that the weak divergence of $\Lbfv$ vanishes.
An application of \eqref{eq:Vest} for $\mathcal{O}=D$ now yields that
\begin{align*}
\Vert \Lbfv \Vert_{H^1(D)^3} \, \leq \, C( \Vert \Lbfv \Vert_{H(\curl,D)} + \Vert \Lbfg \Vert_{H^{1/2}(\Gamma)^3}) \, \leq \, C \left( \frac{|s|^2}{\real s} \Vert \Lbfg \Vert_{\HdivO} + \Vert \Lbfg \Vert_{H^{1/2}(\Gamma)^3} \right)\, .
\end{align*}
To obtain the bound \eqref{eq:RegBound} for $\mathcal{O}=\Omega$ we choose $R>0$ in such a way that $D\subset B_{R}(0)$ and decompose $\Lbfv \in H(\curl, \Omega)$ into
$\Lbfv = \chi \Lbfv + (1-\chi)\Lbfv$, where $\chi \in C_c^\infty(B_{2R}(0))$ is a cut-off function satisfying $\chi = 1$ in $B_{R}(0)$.
The function $\chi \Lbfv$ is supported in $\Omega' = B_{3R}(0) \setminus \overline{D}$.
Straightforward calculations show that $\chi \Lbfv \in H(\curl, \Omega') \cap H(\divop, \Omega')$ and that $\bfnu \times (\chi \Lbfv) = \Lbfg$ on $\Gamma$ and $\bfnu \times (\chi \Lbfv) = 0$ on $\partial B_{3R}(0)$. 
We can now apply the embedding $X_N(\Omega')$ into $H^1(\Omega')^3$ and obtain
\begin{align}\label{eq:v1est}
\begin{split}
\Vert \chi \Lbfv \Vert_{H^1(\Omega')^3} \, &\leq \, C \left( \Vert \chi \Lbfv \Vert_{H(\curl, \Omega')}
+ \Vert \divop (\chi \Lbfv) \Vert_{L^2(\Omega')}
+ \Vert \Lbfg \Vert_{H^{1/2}(\Gamma)^3}
\right) \\
&\leq \, C \left( \Vert \Lbfv \Vert_{H(\curl, \Omega')}
+ \Vert \Lbfg \Vert_{H^{1/2}(\Gamma)^3}
+ \Vert \nabla \chi \Vert_{L^\infty(A)^3} \Vert \Lbfv \Vert_{L^2(A)^3}
\right) \\
&\leq \, C \left( \Vert \Lbfv \Vert_{H(\curl, \Omega)}
+ \Vert \Lbfg \Vert_{H^{1/2}(\Gamma)^3}
\right)\, ,
\end{split}
\end{align}
where in the intermediate step of \eqref{eq:v1est}, $A = B_{2R}(0)\setminus \overline{B_{R}(0)}$.
On the other hand, the function $(1-\chi) \Lbfv$ is supported in $\R^3\setminus \overline{B_{R}(0)}$. We can extend this function by $0$ into all $\R^3$. 
Since $\chi$ is smooth, this extended function is in $H(\curl, \R^3) \cap H(\divop, \R^3)$. It is known that this space coincides with $H^1(\R^3)^3$ and that the corresponding norms are equivalent.
Therefore, 
\begin{align}\label{eq:v2est}
\begin{split}
\Vert (1-\chi)\Lbfv \Vert_{H^1(\R^3)}
&\leq \, C \left( \Vert (1-\chi) \Lbfv \Vert_{H(\curl, \R^3)}
+ \Vert \divop ((1-\chi) \Lbfv) \Vert_{L^2(\R^3)}
\right) \\
&\leq \, C \left( \Vert \Lbfv\Vert_{H(\curl, \Omega)}
+ \Vert \Lbfg \Vert_{H^{1/2}(\Gamma)^3}
\right)\, .
\end{split}
\end{align}
Combining \eqref{eq:v1est} and \eqref{eq:v2est} and using \eqref{eq:Vest} now yields \eqref{eq:RegBound} for $\mathcal{O} = \Omega$.

Under the additional regularity assumption that $\Lbfg\in H^{1/2}(\sdiv,\Gamma)$
there also holds that $\curl \Lbfv\in H^1(\mathcal{O})^3$.  In fact, by using again \cite{Ametal98}, one finds that also the space 
\begin{align*}
X_T(D) \, := \, \left\{ \Lbfu \in H(\curl,D) \cap H(\divop, D) \, : \, \bfnu \cdot \Lbfu \in H^{1/2}(\Gamma) \right\}
\end{align*}
embeds continuously into $H^1(D)^3$. Since $\divop \curl \Lbfv=0$ and $\curl (\curl \Lbfv) = -s^2 \Lbfv$ one finds that $\curl \Lbfv\in H(\curl,D) \cap H(\divop, D)$.
Moreover,
\begin{align*}
\bfnu \cdot \curl \Lbfv \, = \, -\sdiv(\bfnu \times \Lbfv) \, = \, -\sdiv \Lbfg \, .
\end{align*} 
Consequently, by proceeding as before one finds that
\begin{align}\label{eq:RegBound2}
\Vert \curl \Lbfv \Vert_{H^1(\mathcal{O})^3} \, \leq \, C\left( \Vert \curl \Lbfv \Vert_{H(\curl, \mathcal{O})} + \Vert \bfnu \cdot \curl \Lbfv \Vert_{H^{1/2}(\Gamma)} \right) \, . 
\end{align}
Since
\begin{align*}
\Vert \curl \Lbfv \Vert_{H(\curl, \mathcal{O})} \, \leq \, C \frac{|s|^3}{\real s} \Vert \Lbfg \Vert_{\HdivO}\, 
\end{align*}
the inequality \eqref{eq:curlRegBound} follows.
\end{proof}
If the boundary and data is even smoother, higher order regularity results for $\Vop^{-1}(s)$ can be shown. We need these estimates in our error estimates in Section~\ref{sec:FE} below.
\begin{lem}
Let $k\in \N$ and let $D$ be of class $C^{k+1}$, $\Omega = \R^3\setminus \overline{D}$, $\mathcal{O} \in \{D,\Omega\}$. Then, for $\Lbfg_1 \in H^{k-1/2}(\sdiv, \Gamma)$ and $\Lbfg_2 \in H_t^{k+1/2}(\Gamma)$ there holds
\begin{subequations}
\begin{align}
\Vert \Vop^{-1}(s) \Lbfg_1 \Vert_{k-1/2,\sdiv} \, &\leq \, C \left( \frac{|s|^{k+2}}{\real s} \Vert \Lbfg_1 \Vert_{-1/2,\sdiv} + \sum_{\ell=1}^k |s|^{2\left\lfloor (k-\ell)/2\right\rfloor+1} \Vert \Lbfg_1 \Vert_{\ell-1/2,\sdiv} \right) \, , \label{eq:Vm1bound1} \\
\Vert \Vop^{-1}(s) \Lbfg_2 \Vert_{k-1/2} \, &\leq \, C \left( \frac{|s|^{k+1}}{\real s} \Vert \Lbfg_2 \Vert_{1/2} + \sum_{\ell=1}^k |s|^{2\left\lceil (k-\ell)/2\right\rceil-1} \Vert \Lbfg_2 \Vert_{\ell+1/2} \right) \, , \label{eq:Vm1bound2}
\end{align}
\end{subequations}
where we wrote $\Vert \cdot \Vert_{s'}$ and $\Vert \cdot \Vert_{s',\sdiv}$ to abbreviate the norms on $H_t^{s'}(\Gamma)$ and $H^{s'}(\sdiv, \Gamma)$. 
\end{lem}
\begin{proof}
Applying the same strategies as in the proof of Lemma~\ref{lem:reg} and using \cite[Cor.\@ 2.15]{Ametal98} we find that for $\Lbfg \in H^{k-1/2}(\sdiv,\Gamma)$ the solution $\Lbfv$ of \eqref{eq:MWV} satisfies $\Lbfv, \curl \Lbfv \in H^{k}(\mathcal{O})^3$ with
\begin{align*}
\Vert \Lbfv \Vert_{H^k(\mathcal{O})^3} \, &\leq \, C\left( \Vert \Lbfv \Vert_{L^2(\mathcal{O})^3} + \Vert \curl \Lbfv \Vert_{H^{k-1}(\mathcal{O})^3} + \Vert \Lbfg \Vert_{H^{k-1/2}(\sdiv,\Gamma)}\right) \, ,  \\
\Vert \curl \Lbfv \Vert_{H^k(\mathcal{O})^3} \, &\leq \, C\left( \Vert \curl \Lbfv \Vert_{L^2(\mathcal{O})^3} + |s|^2\Vert \Lbfv \Vert_{H^{k-1}(\mathcal{O})^3} + \Vert \Lbfg \Vert_{H^{k-1/2}(\sdiv,\Gamma)}\right) \, .
\end{align*}
Via induction one can now show that
\begin{subequations}
\begin{align}
\Vert \Lbfv \Vert_{H^k(\mathcal{O})^3} \, &\leq \, C\left( \frac{|s|^{k+1}}{\real s} \Vert \Lbfg \Vert_{\HdivO} + \sum_{\ell=1}^k |s|^{2\left\lfloor (k-\ell)/2\right\rfloor} \Vert \Lbfg \Vert_{H^{\ell-1/2}(\sdiv, \Gamma)} \right) \, , \label{eq:vandcurlv1} \\
\Vert \curl \Lbfv \Vert_{H^k(\mathcal{O})^3} \, &\leq \, C\left( \frac{|s|^{k+2}}{\real s} \Vert \Lbfg \Vert_{\HdivO} + \sum_{\ell=1}^k |s|^{2\left\lceil (k-\ell)/2\right\rceil} \Vert \Lbfg \Vert_{H^{\ell-1/2}(\sdiv,\Gamma)} \right) \, . \label{eq:vandcurlv2}
\end{align}
\end{subequations}
Using that
\begin{align}\label{eq:arep}
-\frac{1}{s} \left( \gamma_t (\curl (\Sop(s)\widehat{\bfa})|_+) - \gamma_t (\curl (\Sop(s)\widehat{\bfa})|_-) \right) \, = \, \widehat{\bfa}
\end{align}
for any $\widehat{\bfa} \in \HdivO$ (see, e.g., \cite[p.\@ 155]{BanSay22} or \cite[Thm.\@ 5.52]{KiHe15}) we can now write
\begin{align}\label{eq:arep2}
-\frac{1}{s} \left( \gamma_t \curl \Lbfv|_+ - \gamma_t \curl \Lbfv|_- \right) \, = \, \Vop^{-1}(s) \Lbfg
\end{align}
and using \eqref{eq:vandcurlv2} yields
\begin{align}\label{eq:Vb1}
\Vert \Vop^{-1}(s) \Lbfg \Vert_{k-1/2} \, &\leq \, C \left( \frac{|s|^{k+1}}{\real s} \Vert \Lbfg \Vert_{-1/2,\sdiv} + \sum_{\ell=1}^k |s|^{2\left\lceil (k-\ell)/2\right\rceil-1} \Vert \Lbfg \Vert_{\ell-1/2,\sdiv} \right) \, ,
\end{align}
which in particular shows \eqref{eq:Vm1bound2}. Moreover, by \eqref{eq:arep2},  and $\curl^2 \Lbfv = -s^2 \Lbfv$ we find that
\begin{align*}
\sdiv (\Vop^{-1}(s)\Lbfg) \, = \, -s \bfnu \cdot (\Lbfv|_+ - \Lbfv|_-)
\end{align*}
and applying the bound \eqref{eq:vandcurlv1} yields
\begin{align}\label{eq:Vb2}
\Vert \sdiv (\Vop^{-1}(s)\Lbfg) \Vert_{H^{k-1/2}(\Gamma)} \, &\leq \, C\left( \frac{|s|^{k+2}}{\real s} \Vert \Lbfg \Vert_{-1/2,\sdiv} + \sum_{\ell=1}^k |s|^{2\left\lfloor (k-\ell)/2\right\rfloor+1} \Vert \Lbfg \Vert_{\ell-1/2,\sdiv} \right) \, .
\end{align}
Combining \eqref{eq:Vb1} and \eqref{eq:Vb2} now shows \eqref{eq:Vm1bound1}.
\end{proof}

\begin{rem}
Frequency explicit regularity bounds for Maxwell's equations were also studied in \cite{BanSau12}.
\end{rem}

\section{The domain derivative in the Laplace domain}\label{sec:ddL}
In this section we proceed similar to \cite[Ch.\@ 3]{Hag19}, where the domain derivative for time-harmonic Maxwell's equations has been derived.
For our purpose, we derive estimates explicitly in terms of the norm of the incoming field $\LEi$ and in terms of the Laplace domain parameter $s$.
Recently, estimates that depend explicitly on the norm of the incoming field have also been performed in the time-harmonic setting in \cite{AreGriKnoSchu26} in order to establish Fr\'echet differentiability of the far field operator with respect to the Hilbert--Schmidt norm.

Let $\bfh \in C^1(\Gamma,\R^3)$ be a vector field on the boundary $\Gamma$ of $D$ that is supposed to deform the scatterer.
Let $R>0$ be so large that $D_0 = B_R(0)$ includes all scatterers under consideration.
Due to \cite[Thm.\@ 1.5]{Ho99}, $\bfh$ can be extended to a function (still denoted by) $\bfh \in C_0^1(D_0,\R^3)$ such that $\Vert \bfh \Vert_{C^1(D_0)} \leq C \Vert \bfh \Vert_{C^1(\Gamma)}$.
We use this extended $\bfh \in C_0^1(D_0, \R^3)$ to define a diffeomorphism 
\begin{align*}
\bfzeta: D_0 \to D_0 \, , \qquad \bfzeta(\bfx) \, := \, \bfx + \bfh(\bfx)
\end{align*}
and let $D_\bfh = \bfzeta(D)$ and $\Gamma_\bfh = \bfzeta(\Gamma)$.
We define $M\in \N$ points away from $D_\bfh$ for all $\Vert \bfh \Vert_{C^1(\Gamma)}<h_0$ and denote them by $\bfz_1, \dots, \bfz_M \in \R^3$.
For these fixed points and for a fixed Laplace domain incident wave $\LEi$, let $\widehat{X}$ denote the operator that maps the boundary $\Gamma$ to $\LEs$ evaluated at the points $\bfz_1, \dots, \bfz_M$, i.e.
$\widehat{X}(\Gamma) = (\LEs(\bfz_j))_{j=1,\dots,M} \in \C^{3M}$.
Moreover, we define the operator
\begin{align}\label{eq:FhatOp}
\widehat{F}_\Gamma : D(\widehat{F}_\Gamma) \subset C^1(\Gamma, \R^3) \to \C^{3M}\, , \qquad \widehat{F}_\Gamma(\bfh) \, = \, \widehat{X}(\Gamma_\bfh)\, ,
\end{align}
where $D(\widehat{F}_\Gamma)$ is a neighborhood of the zero function in $C^1(\Gamma, \R^3)$ that is so small that $\Gamma_\bfh = \bfzeta(\Gamma)$ is a well-defined boundary of a scattering object $D_\bfh$.
The aim of this section is to characterize the Fr\'echet derivative of $\widehat{F}_\Gamma$ at $0$, i.e. we want to determine the operator $\widehat{F}_\Gamma'(0) : C^1(\Gamma, \R^3) \to \C^{3M}$ that satisfies
\begin{align*}
\frac{1}{\Vert \bfh \Vert_{C^1(\Gamma)}}|\widehat{F}_\Gamma(\bfh) - \widehat{F}_\Gamma(0) - \widehat{F}_\Gamma'(0)\bfh | \, \to \, 0 \quad \text{as } \, \Vert \bfh \Vert_{C^1(\Gamma)} \, \to \, 0\, .
\end{align*}
We start by writing down the weak formulation of the problem \eqref{eq:LMWEs1} for both the unperturbed object $D$, as well as for the perturbed object $D_\bfh$. 
For this purpose, let $\chi \in C_0^\infty(\R^3)$ denote a smooth cutoff function satisfying $0\leq \chi \leq 1$, $\chi=1$ in $D_0$ and $\chi = 0$ in $\R^3\setminus \overline{D_1}$, where $D \subset \subset D_0 \subset\subset D_1$.
We assume throughout that $\mathrm{supp}(\Ei(\cdot,0))$ is disjoint from $D_1$ and that $\bfz_j \in \R^3\setminus \overline{D_1}$ for all $j=1,\dots,M$.
This particular choice of $\chi$ is useful, since both $\LEs+(\chi \LEi)|_{\Omega}$ and $\LEs_\bfh + (\chi \LEi)|_{\Omega_\bfh}$ have a vanishing tangential trace on $\partial D$ and $\partial D_\bfh$, respectively, while at the same time $\chi \LEi$ remains a solution to the Laplace domain Maxwell's equations in $D_0$.
The weak formulation of \eqref{eq:LMWEs1} for the scatterer $D$ may be formulated as the task to find $\Lbfz \in H_0(\curl,\Omega)$ such that $a(\Lbfz, \Lbfphi) = f(\Lbfphi)$ for all $\Lbfphi \in  H_0(\curl,\Omega)$, where
\begin{subequations}
\begin{align}
a(\Lbfz, \Lbfphi) \, &:= \, \int_\Omega \curl \Lbfz \cdot \overline{\curl \Lbfphi} + s^2 \Lbfz \cdot \overline{\Lbfphi} \dx \label{def:a}\, ,\\
f(\Lbfphi) \, &:= \, \int_\Omega \big(\curl (\LEi\chi)\big) \cdot \overline{\curl \Lbfphi} + s^2 \big(\LEi\chi\big) \cdot \overline{\Lbfphi }\dx \,
\end{align}
\end{subequations}
for all $\Lbfphi \in  H_0(\curl,\Omega)$ and to define $\LEs = \Lbfz - \LEi \chi$.
Note that, by straightforward computations, the antilinear functional $f$ can be written as
\begin{align*}
f(\Lbfphi) \, = \, -\int_{D_1 \setminus \overline{D_0}} \nabla \chi \cdot \left( \overline{\curl \Lbfphi} \times \LEi + \overline{\Lbfphi} \times \curl \LEi \right)\dx \, .
\end{align*}
The same weak formulation holds for the perturbed field $\LEs_\bfh$, when the domain $\Omega$ is replaced by $\Omega_\bfh$. 
In this situation, one states the weak formulation for $ \Lbfz_\bfh$ and defines $\LEs_\bfh = \Lbfz_\bfh - \LEi \chi$.
More importantly, a $\curl$-conforming variable transform back to the unperturbed domain $\Omega$ (see \cite[p.\@ 41-42]{Hag19}) shows, that the weak formulation is equivalent to the task to find $\Lbfz_\bfh^*\in  H_0(\curl,\Omega) $ such that $a_\bfh(\Lbfz_\bfh^*, \Lbfphi) = f(\Lbfphi)$ for all $\Lbfphi \in  H_0(\curl,\Omega)$, where
\begin{align*}
a_\bfh(\Lbfz_\bfh^*, \Lbfphi) \, := \, \int_\Omega \frac{1}{\det (J_\bfzeta)} \curl \Lbfz_\bfh^{* \top} J_\bfzeta^\top J_\bfzeta \overline{\curl \Lbfphi}
+ s^2 \Lbfz_\bfh^{* \top} J_\bfzeta^{-1} J_\bfzeta^{-\top} \overline{\Lbfphi} \det(J_\bfzeta) \dx \, 
\end{align*}
and to define $(\LEs_\bfh)^* =\Lbfz_\bfh^* - \LEi \chi$. The relation between $\Lbfz_\bfh$ and $\Lbfz_\bfh^*$ is given by $\Lbfz_\bfh^*(\bfx) = J_\bfzeta^\top(\bfx) \Lbfz_\bfh(\bfzeta(\bfx))$ (see \cite[Eq.\@ (3.1.2)]{Hag19}).
In the next lemma we study the continuity of $\Lbfz_\bfh^*$ and $(\LEs_\bfh)^*$ and derive important $s$-dependent bounds.
\begin{lem}
It holds that 
\begin{align}\label{eq:contres}
\Vert (\LEs_\bfh)^* - \LEs \Vert_{|s|, \Omega} \, = \, \Vert \Lbfz_\bfh^* - \Lbfz \Vert_{|s|, \Omega} \, \leq \, C \frac{|s|^2}{(\real s)^2} \Vert \bfh \Vert_{C^1(\Gamma)} \Vert \LEi \Vert_{H(\curl, D_1\setminus \overline{D_0})}\, .
\end{align}
\end{lem}
\begin{proof}
By the Riesz representation theorem, there is a well-defined boundedly invertible operator $T:H_0(\curl, \Omega) \to H_0(\curl, \Omega)$ such that
\begin{align}\label{eq:Riesz1}
\langle T\widehat{\bfu}, \Lbfphi \rangle_{|s|, \Omega} \, = \, a(\Lbfu, \Lbfphi) \quad \text{for all } \Lbfu, \Lbfphi \in H_0(\curl, \Omega) \, .
\end{align}
The scalar product on the left hand side of \eqref{eq:Riesz1} is defined in \eqref{eq:skp}.
Moreover, there exists a bounded and linear operator $T_\bfh: H_0(\curl, \Omega) \to H_0(\curl, \Omega)$ such that
\begin{align*}
\langle T_\bfh\Lbfu, \Lbfphi \rangle_{|s|, \Omega} \, = \, a_\bfh(\Lbfu, \Lbfphi) \quad \text{for all } \Lbfu, \Lbfphi \in H_0(\curl, \Omega) \, .
\end{align*}
Due to \cite[Lem.\@ 3.2]{Hag19}, it holds that
\begin{subequations}\label{eq:taylor}
\begin{align}
\Big\Vert \frac{J_\bfzeta^\top J_\bfzeta}{\det(J_\bfzeta)} - \big((1-\divop \bfh) I + J_\bfh + J_\bfh^\top\big) \Big\Vert_{L^\infty(\Omega)} \, &\leq \, C \Vert \bfh \Vert_{C^1(\Gamma)}^2\, , \\
\Big\Vert \det(J_\bfzeta) J_\bfzeta^{-1} J_\bfzeta^{-\top} - \big((1+\divop \bfh) I - J_\bfh - J_\bfh^\top\big) \Big\Vert_{L^\infty(\Omega)} \, &\leq \, C \Vert \bfh \Vert_{C^1(\Gamma)}^2\, .
\end{align}
\end{subequations}
The convergence $J_\bfzeta/(\det(J_\bfzeta))^{1/2} \to I$ and $J_\bfzeta^{-\top}(\det(J_\bfzeta))^{1/2} \to I$ as $\Vert \bfh \Vert_{C^1(\Gamma)} \to 0$ imply that there exist constants $0<c_i$, $i=1,2,3,4$ such that
\begin{align*}
c_1|\bfx| \, \leq \, \left| \frac{J_\bfzeta}{\sqrt{\det(J_\bfzeta)}} \bfx \right| \, \leq c_2|\bfx|\, , \quad
c_3|\bfx| \, \leq \, \left| J_\bfzeta^{-\top} \sqrt{\det(J_\bfzeta)} \bfx \right| \, \leq c_4|\bfx|
\end{align*}
for sufficiently small $\Vert \bfh \Vert_{C^1(\Gamma)}$. Therefore, for any $\Lbfu \in H_0(\curl,\Omega)$ we obtain that
\begin{align*}
\min\{c_1^2, c_3^2 \} \real s \Vert \Lbfu \Vert_{|s|, \Omega}^2 
\, &\leq \, \real s \left( \int_\Omega c_1^2 |\curl \Lbfu|^2 + |s|^2 c_3^2 |\Lbfu|^2 \dx \right) \\
\, &\leq \, \real s \left( \int_\Omega \frac{1}{\det(J_\bfzeta)} \left|J_\bfzeta \curl \Lbfu \right|^2 + s\overline{s} | J_\bfzeta^{-\top} \Lbfu |^2 \det(J_\bfzeta)\dx \right) \\
\, &\leq \, |a_\bfh (\Lbfu, s \Lbfu)| \, = \, |\langle T_\bfh \Lbfu, s \Lbfu \rangle_{|s|, \Omega} | \, \leq \, |s| \left\Vert T_\bfh \Lbfu \right\Vert_{|s|, \Omega} \Vert \Lbfu \Vert_{|s|, \Omega} \, .
\end{align*}
Therefore, $T_\bfh$ is boundedly invertible and $\Vert T_\bfh^{-1} \Vert_{|s|,\Omega \leftarrow |s|, \Omega} \leq C |s|/\real s$. Here, the operator norm indicates the use of the norm $\Vert \cdot \Vert_{|s|,\Omega}$ for both the domain and range space.
Moreover, we find that
\begin{align*}
&\left\Vert (T_\bfh - T) \Lbfu \right\Vert_{|s|, \Omega}^2 \\
\, &= \, |a_\bfh(\Lbfu, (T_\bfh - T) \Lbfu ) - a(\Lbfu, (T_\bfh - T) \Lbfu)| \\
\, &= \, \left| \int_\Omega \curl \Lbfu^\top \left( \frac{J_\bfzeta^\top J_\bfzeta}{\det(J_\bfzeta)} -I \right)\overline{ \curl  ((T_\bfh - T) \Lbfu)} + s^2 \Lbfu^\top \left(J_\bfzeta^{-1}J_\bfzeta^{-\top} \det(J_\bfzeta) - I \right)  \overline{(T_\bfh - T) \Lbfu} \dx \right| \\
\, &\leq \, C\Vert \bfh \Vert_{C^1(\Gamma)} \Vert  (T_\bfh - T) \Lbfu \Vert_{|s|,\Omega} \Vert \Lbfu \Vert_{|s|,\Omega} \, .
\end{align*}
This implies that $\Vert T_\bfh - T \Vert_{|s|,\Omega \leftarrow |s|, \Omega} \leq C \Vert \bfh \Vert_{C^1(\Gamma)}$.
Again, by the Riesz representation theorem, the weak formulations for $\Lbfz$ and $\Lbfz_\bfh^*$ are equivalent to the operator equations $T\Lbfz = \bff$ and $T_\bfh \Lbfz_\bfh^* = \bff$ for some $\bff \in H_0(\curl,\Omega)$. 
Since $T_\bfh( \Lbfz_\bfh^* - \Lbfz) = (T-T_\bfh)\Lbfz$ we get that
\begin{align}\label{eq:finlem1}
\Vert  (\LEs_\bfh)^* -  \LEs \Vert_{|s|,\Omega} 
\, \leq \, C \frac{|s|}{\real s} \Vert (T-T_\bfh) \Lbfz\Vert_{|s|,\Omega} \, \leq \, C \frac{|s|}{\real s} \Vert \bfh \Vert_{C^1(\Gamma)} \Vert \Lbfz \Vert_{|s|, \Omega} \, .
\end{align}
Finally, we insert the function $s \Lbfz$ in the weak formulation and find that
\begin{align}\label{eq:helpZ}
\real s \Vert \Lbfz \Vert_{|s|,\Omega}^2 \, = \, \real a(\Lbfz,s\Lbfz)
\, \leq \,  |f(s\Lbfz)| \, \leq \, C|s| \Vert \LEi \Vert_{H(\curl, D_1 \setminus \overline{D_0})}\Vert \Lbfz \Vert_{|s|,\Omega}\, ,
\end{align}
which implies that $\Vert \Lbfz \Vert_{|s|,\Omega} \leq C |s|/\real s \Vert \LEi \Vert_{H(\curl, D_1 \setminus \overline{D_0})} $.
Using this bound in \eqref{eq:finlem1} yields \eqref{eq:contres}.
\end{proof}

We define the material derivative $\Lbfw \in H_0(\curl,\Omega)$ as the unique solution to 
\begin{align}\label{eq:matder}
a(\Lbfw, \Lbfphi) \, = \, \int_\Omega \curl\Lbfz^\top \left( \divop \bfh I - J_\bfh - J_\bfh^\top \right) \overline{\curl \Lbfphi} - s^2 \Lbfz^\top \left( \divop \bfh I - J_\bfh - J_\bfh^\top \right) \overline{\Lbfphi} \dx
\end{align}
for all $\Lbfphi \in H_0(\curl,\Omega)$. By using \eqref{eq:taylor} we see that the right hand side of \eqref{eq:matder} carries the missing part in the brackets that are required for the second order decay of $|a(\Lbfz_\bfh^* - \Lbfz - \Lbfw, \Lbfphi)|$ with respect to $\Vert \bfh \Vert_{C^1(\Gamma)}$. This is proven in the next lemma.
\begin{lem}
Let $\Lbfz_\bfh^* \in  H_0(\curl,\Omega)$ and $\Lbfz \in H_0(\curl,\Omega)$ be the solutions of $a_\bfh(\Lbfz_\bfh^*, \Lbfphi) = f(\Lbfphi)$ and $a(\Lbfz, \Lbfphi) = f(\Lbfphi)$ for all $\Lbfphi \in H_0(\curl,\Omega)$, respectively and let the material derivative $\Lbfw \in H_0(\curl,\Omega)$ be defined by \eqref{eq:matder}. 
Then, it holds that
\begin{align}\label{eq:wlin}
\Vert \Lbfz_\bfh^* - \Lbfz- \Lbfw \Vert_{|s|,\Omega} \, \leq \, C \frac{|s|^3}{(\real s)^3} \Vert \bfh \Vert_{C^1(\Gamma)}^2 \Vert \LEi \Vert_{H(\curl, D_1\setminus \overline{D})}\, .
\end{align}
\end{lem}
\begin{proof}
By using some simple rearrangements, it can be seen that
\begin{align*}
a(\Lbfz_\bfh^*- \Lbfz- \Lbfw, \Lbfphi) \, &= \,
a(\Lbfz_\bfh^* , \Lbfphi) -
a_\bfh(\Lbfz_\bfh^*, \Lbfphi)-
a(\Lbfw, \Lbfphi) \\
\, &= \, \int_\Omega \bigg(\Big( \curl \Lbfz_\bfh^{* \top} - \curl \Lbfz^{\top}  \Big)\Big(I - \frac{J_\bfzeta^\top J_\bfzeta}{\det(J_\bfzeta)} \Big) \overline{\curl \Lbfphi} \\
& \phantom{= \, \int_\Omega}
+s^2( \Lbfz_\bfh^{* \top} - \Lbfz^{\top}  ) \Big( I - J_\bfzeta^{-1} J_\bfzeta^{-\top} \det(J_\bfzeta) \Big) \overline{\Lbfphi} \\
& \phantom{= \, \int_\Omega}
-\curl \Lbfz^\top \Big( -I + \frac{J_\bfzeta^\top J_\bfzeta}{\det(J_\bfzeta)} + \divop \bfh I - J_\bfh - J_\bfh^\top \Big) \overline{\curl \Lbfphi} \\
& \phantom{= \, \int_\Omega}
+ s^2 \Lbfz^\top ( I - J_\bfzeta^{-1} J_\bfzeta^{-\top} \det(J_\bfzeta) + \divop \bfh I - J_\bfh - J_\bfh^\top) \overline{\Lbfphi} \bigg) \dx\, .
\end{align*}
Next, we apply the triangle inequality, use the linearization formulas in \eqref{eq:taylor} and make use of the auxiliary estimate in \eqref{eq:helpZ} and find that
\begin{align*}
&|a(\Lbfz_\bfh^*- \Lbfz- \Lbfw, \Lbfphi)| 
\, \leq \, \Vert \curl (\Lbfz_\bfh^{*} - \Lbfz) \Vert_{L^2(\Omega)^3} \Big\Vert I - \frac{J_\bfzeta^\top J_\bfzeta}{\det(J_\bfzeta)} \Big\Vert_{L^\infty(\Omega)} \Vert \curl\Lbfphi \Vert_{L^2(\Omega)^3} \\
&\phantom{\qquad \qquad \qquad \leq \,} + \Vert s(\Lbfz_\bfh^{*} - \Lbfz )\Vert_{L^2(\Omega)^3} \Vert I - J_\bfzeta^{-1} J_\bfzeta^{-\top} \det(J_\bfzeta) \Vert_{L^\infty(\Omega)}\Vert s\Lbfphi \Vert_{L^2(\Omega)^3} \\
&\phantom{\qquad \qquad \qquad\leq \,} + \Vert \curl \Lbfz \Vert_{L^2(\Omega)^3} \Big\Vert -I + \frac{J_\bfzeta^\top J_\bfzeta}{\det(J_\bfzeta)} + \divop \bfh I - J_\bfh - J_\bfh^\top \Big\Vert_{L^\infty(\Omega)}\Vert \curl \Lbfphi \Vert_{L^2(\Omega)^3} \\
&\phantom{\qquad \qquad \qquad \leq \,} + \Vert s \Lbfz\Vert_{L^2(\Omega)^3} \Vert I - J_\bfzeta^{-1} J_\bfzeta^{-\top} \det(J_\bfzeta) + \divop \bfh I - J_\bfh - J_\bfh^\top \Vert_{L^\infty(\Omega)} \Vert s \Lbfphi \Vert_{L^2(\Omega)^3} \\
&\phantom{|a(\Lbfz_\bfh^*- \Lbfz- \Lbfw, \Lbfphi)| } \leq \, C \frac{|s|^2}{(\real s)^2} \Vert \bfh \Vert_{C^1(\Gamma)}^2 \Vert \Lbfphi \Vert_{|s|,\Omega} \Vert \LEi \Vert_{H(\curl, D_1\setminus \overline{D})} \, .
\end{align*}
Finally, the inequality $\real s\Vert \Lbfphi \Vert_{|s|,\Omega}^2 \leq |s||a(\Lbfphi, \Lbfphi)|$ implies \eqref{eq:wlin}.
\end{proof}

The next proposition is about the domain derivative corresponding to the Laplace domain Maxwell's equations.
\begin{proposition}
Let $D$ be a bounded $C^2$ domain, $\partial D = \Gamma$ and let $\LEs \in H^1(\Omega)^3$ be the unique solution to \eqref{eq:LMWEs}. Then, the domain derivative in the Laplace domain is given by the solution $\LEt' \in H(\curl,\Omega)$ to
\begin{subequations}\label{eq:ddmw}
\begin{align}
\curl^2 \LEt' + s^2 \LEt' \, = \, 0 \quad &\text{in } \Omega\, ,\\
\bfnu \times \LEt' \, = \, - \scurl \big( \bfh_\bfnu \LEt_\bfnu \big) - \bfh_\bfnu \gamma_T \curl \LEt \quad &\text{on } \Gamma \, ,
\end{align}
\end{subequations}
where $\LEt = \LEs + \LEi \in H^1(D_0\setminus \overline{D})^3$.
It holds that $\widehat{F}_\Gamma'(0)\bfh = (\LEt'(\bfz_j))_{j=1,\dots,M}$; more precisely, there holds the bound
\begin{align}\label{eq:ddprop}
\left| \widehat{F}_\Gamma(\bfh) - \widehat{F}_\Gamma(0) - (\LEt'(\bfz_j))_{j=1,\dots,M} \right| \, \leq \, C |s|^5\rme^{-\mathrm{d}_\bfz^* \real s} \Vert \bfh \Vert_{C^1(\Gamma)}^2 \Vert {\LEi} \Vert_{H(\curl, D_1\setminus \overline{D})}
\end{align}
for $\real s > \sigma$, where we abbreviated  $\mathrm{d}_\bfz^* = \min_{j=1,\dots,M}\text{dist}(\bfz_j, \partial D_1)$ and where $C = C(\sigma, \mathrm{d}_\bfz^*)$.
\end{proposition}
\begin{proof}
The proof consists of showing that $\LEt' = \Lbfw - J_\bfh^\top \Lbfz - J_{\Lbfz}\bfh$, where $\Lbfw \in H_0(\curl,\Omega)$ is the material derivative defined in \eqref{eq:matder}, $\Lbfz$ is given by $\LEs + \LEi\chi$ and $J_\bfa$ denotes the Jacobian of the vector valued quantity $\bfa \in H^1(\Omega)^3$. To prove this identity, we show that
\begin{subequations}
\begin{align}
a(\Lbfw - J_\bfh^\top \Lbfz - J_{\Lbfz}\bfh, \Lbfphi ) \,&=\, 0\quad  \text{for all } \Lbfphi \in H_0(\curl,\Omega)\, , \label{eq:ida} \\ 
\bfnu \times (\Lbfw - J_\bfh^\top \Lbfz - J_{\Lbfz}\bfh )\, &= \, \bfnu \times \LEt' \quad \text{on } \Gamma \, . \label{eq:boundcd}
\end{align}
\end{subequations}
To show the property \eqref{eq:ida} we first show some useful identities.
By using Green's theorem (see e.g.\@ \cite[Thm.\@ 3.24]{Mon03}) one finds that
\begin{align}\label{eq:normaldrop}
\bfnu \cdot \curl \Lbfphi \, = \, 0 \quad \text{in } H^{-1/2}(\Gamma) \quad \text{for any } \Lbfphi \in H_0(\curl,\Omega) \, .
\end{align}
Moreover, for $\bfa \in H^1(\Omega)^3$, it holds that $\curl \bfa \times \bfh = (J_{\bfa} - J_{\bfa}^\top) \bfh$ in $L^2(\Omega)^3$, which follows from straightforward computations.
From the proof of \cite[Thm.\@ 3.6, Eq.\@ (3.1.8)]{Hag19} we see that
\begin{align}\label{eq:help1}
\curl \big( J_\bfh^\top \Lbfz + J_{\Lbfz}\bfh \big) \, = \, \curl \Lbfz \divop \bfh + J_{\curl \Lbfz} \bfh - J_\bfh \curl \Lbfz\, .
\end{align}
Note that due to Lemma~\ref{lem:reg}, $\curl \Lbfz \in H^1(\Omega)$ and $J_{\curl \Lbfz}$ is well-defined. Furthermore, straightforward computations show that
\begin{align*}
\curl^2 \Lbfz \, = \, -s^2 \Lbfz - \LEi\Delta \chi + J_{\nabla \chi} \LEi - J_{\LEi}\nabla \chi + \nabla \chi \times \curl \LEi
\end{align*}
and due to the disjoint support of $\bfh$ and $\nabla \chi$, there holds that $\curl^2 \Lbfz \times \bfh = -s^2 \Lbfz \times \bfh$.
Now, as in the proof of \cite[Thm.\@ 3.6, p.\@ 48]{Hag19} one sees that
\begin{align}\label{eq:help2}
J_\bfh^\top \curl \Lbfz + J_{\curl \Lbfz}\bfh \, = \, \curl^2 \Lbfz \times \bfh + \nabla(\curl \Lbfz \cdot \bfh) \, = \, -s^2 \Lbfz \times \bfh + \nabla(\curl \Lbfz \cdot \bfh) \, .
\end{align}
We now use the definition of $\Lbfw$ in \eqref{eq:matder}, the definition of $a$ from \eqref{def:a} together with \eqref{eq:help1} and \eqref{eq:help2} and find that for any $\Lbfphi \in C_0^\infty(\Omega,\C^3)$
\begin{align*}
&a\big(\Lbfw - J_\bfh^\top\Lbfz - J_{\Lbfz}\bfh, \Lbfphi \big) \\
\, &= \,  \int_\Omega \curl\Lbfz^\top \left( \divop \bfh I - J_\bfh - J_\bfh^\top \right) \overline{\curl \Lbfphi} - s^2 \Lbfz^\top \left( \divop \bfh I - J_\bfh - J_\bfh^\top \right) \overline{\Lbfphi} \dx \\
\, &\phantom{= \, } - \left( \int_\Omega \curl \big( J_\bfh^\top\Lbfz + J_{\Lbfz}\bfh \big) \cdot \overline{\curl \Lbfphi} + s^2 \big( J_\bfh^\top\Lbfz + J_{\Lbfz}\bfh \big) \cdot \overline{\Lbfphi} \dx \right) \\
\, &= \, \int_\Omega - \big( J_\bfh^\top \curl \Lbfz + J_{\curl \Lbfz}\bfh \big) \cdot \overline{\curl \Lbfphi} - s^2 \big( \Lbfz \divop \bfh + J_{\Lbfz}\bfh - J_\bfh\Lbfz \big) \cdot \overline{\Lbfphi} \dx \\
\, &= \, \int_\Omega - ( -s^2 \Lbfz \times \bfh + \nabla(\curl \Lbfz \cdot \bfh) ) \cdot \overline{\curl \Lbfphi} - s^2 \curl(\Lbfz \times \bfh) \cdot \overline{\Lbfphi} \dx \\
\, &= \,  \int_\Omega \divop \left( s^2 \overline{\Lbfphi} \times \big( \Lbfz \times \bfh \big) - \big(\curl \Lbfz \cdot \bfh \big) \overline{\curl \Lbfphi} \right) \dx \\
\, &= \, \int_{\Gamma} s^2 \left( \overline{\Lbfphi} \times \big( \Lbfz \times \bfh \big) \right) \cdot \bfnu - \big( \curl \Lbfz \cdot \bfh \big) \big( \overline{ \curl\Lbfphi} \cdot \bfnu \big) \ds \, = \, 0\, .
\end{align*}
The last equality follows since $\curl \overline{\Lbfphi} \cdot \bfnu = 0$ (see \eqref{eq:normaldrop}) and 
\begin{align*}
\left( \overline{\Lbfphi} \times \big( \Lbfz \times \bfh \big) \right) \cdot \bfnu \, = \, 
 \left( \bfnu \times \overline{\Lbfphi} \right) \cdot \big( \Lbfz \times \bfh \big) \, = \, 0 \quad \text{for any } \Lbfphi \in C_0^\infty(\Omega,\C^3)\, .
\end{align*}
Furthermore, note that $\bfh \in C_0^1(D_0,\R^3)$ and so $\bfh|_\Omega \in W^{1,\infty}(\Omega)^3$ and the regularities of both $\Lbfz$ and $\curl \Lbfz$ imply that $\Lbfz \times \bfh \in H^1(\Omega)^3$, $\curl\Lbfz \cdot \bfh \in H^1(\Omega)$. 
Thus, by the density of $C_0^\infty(\Omega,\C^3)$ in $H_0(\curl, \Omega)$ (see, e.g.\@, \cite[Def.\@ 4.19]{KiHe15}) and the continuity of $a$, \eqref{eq:ida} follows.
To show the boundary condition in \eqref{eq:boundcd}, we use that $\Lbfw \in H_0(\curl, \Omega)$ and find as in the proof of \cite[Thm.\@ 3.6, p.\@ 46]{Hag19} that
\begin{align*}
\bfnu \times \big(\Lbfw - J_\bfh^\top \Lbfz - J_{\Lbfz}\bfh \big) \, = \, - \bfnu \times \big(\curl \Lbfz \times \bfh \big) - \bfnu \times  \nabla \big( \bfh \cdot \Lbfz \big)  \, .
\end{align*}
Moreover, by using the density of $C^\infty(\overline{\Omega})$ in $H^1(\Omega)$ and the fact that $\curl \Lbfz$ coincides with its tangential component (see \eqref{eq:normaldrop}) one finds that
\begin{align*}
- \bfnu \times \big(\curl \Lbfz \times \bfh \big) \, = \, -\bfh_\bfnu \gamma_T \curl \Lbfz \quad \text{and } - \bfnu \times \nabla \big( \bfh \cdot \Lbfz \big)  \, = \, -\scurl \big(\bfh_\bfnu \Lbfz_\bfnu\big)\, .
\end{align*}
The fact that $\Lbfz$ coincides with $\LEs + \LEi$ in a neighborhood of $\Gamma$ now yields the condition in \eqref{eq:boundcd}.

The weak formulations for $\Lbfz_\bfh^*$ and ${\Lbfz}$ and the definition of $\Lbfw$ from \eqref{eq:matder} imply that $\curl \Lbfv \in H(\curl,\Omega\setminus\overline{D_1})$ and $\curl^2 \Lbfv + s^2 \Lbfv = 0$ in $\Omega\setminus \overline{D_1}$ for $\Lbfv \in \{\Lbfz_\bfh^*, \Lbfz, \Lbfw \}$.
Now, we apply the definition of $\widehat{F}_\Gamma$ from \eqref{eq:FhatOp}, the fact that for $\bfz \in \Omega\setminus \overline{D_1}$ there holds that $\Lbfz_\bfh^*(\bfz) = \widehat{\Es_\bfh}^*(\bfz)$, $\Lbfz(\bfz) = \LEs(\bfz)$ and $\LEt'(\bfz) = \Lbfw(\bfz)$ as well as the Stratton--Chu formula \eqref{eq:Stratton-Chu} (integrating over $\partial D_1$) together with the boundedness of the trace $\gamma_t$, the pointwise bounds in Lemma~\ref{lem:pw} and the bound \eqref{eq:wlin}
\begin{align*}
&\big|\LEs_\bfh(\bfz) - \LEs(\bfz) - \LEt'(\bfz)\big| 
= \, \big|\Lbfz_\bfh^*(\bfz) - \Lbfz(\bfz) - \Lbfw(\bfz)\big| \\
&= \, \big| \big(\Dop(s) \gamma_t\big(\Lbfz_\bfh^* - \Lbfz - \Lbfw \big) \big) (\bfz) - \frac{1}{s}\big(\Sop(s) \gamma_t\big(\curl \big(\Lbfz_\bfh^* - \Lbfz - \Lbfw \big) \big) \big) (\bfz)  \big| \\
& \leq \, C |s|^2 \rme^{-\mathrm{d}_\bfz \real s} \Big( \big\Vert \Lbfz_\bfh^* - \Lbfz - \Lbfw  \big\Vert_{H(\curl,\Omega\setminus\overline{D_1})} + \frac{1}{|s|}\big\Vert \curl\big(\Lbfz_\bfh^* - \Lbfz - \Lbfw \big) \big\Vert_{H(\curl,\Omega\setminus\overline{D_1})} \Big) \\
&\leq \, C |s|^5\rme^{-\mathrm{d}_\bfz \real s} \Vert \bfh \Vert_{C^1(\Gamma)}^2 \Vert {\LEi} \Vert_{H(\curl, D_1\setminus \overline{D})}\, ,
\end{align*}
where we abbreviated  $\mathrm{d}_\bfz = \text{dist}(\bfz, \partial D_1)$ and $C = C(\sigma, \mathrm{d}_\bfz)$.
\end{proof}

\subsection{Boundary integral formulation for the domain derivative}\label{subsec:bi}
In order to derive a suitable boundary integral equation for the domain derivative we first introduce the operator $\Lambda_\mathcal{O} : H_t^{3/2}(\Gamma) \to H^1(\mathcal{O} )^3$, that maps $\Lbfg \mapsto \Lbfv$, where $\Lbfv \in H^1(\mathcal{O})^3$ is the unique solution to \eqref{eq:MWV} (see also Lemma~\ref{lem:reg}).
By $\gamma_n^\pm$ we denote the normal trace, which maps from $H^1(\Omega)^3$ and $H^1(D)^3$ to $H^{1/2}(\Gamma)$ respectively. Analogously, we understand $\gamma_T^\pm$, the projections onto the tangential plane.
We define the operator $L(s) : H_t^{3/2}(\Gamma) \to \HdivO$ via
\begin{align}\label{eq:defL}
  L(s)\widehat{\bfa} \, := \, 
-\big(\scurl\big( \bfh_\bfnu ( \gamma_n^- \Lambda_D - \gamma_n^+ \Lambda_\Omega) \widehat{\bfa} \big) + \bfh_\bfnu (\gamma_T^- \curl \Lambda_D - \gamma_T^+ \curl \Lambda_\Omega)\widehat{\bfa} \big) \, ,
\end{align}
as well as the operator $\Fcal(s) : H_t^{3/2}(\Gamma) \to H(\curl,\Omega)$ given by
\begin{align}\label{eq:Fcal}
\Fcal(s) := \Sop(s) \Vop^{-1}(s) L(s) \, .
\end{align}
By the uniqueness of the interior Maxwell problem together with our standing assumption that $\LEi$ is a solution to the Laplace domain Maxwell's equation in $D$, it holds that $\Lambda_D \gamma_t \LEi = \LEi$ in $D$.
Therefore, the operator $\Fcal(s)$ maps $\gamma_t \LEi$ directly to the domain derivative $\LEp$ from \eqref{eq:ddmw}, i.e., $\Fcal(s)\gamma_t\LEi = \LEp$ in $\Omega$.
In the next lemma we derive $s$-dependent bounds for the operator $\Fcal(s)$.
\begin{lem}\label{lem:Fbounds}
Let $\real s \geq \sigma  >0$, let $\Gamma$ be of class $C^2$ and let $\bfh \in C^1(\Gamma, \R^3)$.
For the operator $\Fcal(s)$ from \eqref{eq:Fcal} there holds that
\begin{align}\label{eq:Fcalbound}
\Vert \Fcal(s) \Vert_{H(\curl, \Omega) \leftarrow H_t^{3/2}(\Gamma)} \, \leq \, C_\sigma \frac{|s|^5}{(\real s)^2} \, .
\end{align}
Moreover, for $\bfx \in \Omega$ and $d_\bfx=\mathrm{dist}(\bfx,\Gamma)$ it holds that
\begin{align}\label{eq:Fcalpwbound}
\Vert \Fcal(s)\cdot (\bfx) \Vert_{\C^3 \leftarrow H_t^{3/2}(\Gamma)} \, \leq \, C(\sigma,\mathrm{d}_\bfx) \rme^{- \mathrm{d}_\bfx \real s}|s|^7\, , 
\end{align}
where $C(\sigma,\mathrm{d}_\bfx)$ does not depend on $s$ but on $\sigma$ and on $\mathrm{d}_\bfx^{-1}$.
\end{lem}
\begin{proof}
The function $\Lbfv = \Fcal(s)\widehat{\bfa} \in H(\curl,\Omega)$ is the unique solution to \eqref{eq:MWV} with $\widehat{\bfg}$ replaced by $L(s)\widehat{\bfa}$ in \eqref{eq:MWV2}.
Therefore, by \eqref{eq:Vest}, it holds that
\begin{align}\label{eq:Vest3}
\Vert \Lbfv \Vert_{H(\curl,\Omega)} \, \leq \, C_\sigma \frac{|s|^2}{\real s}\Vert L(s)\widehat{\bfa} \Vert_{\HdivO} \, .
\end{align}
We denote by $\Lbfw_D \in H^1(D)^3$ and $\Lbfw_\Omega\in H^1(\Omega)^3$ the unique solutions of \eqref{eq:MWV} with $\widehat{\bfg} = \widehat{\bfa}$ and $\mathcal{O}=D$ and $\mathcal{O} = \Omega$, respectively.
In light of the definition of $L(s)$ from \eqref{eq:defL} and by using \cite[Thm.\@ 3.20]{McLean00} we estimate
\begin{align*}
\Vert \scurl\big( \bfh_\bfnu ( \gamma_n^- \Lambda_D - \gamma_n^+ \Lambda_\Omega) \widehat{\bfa} \Vert_{\HdivO} \, &\leq \, C \Vert \bfh_\bfnu ( \gamma_n^- \Lambda_D - \gamma_n^+ \Lambda_\Omega) \widehat{\bfa} \Vert_{H^{1/2}(\Gamma)} \\
\, &\leq \, C  \left( \Vert \Lbfw_D \Vert_{H^1(D)^3} + \Vert \Lbfw_\Omega \Vert_{H^1(\Omega)^3} \right) \\
\, &\leq \, C_\sigma  \frac{|s|^2}{\real s} \Vert \widehat{\bfa} \Vert_{H^{1/2}(\Gamma)^3}\, ,
\end{align*}
where we used the boundedness of $\gamma_n^\pm$ and the bound in \eqref{eq:RegBound}. We use that for sufficiently smooth domains the space $H_t^{1/2}(\Gamma)$ is embedded in $\HdivO$ and estimate now using \eqref{eq:curlRegBound}
\begin{align*}
\Vert \bfh_\bfnu (\gamma_T^- \curl \Lambda_D - \gamma_T^+ \curl \Lambda_\Omega)\widehat{\bfa} \big) \Vert_{H^{1/2}(\Gamma)^3}
\, &\leq \, C \left( \Vert \curl \Lbfw_D \Vert_{H^1(D)^3} + \Vert \curl \Lbfw_\Omega \Vert_{H^1(\Omega)^3} \right) \\
\, &\leq \, C_\sigma \frac{|s|^3}{\real s}\Vert \widehat{\bfa} \Vert_{H^{3/2}(\Gamma)^3}\, .
\end{align*}
Therefore, by using \eqref{eq:Vest3}, we find that
\begin{align*}
\Vert \Fcal(s) \widehat{\bfa} \Vert_{H(\curl,\Omega)} \, = \,
\Vert \Lbfv \Vert_{H(\curl,\Omega)} \, \leq \, C_\sigma \frac{|s|^5}{(\real s)^2}\Vert \widehat{\bfa} \Vert_{H^{3/2}(\Gamma)^3} \,  ,
\end{align*}
which yields the bound in \eqref{eq:Fcalbound}.
Finally, we apply the Stratton--Chu formula in the exterior \eqref{eq:Stratton-Chu}, the pointwise operator bounds from \eqref{eq:pwbounds} and the computation from above to obtain for a fixed $\bfx \in \Omega$ that
\begin{align*}
|\Fcal(s)\widehat{\bfa}(\bfx)| \, &= \, |\Lbfv(\bfx)| \, = \, \Big|(\Dop(s) \gamma_t \Lbfv)(\bfx) - \frac{1}{s}(\Sop(s) \gamma_t \curl \Lbfv)(\bfx)\Big| \\
\, &\leq \, C(\sigma,\mathrm{d}_\bfx)\rme^{-\mathrm{d}_\bfx \real s}\left( |s|^2\Vert \gamma_t \Lbfv \Vert_{\HdivO} + |s|\Vert \gamma_t \curl \Lbfv \Vert_{\HdivO} \right)\\
\, &\leq \, C(\sigma,\mathrm{d}_\bfx)\rme^{-\mathrm{d}_\bfx \real s}|s|^7\Vert \widehat{\bfa} \Vert_{H^{3/2}(\Gamma)^3} \, ,
\end{align*}
which yields \eqref{eq:Fcalpwbound}.
\end{proof}
Combining the previous lemma with the fact that $\Fcal(s)\gamma_t\LEi = \LEp$ in $\Omega$ immediately yields the following corollary about frequency dependent bounds for the domain derivative in the Laplace domain.
\begin{cor}
Let $\real s \geq \sigma  >0$, let $\Gamma$ be of class $C^{2}$ and let $\bfh \in C^1(\Gamma, \R^3)$.
Let $\LEp \in H(\curl,\Omega)$ be the domain derivative, i.e.\@ the unique solution to \eqref{eq:ddmw}.
Then,
\begin{align*}
\Vert \LEp \Vert_{H(\curl, \Omega)} \, \leq \, C_\sigma \frac{|s|^5}{(\real s)^2}\Vert \gamma_t \LEi \Vert_{H^{3/2}(\Gamma)^3}
\end{align*}
and for $\bfx \in \Omega$, it holds that
\begin{align}\label{eq:Eppwbound}
| \LEp(\bfx) | \, \leq \, C(\sigma,\mathrm{d}_\bfx) \rme^{- \mathrm{d}_\bfx \real s}|s|^7\Vert \gamma_t \LEi \Vert_{H^{3/2}(\Gamma)^3} \, , 
\end{align}
where $C(\sigma,\mathrm{d}_\bfx)$ does not depend on $s$ but on $\sigma$ and on $\mathrm{d}_\bfx^{-1}$.
\end{cor}
We define the operator $K_\bfh(s)$ for $\bfh \in C^1(\Gamma, \R^3)$ with $0<\Vert \bfh \Vert_{C^1(\Gamma)} < h_0$ by
\begin{align*}
K_\bfh(s) \, := \, -\Sop_{\Gamma_\bfh}(s)\Vop_{\Gamma_\bfh}^{-1}(s)\gamma_{t,\Gamma_\bfh} + \Sop(s)\Vop^{-1}(s)\gamma_{t,\Gamma} - \Fcal(s)\gamma_{t,\Gamma}\, ,
\end{align*}
where $\Sop_{\Gamma_\bfh}$ and $\Vop_{\Gamma_\bfh}$ denote the single layer potential and operator on $\Gamma_h$ and $\gamma_{t,\Gamma_\bfh}$ and $\gamma_{t,\Gamma}$ denote the tangential traces on $\Gamma_\bfh$ and $\Gamma$. We evaluate this operator pointwise in space at the receiver points $\bfz_1, \dots, \bfz_M$, such that
\begin{align*}
(K_\bfh(s) \cdot (\bfz_j))_{j=1,\dots,M} : H^1(D_1)^3 \to \C^{3M} \, .
\end{align*}
A prior evaluation of $\Lambda_{D_1}: H_t^{1/2}(\partial D_1) \to H^1(D_1)^3$ with $\widehat{\bfg} \mapsto \Lbfw$ with $\Lbfw$ solving 
\begin{align*}
\curl^2 \Lbfw + s^2 \Lbfw \, &= \, 0 \quad \text{in } D_1 \, , \\
\bfnu \times \Lbfw \, &= \, \widehat{\bfg} \quad \text{on } \partial D_1 \, ,
\end{align*}
guarantees that the input of $K_\bfh(s)$ is a solution to Laplace domain Maxwell's equations. The combination of \eqref{eq:ddprop} with \eqref{eq:RegBound} yields the following corollary.
\begin{cor}\label{cor:dd}
The operator
\begin{align*}
M_\bfh(s): H_t^{1/2}(\partial D_1 ) \to \C^{3M}\, ,  \quad M_\bfh(s) \, := \, (K_\bfh(s)(\Lambda_{D_1} \cdot ) (\bfz_j))_{j=1,\dots,M}
\end{align*}
is a bounded operator with the bound
\begin{align*}
|M_\bfh(s)\widehat{\bfg}| \, \leq \, C(\sigma,d_\bfz^*) |s|^7 \rme^{-d_\bfz^* \real s} \Vert \bfh \Vert_{C^1(\Gamma)}^2 \Vert \widehat{\bfg} \Vert_{H^{1/2}(\partial D_1)^3}
\end{align*}
for $\real s > \sigma$ and $\mathrm{d}_\bfz^* = \min_{j=1,\dots,M}\text{dist}(\bfz_j, \partial D_1)$.
\end{cor}

\begin{rem}
Corollary \ref{cor:dd} is used in Proposition~\ref{prop:main} below to derive the temporal domain derivative. 
One can refrain from its use and apply \eqref{eq:ddprop} directly to obtain the same result, when making the assumption that the incident time-dependent wave $\Ei(\cdot, t)$ is not supported in $D_1$ for $t>T'$ for some $0<T'<T$.
\end{rem}

\section{The temporal domain derivative for perfect conductors}\label{sec:ddT}
The Laplace domain bounds of the previous section can now be used to return to the time domain. 
We recall the terminology introduced in the beginning of Section~\ref{sec:ddL} for computing Fr\'echet derivatives in the time domain.
For a fixed incident wave $\Ei$ satisfying 
\begin{align}\label{eq:Ei}
\partial_t^2 \Ei + \curl^2 \Ei \, = \, 0 \quad \text{in } \R^3 \times (0,\infty)
\end{align}
with $\Ei(\cdot, 0)$ supported away from $D_1$ with $D\subset \subset D_0\subset \subset D_1$, let $X$ denote the operator that maps the boundary $\Gamma$ to the scattered field $\Es$ solving \eqref{eq:MWEs}, evaluated at the points $\bfz_1, \dots, \bfz_M$, i.e.
$X(\Gamma) = (\Es(\bfz_j))_{j=1,\dots,M} \in L^2(0,T; \R^{3M})$.
Thus, let
\begin{align}\label{eq:FOp}
F_\Gamma : D(F_\Gamma) \subset C^1(\Gamma, \R^3) \to L^2(0,T; \R^{3M}) \, , \qquad F_\Gamma(\bfh) \, = \, X(\Gamma_\bfh)\, ,
\end{align}
where $D(F_\Gamma)$ is a neighborhood of the zero function in $C^1(\Gamma, \R^3)$ that is so small that $\Gamma_\bfh = \bfzeta(\Gamma)$ is a well-defined boundary of a scattering object $D_\bfh$.
By the previous insights and bounds we see that the Fr\'echet derivative of $F_\Gamma$ at the point $0$, that we denote by $F_\Gamma'(0) : C^1(\Gamma,\R^3) \to L^2(0,T; \R^{3M})$ is given by
\begin{align*}
F_\Gamma'(0)\bfh \, = \, ((\Fcal(\partial_t)\gamma_t \Ei)(\bfz_j,\cdot))_{j=1,\dots,M}\, ,
\end{align*}
where $\Fcal(\partial_t)$ is the convolution type operator defined as in \eqref{Heaviside} corresponding to the Laplace domain operator $\Fcal(s)$ from \eqref{eq:Fcal}.
This is a consequence of Corollary~\ref{cor:dd}. The mapping properties of $F_\Gamma'(0)$ in the bulk $\Omega \times (0,T)$ as well as pointwise in $\{\bfz_j\} \times (0,T)$ for $j=1,\dots, M$, follow from Lemma~\ref{lem:Fbounds}. We summarize this in the following proposition.
\begin{proposition}\label{prop:main}
Let $D$ be a bounded $C^2$ domain, $\bfh \in C^1(\Gamma, \R^3)$ and for $r \geq 0$ let $\Ei \in H^{r+7}(0,T; H_{\mathrm{loc}}^{2}(\R^3)^3)$ solve \eqref{eq:Ei}. Let $\Ei(\cdot, 0)$ be supported away from $D_1$ with $D\subset \subset D_0 \subset \subset D_1$. Denote by $\Es \in H_0^{r+5}(0,T; H^1(\Omega)^3)$ the unique solution to \eqref{eq:MWEs}.
The temporal domain derivative for perfect conductors is the causal field $\Et'\in H_0^{r+2}(0, T; \HcurlOmega)$ satisfying
\begin{subequations}\label{eq:ddmwtd}
\begin{align}
\partial_t^2 \Et' + \curl^2 \Et' \, = \, 0 \quad &\text{in } \Omega \times (0, \infty)\, ,\\
\bfnu \times \Et' \, = \, - \scurl \big( \bfh_\bfnu \Et_\bfnu \big) - \bfh_\bfnu \gamma_T \curl \Et \quad &\text{on } \Gamma \times (0,\infty) \, ,
\end{align}
\end{subequations}
where $\Et = \Ei + \Es$. Evaluating $\Et'$ pointwise in space at $\bfz_j$, $j=1,\dots,M$ provides the Fr\'echet derivative of $F_\Gamma$ at zero from \eqref{eq:FOp}, i.e., 
\begin{align}\label{eq:FrechRep}
F_\Gamma'(0)\bfh \, = \, (\bfE'(\bfz_j, \cdot))_{j=1,\dots,M} \in H_0^r(0,T; \R^{3M}) \, .
\end{align}
\end{proposition}

\begin{proof}
By the discussion in the beginning of Section~\ref{subsec:bi} and by using the Heaviside calculus notation \eqref{Heaviside} there holds
\begin{align}\label{eq:Etrep}
\Et' \, = \, \Fcal(\partial_t) \gamma_t \Ei\, ,
\end{align}
where, by Lemma~\ref{lem:Fbounds}, the operator $\Fcal(\partial_t)$ has the mapping properties
\begin{align*}
\Fcal(\partial_t) &: H_0^{r+7}(0,T; H_t^{3/2}(\Gamma)) \to H_0^{r+2}(0,T; \HcurlOmega) \, , \\
\Fcal(\partial_t) \cdot (\bfz_j) &: H_0^{r+7}(0,T; H_t^{3/2}(\Gamma)) \to H_0^{r}(0,T; \R^3) \, , \quad j=1,\dots,M\, .
\end{align*}
In the same way and by using Corollary~\ref{cor:dd} one finds that
\begin{align*}
M_\bfh(\partial_t) : H_0^{r+7}(0,T; H_t^{1/2}(\partial D_1)) \to H_0^{r}(0,T; \R^{3M})\, .
\end{align*}
Moreover, Corollary~\ref{cor:dd} yields the bound 
\begin{align}\label{eq:hbound}
\Vert M_\bfh(\partial_t) \bfg \Vert_{H_0^{r}(0,T; \R^{3M})} \, \leq \, C \Vert \bfh \Vert_{C^1(\Gamma)}^2 \Vert \bfg\Vert_{H_0^{r+7}(0,T; H_t^{1/2}(\partial D_1))}\, .
\end{align}
Now we use that 
\begin{align*}
M_\bfh(\partial_t) \gamma_{t,\partial D_1}\Ei \, = \, ((\Es_\bfh - \Es - \Et')(\bfz_j,\cdot))_{j=1,\dots,M} \, = \, F_\Gamma(\bfh) - F_\Gamma(0) - (\Et'(\bfz_j,\cdot))_{j=1,\dots,M}\, ,
\end{align*}
where $\gamma_{t,\partial D_1}$ denotes the tangential trace taken on $\partial D_1$. Applying \eqref{eq:hbound} now shows \eqref{eq:FrechRep}.
\end{proof}
For the temporal domain derivative $\Et'$ from \eqref{eq:ddmwtd} the previous proof shows in particular the regularity bounds
\begin{align*}
\Vert \Et' \Vert_{H_0^{r}(0,T; \HcurlOmega)} \, &\leq \, C \Vert \gamma_t \Ei \Vert_{H_0^{r+5}(0,T; H_t^{3/2}(\Gamma))}\, , \\
\Vert \Et'(\bfz_j) \Vert_{H_0^{r}(0,T; \R^3)} \, &\leq \, C \Vert \gamma_t \Ei \Vert_{H_0^{r+7}(0,T; H_t^{3/2}(\Gamma))}
\end{align*}
for any $r\geq 0$ and $j=1,\dots,M$.

\subsection{Approximation in time via Runge--Kutta convolution quadrature}
As a semi-discretization in time we consider the Runge--Kutta convolution quadrature method as initially introduced in \cite{BLM11}. For a recent revision of this method including a section on implementation remarks we refer to \cite[Ch.\@ 5]{BanSay22}.
We use an $m$-stage Radau IIA convolution quadrature method to approximate solutions to Maxwell's equation \eqref{eq:MWEs} as well as the temporal domain derivative \eqref{eq:ddmwtd} in time.
For this purpose, let $\tau>0$ be some constant time step size and denote by $t_n = \tau n$ equidistant time points.
We denote the temporal approximations to $\Es(\cdot,t_n)$ and $\Et'(\cdot,t_n)$ by $(\Es_\tau)_n$ and $(\Et'_\tau)_n$, respectively.
Similarly to our previous notation we let $X_\tau(\Gamma) = (\Es_\tau(\bfz_j))_{j=1,\dots,M} \in \ell_\tau^2(0,T; \R^{3M})$ be the sequence of temporal approximations to $\Es$ at the spatial points $\bfz_j$, $j=1,\dots,M$ and for the time points $t_n$, $n=0,\dots, N$.
The space $\ell^2_\tau(0,T; \R^{3M})$ consists of the space of finite sequences endowed with the norm
\begin{align}\label{def:l2norm}
\Vert \Es_\tau \Vert_{\ell^2_\tau(0,T; \R^{3M})}^2 \, := \, \tau \sum_{n=0}^N \sum_{j=1}^M \left|(\Es_\tau)_n(\bfz_j)\right|_{\R^{3}}^2\, .
\end{align}
As earlier, we are interested in finding the Fr\'echet derivative of the time-discrete function
\begin{align}\label{eq:FOpd}
F_{\Gamma,\tau} : D(F_{\Gamma,\tau}) \subset C^1(\Gamma, \R^3) \to \ell_\tau^2(0,T; \R^{3M}) \, , \qquad F_{\Gamma,\tau}(\bfh) \, = \, X_\tau(\Gamma_\bfh)\, .
\end{align}
In light of Proposition~\ref{prop:main} it is unsurprising that the Fr\'echet derivative at zero is given by applying the same Radau IIA convolution quadrature method to the temporal domain derivative's representation from \eqref{eq:Etrep}. We summarize this in the next lemma.
\begin{lem}
Under the same regularity assumptions as in Proposition~\ref{prop:main}, the Fr\'echet derivative of the time discrete operator $F_{\Gamma,\tau}$ from \eqref{eq:FOpd} at zero is given by the convolution quadrature applied to \eqref{eq:Etrep} and evaluated in space at $\bfz_j$, $j=1,\dots,M$, i.e., 
\begin{align}\label{eq:repEtp}
F_{\Gamma,\tau}'(0)\bfh \, = \, (\Et_\tau'(\bfz_j))_{j=1,\dots,M} \, .
\end{align}
\end{lem}
\begin{proof}
We only draw a sketch of this proof as it does not include new techniques.
The idea is to consider the generating functions of $(\Es_{h,\tau}(\bfz_j))_{j=1,\dots,M}$, $(\Es_\tau(\bfz_j))_{j=1,\dots,M}$ and of $(\Et_\tau'(\bfz_j))_{j=1,\dots,M}$.
Then, as in the proof of Proposition~\ref{prop:main}, one applies Corollary~\ref{cor:dd}. 
Here, one needs to apply it to $M_\bfh(s)$ pointwise on the contour described by the Radau IIA convolution quadrature method.
\end{proof}

The application of a Radau IIA convolution quadrature method to approximate both the scattered field $\Es(\bfz_j,\cdot)$ and the domain derivative $\Et'(\bfz_j,\cdot)$ for $j=1,\dots,M$ in time has the extraordinary property that they achieve their full convergence order, which is $p=2m-1$ provided that the incident wave $\Ei$ is smooth enough in time and space. 
While this has been proven in \cite{BLM11}, where it is shown that the full order in time is obtained through the exponential decay in $\real s$, the first application to Maxwell's equations was done in \cite{BalBanSauVeit13}. 
From \eqref{eq:pwbounds} and \eqref{eq:Vcoerc} we see that
\begin{equation*}
| \delta_{\bfz_j} \LEs | \, \leq \, C(\sigma, d_{\bfz_j}) \rme^{-d_{\bfz_j} \real s} |s|^4 \Vert \gamma_t \LEi \Vert_{\HdivO}
\end{equation*}
and applying \cite[Thm.\@ 3]{BLM11} yields that
\begin{align*}
\left| ( \Es(\bfz_j,t_n) - (\Es_\tau)_n(\bfz_j))_{j=1,\dots, M} \right|_{\R^{3M}} \, \leq \, C \tau^{2m-1} \Vert \gamma_t \Ei \Vert_{H_0^r(0,T; \HdivO)}
\end{align*}
for any $r>2m+11/2$ (compare also to \cite[Sec.\@ 6.3]{BanSay22}). The particular regularity in time follows from the requirement in \cite[Thm.\@ 3]{BLM11} that $r>2m+\kappa$ with $\kappa$ from \eqref{eq:pol_bound}  (here $\kappa = 4$) together with the embedding $H_0^{\tilde{r}}(0,T;X) \subset C_0^{r+1}([0,T]; X)$ for any $\tilde{r}>r+3/2$ that is required for the convergence of the remaining integral in \cite[Thm.\@ 3]{BLM11}.
For the temporal domain derivative we gather the approximation in time in the next theorem. We stress that it follows directly from \cite[Thm.\@ 3]{BLM11} and the pointwise bound in Lemma~\ref{lem:Fbounds}.
\begin{thm}\label{thm:sintime}
Let $D$ be a bounded $C^2$ domain and let $\bfh \in C^1(\Gamma, \R^3)$.
The convolution quadrature semi-discretization in time based on a Radau IIA method with $m$ stages from \eqref{eq:repEtp} provides the following estimate at $t_n = n\tau$:
\begin{align*}
\left| ( \Et'(\bfz_j,t_n) - (\Et'_\tau)_n(\bfz_j))_{j=1,\dots, M} \right|_{\R^{3M}} \, \leq \, C \tau^{2m-1} \Vert \gamma_t \Ei \Vert_{H_0^r(0,T; H_t^{3/2}(\Gamma))}\, 
\end{align*}
for $\gamma_t \Ei \in H_0^r(0,T; H_t^{3/2}(\Gamma))$ with $r>2m+17/2$.
\end{thm}
If the temporal error is measured in the $\HcurlOmega$ norm instead, then order reduction occurs. This is also a consequence of \cite[Thm.\@ 3]{BLM11} and we summarize this in the next remark.
\begin{rem}
Under the same assumptions as in Theorem~\ref{thm:sintime} the following error bound holds at $t_n = n\tau$:
\begin{align*}
\left\Vert  \Et'(\cdot,t_n) - (\Et'_\tau)_n \right\Vert_{\HcurlOmega} \, \leq \, C \tau^{m-2} \Vert \gamma_t \Ei \Vert_{H_0^r(0,T; H_t^{3/2}(\Gamma))}\, 
\end{align*}
for $\gamma_t \Ei \in H_0^r(0,T; H_t^{3/2}(\Gamma))$ with $r>2m+13/2$. The order reduction is typical when applying Runge--Kutta convolution quadrature in a bulk.
\end{rem}
\subsection{Approximation in space via Galerkin discretization}
Let $\{\mathcal{T}_h\}$ be a shape regular, quasi-uniform family of meshes of the boundary $\Gamma$.
To achieve the full order of convergence in space we assume that $\Gamma$ is $C^4$ and we perform estimates in the mesh size $h$ always in their full order.
We denote by $X_h$ the finite dimensional $H(\sdiv,\Gamma)$-conforming subspace of Raviart--Thomas basis functions.
Furthermore, denote by $P_h : X_h \to \HdivO$ the natural embedding of $X_h$ into $\HdivO$ and by $P_h': \HdivO \to X_h$ its adjoint with respect to the dual pairing
\begin{align}\label{eq:nuskp}
\langle \bfu, \bfv\rangle_{\bfnu} \, := \, \int_\Gamma (\bfnu \times \bfu) \cdot \overline{\bfv} \ds\, .
\end{align}
The space-discrete, Laplace domain single layer operator $\Vop_h(s)$ is defined by
\begin{equation*}
\Vop_h(s) : X_h \to X_h \, , \quad \Vop_h(s) \, = \, P_h' \Vop(s) P_h \, .
\end{equation*}
The operator $\Vop_h(s)$ inherits the coercivity properties of the operator $\Vop(s)$ from \eqref{eq:Vcoerc} (see also \cite[Lem.\@ 4.8]{BalBanSauVeit13}), i.e., it holds that
\begin{align*}
\Vert \Vop_h^{-1}(s) \Vert_{\HdivO \leftarrow \HdivO} \, \leq \, C \frac{|s|^2}{\real s}\max\{1, |s|^{-1}\}\, ,
\end{align*}
where the constant does not depend on $h$.
Full-discretization error estimates for the direct scattering problem have been performed in \cite{BalBanSauVeit13}.
Here, we can profit from this procedure in the following way.
We write $K(\partial_t^\tau)$ for an operator that has been discretized in time at the time points $t_n = n\tau$ by a Radau IIA convolution quadrature scheme with $m$ stages. 
Moreover, we write $K(\underline{\partial_t^\tau})$ to indicate the $m$ stages approximating the operator $K(\partial_t)$ at the intermediate steps
$\underline{t_n} = (t_n + c_\ell \tau)_{\ell=1}^m$, where $c$ stems from the Runge--Kutta scheme (see \cite[Ch.\@ 5]{BanSay22} for details).
We define the densities 
\begin{align}\label{eq:phis}
\begin{split}
\bfphi \, &:= \, \Vop^{-1}(\partial_t) L(\partial_t) \gamma_t \Ei\, , \\
\widetilde{\bfphi}_{h,\tau} \, &:= \, \Vop_h^{-1}(\underline{\partial_t^\tau})P_h' L(\partial_t) \gamma_t \Ei \, ,\\
 \bfphi_{h,\tau} \, &:= \, \Vop_h^{-1}(\underline{\partial_t^\tau}) L_h(\underline{\partial_t^\tau}) \gamma_t \Ei\, , 
 \end{split}
\end{align}
where $L_h$ is an approximation to $L$ that we define in \eqref{eq:Lhsplit} below and split the full error $\Sop(\partial_t) \bfphi - \Sop(\partial_t^\tau)\bfphi_{h,\tau}$ into
\begin{align}\label{eq:spliterr}
(\Sop(\partial_t) \bfphi)(\bfz,t_n) - (\Sop(\partial_t^\tau)\bfphi_{h,\tau})_n(\bfz)  \,=\, \bfE'(\bfz, t_n) - (\widetilde{\bfE}'_{h,\tau})_n(\bfz) + (\widetilde{\bfE}'_{h,\tau})_n(\bfz) - (\bfE'_{h,\tau})_n(\bfz) \, ,
\end{align}
where $(\widetilde{\bfE}'_{h,\tau})_n(\bfz) = (\Sop(\partial_t^\tau)\widetilde{\bfphi}_{h,\tau})_n(\bfz)$. 
The difference of the first two terms on the right hand side of \eqref{eq:spliterr} can be estimated as in \cite[Thm.\@ 4.11]{BalBanSauVeit13}, since this difference carries the exact load $L(\partial_t)\gamma_t \Ei$ (which is, however, not available in practice).
The difference of the last two terms on the right hand side of \eqref{eq:spliterr} requires a deeper understanding in approximating the boundary load of the domain derivative.

Using the definition of $L(s)$ from \eqref{eq:defL}, it can be seen that for $\widehat{\bfa} = \gamma_t \LEi$ there holds
\begin{align}\label{eq:Lsplit}
L(s)\gamma_t \LEi \, = \, T_1(s)\gamma_t\LEi + T_2(s)\gamma_t\LEi
\end{align}
with
\begin{subequations}\label{eq:Talter}
\begin{align}
T_1(s)\gamma_t\LEi \, &:= \, -\scurl\big( \bfh_\bfnu ( \gamma_n^- \Lambda_D - \gamma_n^+ \Lambda_\Omega) \gamma_t\LEi \big) = -\frac{1}{s}\scurl\left(\bfh_\bfnu \sdiv \bigl( \Vop^{-1}(s) \gamma_t \LEi \bigr)\right) \, , \label{eq:T1alter}\\
 T_2(s)\gamma_t\LEi \, &:= -\bfh_\bfnu (\gamma_T^- \curl \Lambda_D - \gamma_T^+ \curl \Lambda_\Omega)\gamma_t \LEi \big)  =  s \bfh_\bfnu \bigl( \bfnu \times \Vop^{-1}(s) \gamma_t \LEi\bigr)\, . 
\end{align}
\end{subequations}
Using the operator bounds in the proof of Lemma \ref{lem:Fbounds} shows that
\begin{equation*}
\Vert T_1(s) \Vert_{\HdivO \leftarrow H_t^{1/2}(\Gamma)} \, \leq \, C \frac{|s|^{2}}{\real s}
\quad \text{and } \quad 
\Vert T_2(s) \Vert_{\HdivO \leftarrow H_t^{3/2}(\Gamma)} \, \leq \, C \frac{|s|^3}{\real s}\, .
\end{equation*}

We let $\widehat{\bfg}$ be sufficiently regular (e.g., $\widehat{\bfg} \in H^{3/2}(\sdiv,\Gamma)$) and abbreviate $\bfj = \Vop^{-1}(s) \widehat{\bfg}$. Using an argument based on C\'ea's lemma together with a best-approximation result using the orthogonal projection $\Pi_{h, \sdiv} : \HdivO \to X_h$ (see \cite[Thm.\@ 14]{BufHip03}) one finds that the Galerkin approximation $\bfj_h = \Vop_h^{-1}(s)P_h'\widehat{\bfg}$ satisfies
\begin{align}\label{eq:galapp}
\Vert \bfj - \bfj_h \Vert_{\HdivO} \, \leq \, C \frac{|s|^4}{(\real s)^2} h^{3/2}\Vert \bfj \Vert_{H^1(\sdiv, \Gamma)}
\, \leq \, C \frac{|s|^{8}}{(\real s)^3} h^{3/2}\Vert \widehat{\bfg} \Vert_{H^{3/2}(\sdiv,\Gamma)}\, ,
\end{align}
where we used the bound \eqref{eq:Vm1bound1} for the last inequality.
We can also prove an error bound in the $H(\sdiv,\Gamma)$-norm, which we state in the next lemma.

\begin{lem}\label{lem:jmjh}
Let $\Gamma$ be at least $C^4$, $\widehat{\bfg} \in H^{3/2}(\sdiv,\Gamma)$, $\bfj = \Vop^{-1}(s) \widehat{\bfg}$ and $\bfj_h = \Vop_h^{-1}(s)P_h'\widehat{\bfg}$. Then,
\begin{align}\label{eq:jbound}
\Vert \bfj - \bfj_h \Vert_{H(\sdiv, \Gamma)} \, \leq \, C h \frac{|s|^{8}}{(\real s)^3}  \Vert \widehat{\bfg} \Vert_{H^{3/2}(\sdiv,\Gamma)}
\end{align}
\end{lem}
\begin{proof}
We consider the orthogonal projection with respect to the $H(\sdiv, \Gamma)$ scalar product into $X_h$ denoted by
$Q_{h,\sdiv}$.
This operator satisfies
\begin{align}\label{eq:Qh0def}
h^{1/2}\Vert \bfu - Q_{h, \sdiv}\bfu \Vert_{H(\sdiv,\Gamma)} + \Vert \bfu - Q_{h, \sdiv}\bfu \Vert_{\HdivO} \, \leq \, C h^{3/2} \Vert \bfu \Vert_{H^{1}(\sdiv, \Gamma)}  
\end{align}
for any $\bfu \in H^1(\sdiv, \Gamma)$ (see \cite[Eq.\@ (3.41), (3.47)]{Chris04}). Now, we use the triangle inequality and apply the inverse estimates \cite[Eq.\@ (1.47), (1.48)]{Ben84} 
\begin{align*}
\Vert \bfj - \bfj_h \Vert_{H(\sdiv, \Gamma)} \, &\leq \, \Vert \bfj - Q_{h,\sdiv} \bfj \Vert_{H(\sdiv, \Gamma)} + \Vert Q_{h,\sdiv} \bfj - \bfj_h \Vert_{H(\sdiv, \Gamma)} \\
\,& \leq \, \Vert \bfj - Q_{h,\sdiv} \bfj \Vert_{H(\sdiv, \Gamma)}+ Ch^{-1/2}\Vert Q_{h,\sdiv} \bfj - \bfj_h \Vert_{\HdivO} \, .
\end{align*}
The first term can be bounded by \eqref{eq:Qh0def}. For the second term we apply another triangle inequality and use \eqref{eq:galapp} and \eqref{eq:Qh0def} to see that
\begin{align*}
\Vert Q_{h,\sdiv} \bfj - \bfj_h \Vert_{\HdivO} \, &\leq \, \Vert Q_{h,\sdiv} \bfj - \bfj \Vert_{\HdivO} + \Vert \bfj- \bfj_h \Vert_{\HdivO} \\
&\leq \, C h^{3/2}  \frac{|s|^{8}}{(\real s)^3} \Vert \widehat{\bfg} \Vert_{H^{3/2}(\sdiv,\Gamma)} \, .
\end{align*}
This yields \eqref{eq:jbound}.
\end{proof}

We consider now the dual problem as also done in \cite[Thm.\@ 4.11]{BalBanSauVeit13}.
\begin{lem}\label{lem:dual}
Let $\bfz \in \Omega$ and define $\ell_i(\bfv) = \Sop_i(s)\bfv(\bfz) = \delta_\bfz \Sop_i(s)\bfv$, where $\Sop_i(s)\bfv$ denotes the $i$-th component of $\Sop(s)\bfv$.
Let $\widehat{\bfw}_i\in \HdivO$ and $\widehat{\bfw}_{i,h} \in X_h$ satisfy the problems
\begin{subequations}
\begin{align}
\overline{\langle \widehat{\bfw}_i, \Vop(s)\bfv \rangle_{\bfnu}} \, &= \, \ell_i(\bfv) \quad \text {for all } \bfv \in \HdivO \, , \label{def:wi} \\
\overline{\langle \widehat{\bfw}_{i,h}, \Vop_h(s)\bfv_h \rangle_{\bfnu}} \, &= \, \ell_i(\bfv_h) \quad \text {for all } \bfv_h \in X_h \, . \label{def:wih}
\end{align}
\end{subequations}
If $\Gamma$ is of regularity $C^{k+1}$, $k\in \N$, the function $\widehat{\bfw}_i \in H_t^{k-1/2}(\Gamma)$ and satisfies
\begin{align}\label{eq:wbound}
 \Vert \widehat{\bfw}_i \Vert_{H^{k-1/2}(\Gamma)^3} \, \leq \, C(\sigma,d_\bfz) |s|^{k+3}\rme^{-d_\bfz \real s } \quad \text{for any } s \in \C\text{ with } \real s > \sigma > 0\, ,
\end{align}
where $d_\bfz=\mathrm{dist}(\bfz,\Gamma)$. Moreover, for $\Gamma$ of regularity $C^4$ we obtain the bound
\begin{align}\label{eq:dualbounds}
\Vert \widehat{\bfw}_i - \widehat{\bfw}_{i,h}\Vert_{H(\sdiv,\Gamma)} \, \leq \, C(\sigma,d_\bfz)h |s|^{10}\rme^{-d_\bfz \real s} \, .
\end{align}
\end{lem}
\begin{proof}
Taking slight notational differences into account, one can proceed as in \cite[Proof of Thm.\@ 9]{BanSau12} to see that
\begin{align*}
\widehat{\bfw}_i \, = \, \overline{\Vop^{-1}(s) \gamma_t K_i(\bfz, \cdot)}\, , \quad \text{where} \quad K_i(\bfz, \bfy) \, = \, -s \Phi_s(\bfz, \bfy) \bfe_i - \frac{1}{s}\frac{\partial}{\partial z_i} \sgrad_\bfy \Phi_s(\bfz, \bfy)\, .
\end{align*}
For $\bfz$ away from $\Gamma$, the function $K_i$ is smooth and straightforward computation yields (see \cite[p.\@ 13]{BanSau12})
\begin{align*}
\Vert \gamma_t K_i(\bfz, \cdot) \Vert_{H^{k-1/2}(\Gamma)^3} \, \leq \, C(d_\bfz) \rme^{-d_\bfz \real s} |s|^{k+1}\, ,
\end{align*}
which, together with \eqref{eq:Vm1bound2}, yields the bounds on $\widehat{\bfw}_i$ in \eqref{eq:wbound}.
The bounds in \eqref{eq:dualbounds} are obtained as follows. 
For the bound in the energy norm, one proceeds by using C\'ea's lemma together with \cite[Thm.\@ 14]{BufHip03} as cited earlier and obtains
\begin{align*}
\Vert \widehat{\bfw}_i - \widehat{\bfw}_{i,h}\Vert_{\HdivO} \, \leq \, C h^{3/2} \frac{|s|^4}{(\real s)^2}\Vert \widehat{\bfw}_i\Vert_{H^1(\sdiv, \Gamma)}  
\, \leq \, C(\sigma,d_\bfz)h^{3/2} |s|^{10}\rme^{-d_\bfz \real s}\, ,
\end{align*}
where we used that $H_t^{5/2}(\Gamma) \subset H^1(\sdiv,\Gamma)$.
The error estimates in \eqref{eq:dualbounds} can be obtained by proceeding as in Lemma~\ref{lem:jmjh}.
\end{proof}
By duality of $\scurl$ and $\sCurl$ it holds that 
\begin{align*}
\langle T_1(s) \widehat{\bfg}, \bfy \rangle_{L_t^2(\Gamma)}
\, = \, \frac{1}{s} \langle h_\bfnu \sdiv( \Vop^{-1}(s) \widehat{\bfg}), \sCurl \bfy \rangle_{L^2(\Gamma)} \quad \text{for all } \bfy \in \HcurlO\, .
\end{align*}
Therefore, we approximate this scalar product using $T_{1,h}(s)\widehat{\bfg}$ satisfying
\begin{align}\label{def:T1h}
\langle T_{1,h}(s) \widehat{\bfg}, \bfy_h \rangle_{L_t^2(\Gamma)} 
\, := \, \frac{1}{s} \langle h_\bfnu \sdiv( \Vop_h^{-1}(s)P_h' \widehat{\bfg}), \sCurl \bfy_h \rangle_{L^2(\Gamma)} \quad \text{for all } \bfy_h \in \bfnu \times X_h\, .
\end{align}
Furthermore, we define $T_{2,h}(s)\widehat{\bfg}$ via
\begin{align}\label{def:T2h}
\langle T_{2,h}(s)\widehat{\bfg}, \bfy_h \rangle_{L_t^2(\Gamma)} \, := \, \langle s \bfh_\bfnu \bigl( \bfnu \times \Vop_h^{-1}(s)P_h' \widehat{\bfg}\bigr),  \bfy_h \rangle_{L_t^2(\Gamma)} \quad \text{for all } \bfy_h \in \bfnu \times X_h\, .
\end{align}
\begin{rem}
We do not intend to derive error estimates of $T_1(s)\widehat{\bfg} - T_{1,h}(s)\widehat{\bfg}$ in the $\HdivO$-norm and of 
$T_2(s)\widehat{\bfg} - T_{2,h}(s)\widehat{\bfg}$ in the $L_t^2(\Gamma)$-norm.
While this would be possible through the introduction of Buffa--Christiansen functions, it is not required in order to obtain pointwise estimates in space for the domain derivative.
\end{rem}
The next lemma will be needed in order to derive the full error estimate.
\begin{lem}\label{lem:TmTh}
Let $\Gamma$ be at least $C^4$, $\widehat{\bfg} \in H^{3/2}(\sdiv,\Gamma)$, $\bfh \in C^2(\Gamma,\R^3)$ and let $\widehat{\bfw}_{i,h} \in X_h$ solve \eqref{def:wih}. 
Furthermore, let $\bfz \in \Omega$.
Then,
\begin{subequations}
\begin{align}
|\langle \widehat{\bfw}_{i,h}, (T_1(s) - T_{1,h}(s)) \widehat{\bfg} \rangle_\bfnu | \, &\leq \, C(\sigma,d_\bfz) h^2 |s|^{18} \Vert \widehat{\bfg} \Vert_{H^{3/2}(\sdiv,\Gamma)} \Vert \bfh_\nu \Vert_{C^2(\Gamma)}\rme^{-d_\bfz \real s}\, , \label{eq:T1mT1h} \\
|\langle \widehat{\bfw}_{i,h}, (T_2(s) - T_{2,h}(s)) \widehat{\bfg} \rangle_\bfnu | \, &\leq \, C(\sigma,d_\bfz) h^2 |s|^{19} \Vert \widehat{\bfg} \Vert_{H^{3/2}(\sdiv,\Gamma)} \Vert \bfh_\nu \Vert_{C^2(\Gamma)}\rme^{-d_\bfz \real s}\, , \label{eq:T2mT2h}
\end{align}
\end{subequations}
for any $s \in \C$ with $\real s > \sigma > 0$, where $d_\bfz=\mathrm{dist}(\bfz,\Gamma)$.
\end{lem}
\begin{proof}
Since $\sCurl (\bfnu \times \bfu) = \sdiv(\bfu)$ for all $\bfu \in \HdivO$ (see, e.g., \cite[p.\@ 261]{KiHe15}) there holds that
\begin{align*}
\langle \widehat{\bfw}_{i,h}, (T_1(s) - T_{1,h}(s)) \widehat{\bfg} \rangle_\bfnu \, &= \, \frac{1}{\overline{s}}\langle \sCurl (\bfnu \times \widehat{\bfw}_{i,h}), \bfh_\bfnu \sdiv(\bfj - \bfj_h) \rangle_{L^2(\Gamma)} \\
\, &= \, \frac{1}{\overline{s}}\Big( \langle \sdiv (\widehat{\bfw}_{i,h} - \widehat{\bfw}_i), \bfh_\bfnu \sdiv(\bfj - \bfj_h) \rangle_{L^2(\Gamma)} \\
&\,  \phantom{= -\frac{1}{s}\Big(}  + \langle \sdiv \widehat{\bfw}_{i}, \bfh_\bfnu \sdiv(\bfj - \bfj_h) \rangle_{L^2(\Gamma)}\Big)\, .
\end{align*}
For the first term we use the Cauchy--Schwarz inequality, \eqref{eq:jbound} and \eqref{eq:dualbounds} to get that
\begin{align*}
|\langle \sdiv (\widehat{\bfw}_{i,h} - \widehat{\bfw}_i), \bfh_\bfnu \sdiv(\bfj - \bfj_h) \rangle_{L^2(\Gamma)}| \, \leq \, C(\sigma, d_\bfz) h^2 |s|^{18} \Vert \bfh_\bfnu \Vert_{C^0} \rme^{-d_\bfz \real s} \Vert \widehat{\bfg} \Vert_{H^{3/2}(\sdiv,\Gamma)}\, .
\end{align*}
For the second term we use the duality between $\sdiv$ and $\sgrad$, the property that the adjoint of $\Vop(s)$ with respect to the dual pairing \eqref{eq:nuskp} fulfills $\Vop(s)^* = -\overline{\Vop(s)}$ (see \cite[Lem.\@ 5.61]{KiHe15}) and Galerkin orthogonality to see that
\begin{align*}
\langle \sdiv \widehat{\bfw}_{i}, \bfh_\bfnu \sdiv(\bfj - \bfj_h) \rangle_{L^2(\Gamma)} \, &= \, - \langle \sgrad (\bfh_\bfnu \sdiv \widehat{\bfw}_i), \bfj - \bfj_h \rangle_{L_t^2(\Gamma)} \\
\, &= \,  \langle \overline{\Vop(s)\Vop^{-1}(s)}\scurl (\bfh_\bfnu \sdiv \widehat{\bfw}_i), \bfj - \bfj_h \rangle_{\bfnu} \\
\ &= \,  -\langle \overline{\Vop^{-1}(s)}\scurl (\bfh_\bfnu \sdiv \widehat{\bfw}_i), \Vop(s)(\bfj - \bfj_h) \rangle_{\bfnu} \\
\ &= \,  -\langle \overline{\Vop^{-1}(s)}\scurl (\bfh_\bfnu \sdiv \widehat{\bfw}_i) - \bfx_h, \Vop(s)(\bfj - \bfj_h) \rangle_{\bfnu}
\end{align*}
for any $\bfx_h \in X_h$. The Cauchy--Schwarz inequality, the bound of $\Vop(s)$ from \cite[Thm.\@ 6.4]{BanSay22}, the use of \cite[Thm.\@ 14]{BufHip03} the embedding $H_t^{3/2}(\Gamma)\subset H^{1/2}(\sdiv,\Gamma)\subset H(\sdiv,\Gamma)$ as well as the boundedness of the operators $\scurl : H^s(\Gamma) \to H_t^{s-1}(\Gamma)$ and $\sdiv : H_t^s(\Gamma) \to H^{s-1}(\Gamma)$ (see e.g. \cite[Prop.\@ 3.3]{MelWör26}) together with $\sdiv (\scurl u) = 0$ for $u\in H^{3/2}(\Gamma)$  give 
\begin{align}\label{eq:comp2T1mT1h}
\begin{split}
&|\langle \sdiv \widehat{\bfw}_{i}, \bfh_\bfnu  \sdiv(\bfj - \bfj_h) \rangle_{L^2(\Gamma)}| \\
 \, &\leq \, \inf_{\bfx_h \in X_h}\Vert \overline{\Vop^{-1}(s)}\scurl (\bfh_\bfnu \sdiv \widehat{\bfw}_i) - \bfx_h \Vert_{\HdivO} \Vert \Vop(s)(\bfj - \bfj_h)  \Vert_{\HdivO} \\
 \, &\leq \, C(\sigma) h^2 |s|^{10} \Vert \widehat{\bfg} \Vert_{H^{3/2}(\sdiv,\Gamma)} \Vert  \overline{\Vop^{-1}(s)}\scurl (\bfh_\bfnu \sdiv \widehat{\bfw}_i) \Vert_{H(\sdiv, \Gamma)} \\
 \, &\leq \, C(\sigma) h^2 |s|^{13} \Vert \widehat{\bfg} \Vert_{H^{3/2}(\sdiv,\Gamma)} \Vert  \scurl (\bfh_\bfnu \sdiv \widehat{\bfw}_i) \Vert_{H^{1/2}(\Gamma)^3} \\
  \, &\leq \, C(\sigma) h^2 |s|^{13} \Vert \widehat{\bfg} \Vert_{H^{3/2}(\sdiv,\Gamma)} \Vert \bfh_\bfnu \Vert_{C^2} \Vert  \widehat{\bfw}_i \Vert_{H^{5/2}(\Gamma)^3} \, ,
  \end{split}
\end{align}
where $C(\sigma)$ depends on $\sigma$, where $\real s > \sigma >0$. Applying the bound in \eqref{eq:wbound} yields the bound in \eqref{eq:T1mT1h}. To prove \eqref{eq:T2mT2h} we compute, similarly as before
\begin{align*}
\langle \widehat{\bfw}_{i,h}, (T_2(s) - T_{2,h}(s)) \widehat{\bfg} \rangle_\bfnu 
\, = \, \overline{s}\left( \langle \bfh_\bfnu( \widehat{\bfw}_{i,h} - \widehat{\bfw}_{i}), \bfj - \bfj_h\rangle_{L_t^2(\Gamma)} + \langle \bfh_\bfnu \widehat{\bfw}_{i}, \bfj - \bfj_h\rangle_{L_t^2(\Gamma)} \right)
\end{align*}
and estimate both terms. For the first one we use Cauchy--Schwarz, \eqref{eq:jbound} and \eqref{eq:dualbounds} and see that
\begin{align*}
|\langle \bfh_\bfnu( \widehat{\bfw}_{i,h} - \widehat{\bfw}_{i}), \bfj - \bfj_h\rangle_{L_t^2(\Gamma)}| \, &\leq \, C(\sigma, d_\bfz) h^2 |s|^{18} \Vert \widehat{\bfg} \Vert_{H^{3/2}(\sdiv,\Gamma)} \Vert \bfh_\bfnu \Vert_{C^0} \rme^{-d_\bfz \real s} \, .
\end{align*}
Using the same results as in the computation of \eqref{eq:comp2T1mT1h} shows that
\begin{align*}
|\langle \bfh_\bfnu \widehat{\bfw}_{i}, \bfj - \bfj_h\rangle_{L_t^2(\Gamma)}| \, &\leq \, C(\sigma) |s|^{10} h^{2} \Vert \widehat{\bfg} \Vert_{H^{3/2}(\sdiv,\Gamma)} \Vert \overline{\Vop^{-1}(s)} \bfh_\bfnu \bfnu\times \widehat{\bfw}_i \Vert_{H(\sdiv,\Gamma)} \\
\, &\leq \, C(\sigma) |s|^{13} h^{2} \Vert \widehat{\bfg} \Vert_{H^{3/2}(\sdiv,\Gamma)} \Vert \bfh_\bfnu \Vert_{C^2} \Vert  \widehat{\bfw}_i \Vert_{H^{3/2}(\Gamma)^3} \, .
\end{align*}
The bound in \eqref{eq:wbound} now ultimately yields \eqref{eq:T2mT2h}.
\end{proof}

\subsection{Full error estimation}\label{sec:FE}
Finally we can estimate the full error $|(\Sop(\partial_t) \bfphi)(\bfz,t_n) - (\Sop(\partial_t^\tau)\bfphi_{h,\tau})_n(\bfz)|$ with $\bfphi$ and $\bfphi_{h,\tau}$ as in \eqref{eq:phis}. 
We return to the split in \eqref{eq:spliterr}. To estimate the first difference, we proceed similar to \cite[Proof of Thm.\@ 4.11]{BalBanSauVeit13}, while for the second difference we use the estimates we proved earlier.
The space-discrete operator $L_h(s)\Lbfg$ is defined by
\begin{align}\label{eq:Lhsplit}
L_h(s) \Lbfg \, = \, T_{1,h}(s)\Lbfg + T_{2,h}(s)\Lbfg
\end{align}
with $T_{1,h}(s)\widehat{\bfg}$ and $T_{2,h}(s)\widehat{\bfg}$ defined in \eqref{def:T1h} and \eqref{def:T2h}.
\begin{thm}\label{thm:main}
Let a convolution quadrature method in time based on a Radau IIA method with $m$ stages and a Galerkin approximation in space with lowest order Raviart--Thomas functions be applied to approximate the temporal domain derivative $\bfE'$ from \eqref{eq:ddmwtd}.
In particular, let the time and space discretized temporal domain derivative be denoted by $\bfE_{h,\tau}' = \Sop(\partial_t^\tau)\bfphi_{h,\tau}$, where $\bfphi_{h,\tau}$ is defined in \eqref{eq:phis} with the operator $L_h$ defined in the Laplace domain in \eqref{eq:Lhsplit}.
Let $\Gamma$ be at least $C^4$, $\bfh \in C^2(\Gamma,\R^3)$ and let $\gamma_t \Ei \in H_0^r(0,T; H^{3/2}(\sdiv,\Gamma))$ with $r>2m+41/2$.
Then, for any $\bfz \in \Omega$ it holds that
\begin{align*}
\left|\bfE'(\bfz, t_n) - (\bfE_{h,\tau}')_n(\bfz) \right| \, \leq \, C(h^{2} + \tau^{2m-1})\Vert \gamma_t \Ei \Vert_{H_0^{r}(0,T ; H^{3/2}(\sdiv,\Gamma))}\, .
\end{align*}
\end{thm}
\begin{proof}
We can split $\bfE'(\bfz, t_n) - (\bfE_{h, \tau}')_n(\bfz)$ as in \eqref{eq:spliterr}. We start by estimating 
\begin{align*}
\left| \bfE'(\bfz, t_n) - (\widetilde{\bfE}_{h, \tau}')_n(\bfz)  \right| \, \leq \, \left| \bfE'(\bfz, t_n) - \bfE'_h(\bfz,t_n)  \right| + \left| \bfE'_h(\bfz,t_n) - (\widetilde{\bfE}_{h, \tau}')_n(\bfz)  \right|\, ,
\end{align*}
where $\bfE'_h(\bfz,t_n) = (\Sop(\partial_t) \bfphi_h)(\bfz,t_n)$ with $\bfphi_h = \Vop_h^{-1}(\partial_t) P_h' L(\partial_t)\gamma_t \Ei$.
Since $\Vop_h^{-1}(s)$ satisfies the same bound as $\Vop^{-1}(s)$ uniformly in $h$, the same argument as for the semi-discrete error in time from Theorem~\ref{thm:sintime} can be applied to get the bound
\begin{align*}
\left| \bfE'_h(\bfz,t_n) - (\widetilde{\bfE}_{h, \tau}')_n(\bfz)  \right| \, \leq \, C \tau^{2m-1} \Vert \gamma_t \Ei \Vert_{H_0^r(0,T; H^{3/2}(\Gamma)^3)}
\end{align*}
for $r >2m+17/2$. For the other term we apply Nitsche's trick. Let $\widehat{\bfw}_i$ solve \eqref{def:wi}. Then,
\begin{align*}
\Sop_i(s)(\widehat{\bfphi} - \widehat{\bfphi}_h)(\bfz) \, = \, \overline{\langle \widehat{\bfw}_i, \Vop(s)(\widehat{\bfphi} - \widehat{\bfphi}_h) \rangle_\bfnu} \, = \, \overline{\langle \widehat{\bfw}_i - \Pi_{h,\sdiv}\widehat{\bfw}_i, \Vop(s)(\widehat{\bfphi} - \widehat{\bfphi}_h) \rangle_\bfnu}\, ,
\end{align*}
where the last step follows from Galerkin orthogonality.
Using Cauchy--Schwarz, \cite[Thm.\@ 14]{BufHip03} and the bounds on $\widehat{\bfw}_i$ from Lemma~\ref{lem:dual} gives 
\begin{align*}
\left| \LEt'(\bfz) - \LEt_h'(\bfz)  \right| \, \leq \, C(\sigma, d_\bfz) h^{3/2} |s|^{6} \rme^{-d_\bfz \real s} \Vert \Vop(s)(\widehat{\bfphi} - \widehat{\bfphi}_h) \Vert_{\HdivO} \, .
\end{align*}
Estimating using the bound for $\Vop(s)$ from \cite[Thm.\@ 6.4]{BanSay22} and the estimate \eqref{eq:Vcoerc} (see also \cite[Sec.\@ 2.8]{BanSay22}) together with \cite[Thm.\@ 14]{BufHip03} yields
\begin{align*}
\Vert \Vop(s)(\widehat{\bfphi} - \widehat{\bfphi}_h) \Vert_{\HdivO} \, \leq \, C(\sigma) |s|^6 h^{1/2}\Vert \widehat{\bfphi} \Vert_{H(\sdiv,\Gamma)}\, .
\end{align*}
Applying \eqref{eq:Vm1bound1} with $k=1$ yields
\begin{align*}
\Vert \widehat{\bfphi} \Vert_{H(\sdiv,\Gamma)} \, \leq \, C(\sigma) |s|^3 \left( \Vert T_1(s) \gamma_t \LEi \Vert_{H^{1/2}(\sdiv,\Gamma)} + \Vert T_2(s)\gamma_t \LEi \Vert_{H^{1/2}(\sdiv,\Gamma)} \right)\,.
\end{align*}
As before, we estimate both terms and obtain
\begin{align*}
|s|\Vert T_1(s) \gamma_t \LEi \Vert_{H^{1/2}(\sdiv,\Gamma)} + \Vert T_2(s)\gamma_t \LEi \Vert_{H^{1/2}(\sdiv,\Gamma)} \, \leq \, C(\sigma) |s|^4 \Vert \bfh_\bfnu \Vert_{C^2} \Vert \gamma_t \LEi \Vert_{H^{3/2}(\sdiv,\Gamma)}\, ,
\end{align*}
where we used \eqref{eq:Vb1} for the estimation of $\Vert T_2(s)\gamma_t \LEi \Vert_{H^{1/2}(\sdiv,\Gamma)}$. Therefore,
\begin{align*}
\left| \LEt'(\bfz) - \LEt_h'(\bfz)  \right| \, \leq \, C(\sigma, d_\bfz) h^{2} |s|^{19} \rme^{-d_\bfz \real s} \Vert \bfh_\bfnu \Vert_{C^2} \Vert \gamma_t \LEi \Vert_{H^{3/2}(\sdiv,\Gamma)}
\end{align*}
for any $s$ with $\real s > \sigma > 0$. Back to the time domain, this yields
\begin{align*}
\left| \bfE'(\bfz, t_n) - \bfE'_h(\bfz,t_n)  \right| \, \leq \, C h^2 \Vert \gamma_t \Ei \Vert_{H_0^r(0,T; H^{3/2}(\sdiv,\Gamma))}
\end{align*}
for $r >39/2$, so in total we obtain that 
\begin{align}\label{eq:bound1}
\left| \bfE'(\bfz, t_n) - (\widetilde{\bfE}_{h, \tau}')_n(\bfz)  \right| \, \leq \, C ( \tau^{2m-1} + h^2 ) \Vert \gamma_t \Ei \Vert_{H_0^r(0,T; H^{3/2}(\sdiv,\Gamma))}
\end{align}
for $r > \max\{2m+17/2, 39/2 \}$.
For the second part we find
\begin{align}\label{eq:2ndterm}
(\widetilde{\bfE}'_{h,\tau})_n(\bfz) - (\bfE'_{h,\tau})_n(\bfz) \, = \, (\Sop(\partial_t^\tau)(\widetilde{\bfphi}_{h,\tau} - \bfphi_{h,\tau}))_n(\bfz)
\end{align}
with $\widetilde{\bfphi}_{h,\tau}$ and $\bfphi_{h,\tau}$ from \eqref{eq:phis}. Thus, we consider the term
\begin{align}\label{eq:2ndtermlong}
\delta_\bfz\Sop(\underline{\partial_t^\tau})\Vop_h^{-1}(\underline{\partial_t^\tau})(P_h' L(\partial_t) - L_h(\underline{\partial_t^\tau})) \gamma_t\Ei\, ,
\end{align}
since the term in \eqref{eq:2ndterm} is just a restriction to some special time points of \eqref{eq:2ndtermlong}.
We split the term in \eqref{eq:2ndtermlong} into $P_1+P_2$ with
\begin{subequations}
\begin{align}
P_1 \, &:= \, \delta_\bfz\Sop(\underline{\partial_t^\tau})\Vop_h^{-1}(\underline{\partial_t^\tau})(P_h' L(\underline{\partial_t^\tau}) - L_h(\underline{\partial_t^\tau})) \gamma_t\Ei\, , \\
P_2 \, &:= \,   \delta_\bfz\Sop(\underline{\partial_t^\tau})\Vop_h^{-1}(\underline{\partial_t^\tau})(P_h' L(\partial_t) - P_h'L(\underline{\partial_t^\tau})) \gamma_t\Ei \, . \label{def:P2}
\end{align}
\end{subequations}
For the first term we consider the Laplace domain expression that is $\delta_\bfz\Sop(s)(\widehat{\bfm}^*_h - \widehat{\bfm}_h)$ with
\begin{align*}
\widehat{\bfm}^*_h \, := \Vop_h^{-1}(s)P_h'L(s) \gamma_t\LEi \in X_h  \qquad \text{and} \qquad
\widehat{\bfm}_h \, := \, \Vop_h^{-1}(s)L_h(s) \gamma_t\LEi \in X_h \, .
\end{align*}
Using \eqref{def:wih} we recognize that
\begin{align*}
\delta_\bfz\Sop_i(s)(\widehat{\bfm}^*_h - \widehat{\bfm}_h) \, = \, \overline{\langle \widehat{\bfw}_{i,h}, \Vop_h(s)(\widehat{\bfm}^*_h - \widehat{\bfm}_h) \rangle_\bfnu}
\, = \, \overline{\langle \widehat{\bfw}_{i,h}, (L(s) - L_h(s))\gamma_t\LEi \rangle_\bfnu}\, .
\end{align*}
Now we split using \eqref{eq:Lsplit} and \eqref{eq:Lhsplit} to get
\begin{align*}
\langle \widehat{\bfw}_{i,h}, (L(s) - L_h(s))\gamma_t\LEi \rangle_\bfnu \, = \, \langle \widehat{\bfw}_{i,h}, (T_1(s) - T_{1,h}(s))\gamma_t\LEi \rangle_\bfnu + \langle \widehat{\bfw}_{i,h}, (T_2(s) - T_{2,h}(s))\gamma_t\LEi \rangle_\bfnu
\end{align*}
and apply Lemma~\ref{lem:TmTh} to get
\begin{align*}
|\delta_\bfz\Sop_i(s)(\widehat{\bfm}^*_h - \widehat{\bfm}_h) | \, \leq \, 
C(\sigma,d_\bfz) h^2 |s|^{19} \Vert \gamma_t\LEi \Vert_{H^{3/2}(\sdiv,\Gamma)} \Vert \bfh_\nu \Vert_{C^2(\Gamma)}\rme^{-d_\bfz \real s} \, .
\end{align*}
Now we can use \cite[Thm.\@ 3]{BLM11} to get the bound
\begin{align}\label{eq:boundP1}
| (P_1)_n | \, \leq \, C h^{2}( \tau^{2m-1} + 1 ) \Vert \gamma_t \Ei \Vert_{H_0^r(0,T; H^{3/2}(\sdiv,\Gamma))} 
\end{align}
for $r>2m+41/2$. We split $P_2$ from \eqref{def:P2} into 
\begin{multline}\label{eq:P2split}
P_2 \, = \, \delta_\bfz\Sop(\partial_t)\Vop_h^{-1}(\partial_t)P_h' L(\partial_t)\gamma_t \Ei - \delta_\bfz\Sop(\underline{\partial_t^\tau})\Vop_h^{-1}(\underline{\partial_t^\tau})P_h' L(\underline{\partial_t^\tau})\gamma_t \Ei \\
+\delta_\bfz\Sop(\underline{\partial_t^\tau})\Vop_h^{-1}(\underline{\partial_t^\tau})P_h' L(\partial_t)\gamma_t \Ei - \delta_\bfz\Sop(\partial_t)\Vop_h^{-1}(\partial_t)P_h' L(\partial_t)\gamma_t \Ei \, .
\end{multline}
Due to the Laplace domain bound 
\begin{align*}
\left| \delta_\bfz\Sop(s)\Vop_h^{-1}(s)P_h' L(s)\gamma_t \LEi \right| \, \leq \, C(\sigma,\mathrm{d}_\bfz) \rme^{- \mathrm{d}_\bfz \real s}|s|^7\Vert \gamma_t \LEi \Vert_{H^{3/2}(\Gamma)^3}
\end{align*}
that we also found in \eqref{eq:Eppwbound} with $\Vop^{-1}(s)$ instead of $\Vop^{-1}_h(s)$, we can immediately apply \cite[Thm.\@ 3]{BLM11} to bound the first difference on the right hand side of \eqref{eq:P2split}.
Similarly, we have that 
\begin{align*}
\left| \delta_\bfz\Sop(s)\Vop_h^{-1}(s)P_h' \widehat{\bfg} \right| \, \leq \, C(\sigma,\mathrm{d}_\bfz) \rme^{- \mathrm{d}_\bfz \real s}|s|^4\Vert \widehat{\bfg} \Vert_{\HdivO}
\end{align*}
and again, one can apply \cite[Thm.\@ 3]{BLM11} as follows. For $r-3>2m+11/2$ we find that 
\begin{multline*}
\left| \left(\delta_\bfz\Sop(\underline{\partial_t^\tau})\Vop_h^{-1}(\underline{\partial_t^\tau})P_h' L(\partial_t)\gamma_t \Ei - \delta_\bfz\Sop(\partial_t)\Vop_h^{-1}(\partial_t)P_h' L(\partial_t)\gamma_t \Ei  \right)_n\right| \\
\, \leq \, C \tau^{2m-1} \Vert L(\partial_t) \gamma_t \Ei \Vert_{H_0^{r-3}(0,T; \HdivO)} \, \leq \, C \tau^{2m-1} \Vert \gamma_t \Ei \Vert_{H_0^{r}(0,T ; H_t^{3/2}(\Gamma))}\, .
\end{multline*}
In total we get that 
\begin{align}\label{eq:boundP2}
\left| (P_2)_n \right| \, \leq \, C \tau^{2m-1} \Vert \gamma_t \Ei \Vert_{H_0^{r}(0,T ; H_t^{3/2}(\Gamma))}\, .
\end{align}
Finally we can combine the two bounds from \eqref{eq:boundP1} and \eqref{eq:boundP2} to get that
\begin{align*}
\left| (\widetilde{\bfE}'_{h,\tau})_n(\bfz) - (\bfE'_{h,\tau})_n(\bfz) \right| \, \leq \, C(h^{2} + \tau^{2m-1})\Vert \gamma_t \Ei \Vert_{H_0^{r}(0,T ; H^{3/2}(\sdiv,\Gamma))}\,. 
\end{align*}
Combining this with \eqref{eq:bound1} yields the overall bound.
\end{proof}
\section{Numerical experiments}\label{sec:InvScat}
\subsection{Setting and geometries}
Throughout this section we consider an incident electric field $\Ei$ solving Maxwell's equations $c^{-2} \partial_t^2 \Ei + \curl^2 \Ei = 0$ in $\R^3\times (0,\infty)$, which is given by a modulated Gaussian plane wave
\begin{align}\label{eq:EiG}
\Ei(\bfx,t) \, := \, \bfA\rme^{-1/(2\sigma^2)(t - \bfx\cdot \bfd/c - t_{\mathrm{lag}})^2}\cos(2\pi f_0(t - \bfx\cdot \bfd/c)) \, ,
\end{align}
where $\bfA \in \R^3$ is the polarization, $\bfd \in S^2$ is the direction of propagation with $\bfA \cdot \bfd = 0$, $\sigma>0$ is the width of the Gaussian plane wave, $ t_{\mathrm{lag}}>0$ is some time delay and $f_0$ is the center of support of the wave in the frequency domain.
Modulated Gaussian plane waves are often considered in the context of time-dependent scattering problems (see, e.g., the numerical examples in \cite{BalBanSauVeit13, DeAnCo20, LiMonWei15}).
We note that this incident wave is not causal on any bounded scattering object. 
However, with $t_\mathrm{lag}$ chosen properly, we make sure that the size of $|\Ei|$ on the boundaries of all scattering objects under consideration is less than the machine tolerance.

For both our numerical validation of the temporal domain derivative as well as the shape reconstruction considered below, we focus on star-shaped scattering objects, centered at the origin.
We develop the radius $r=r(\alpha_n^m, \beta_n^m)$ of star-shaped objects into real-valued spherical harmonics. 
To be precise, for $\xhat = \bfx/|\bfx|\in S^2$, let the spherical harmonics be given by
\begin{align*}
Y_n^m(\xhat) \, := \, \sqrt{\frac{2n+1}{4\pi} \frac{(n-|m|)!}{(n+|m|)!}}P_n^{|m|}(\cos \theta)\rme^{\rmi m \varphi}\quad \text{for } m=-n,\dots, n \text{ and } n=0,1, \dots \, 
\end{align*}
for $\theta \in [0,\pi]$ and $\varphi \in [0,2\pi)$.
Here, the associated Legendre polynomials are defined by
$P_n^m(t) = (1-t^2)^{m/2}(d/dt)^m P_n(t)$,
where $P_n$ denotes the Legendre polynomial of degree $n$ (see, e.g., \cite[pp.~26]{ColKre19}). Moreover, the functions $Y_n^m$ form a complete orthonormal basis of $L^2(S^2)$ (see, e.g., \cite[p.~41]{KiHe15}).
Abbreviating two finite sequences of numbers by $\mathbf{\alpha}_N = (\alpha_n^m)_{n=0,\dots, N, m=0,\dots,n}$ and $\mathbf{\beta}_N = (\beta_n^m)_{n=1,\dots, N, m=1,\dots,n}$ we write for the radius of the star-shaped object
\begin{align}\label{eq:rRep}
r(\alpha_N,\beta_N; \xhat) \, = \, \sum_{n=0}^N\sum_{m=0}^n\alpha_n^m \real Y_n^m(\xhat) + \sum_{n=1}^N\sum_{m=1}^n\beta_n^m \imag Y_n^m(\xhat)\, .
\end{align}
Any star-shaped object's boundary, whose radius $r$ is parametrized by the finite sequences $\alpha_N$ and $\beta_N$ as in \eqref{eq:rRep}, is denoted by $\Gamma_{(\alpha_N,\beta_N)}$, i.e.,
\begin{align*}
\Gamma_{(\alpha_N,\beta_N)} \, := \, \{r(\alpha_N,\beta_N; \xhat) \xhat \, : \, \xhat \in S^2 \}\, 
\end{align*}
and we write $\Es_{h,\tau}[\Gamma_{(\alpha_N,\beta_N)}]$ to indicate the time and space discretized scattered electric field that corresponds to the scatterer $\Gamma_{(\alpha_N,\beta_N)}$. 
Moreover, we write $\Et'_{h,\tau}[\Gamma_{(\alpha_N,\beta_N)}](\alpha_N',\beta_N')$ for the time and space discretized domain derivative that corresponds to the scatterer $\Gamma_{(\alpha_N,\beta_N)}$ with a perturbation $\bfh = r(\alpha_N',\beta_N';\xhat)\xhat$ with $r(\alpha_N',\beta_N';\xhat)$ computed as in \eqref{eq:rRep} with coefficients defined through the finite sequences $\alpha_N'$ and $\beta_N'$.
More precisely, the time and space discrete quantities are obtained via a 3-stage Runge--Kutta convolution quadrature in time and a Galerkin discretization in space.
For our numerical treatment of boundary integral equations we employ the boundary element method Python package \textit{bempp} (see \cite{Betcke2021} and https://bempp.com).
The just-in-time compilation of \textit{bempp} is used to evaluate time-dependent functions on the boundary efficiently.
We use the all-at-once formulation for convolution quadrature from \cite{BanSay22}. 
This transforms time-domain problems to several, decoupled frequency domain problems, which can be solved in parallel.
To solve those problems, we apply the LU-decomposition to the system matrices. 
As we see from the definition of the domain derivative in \eqref{eq:ddmwtd}, different perturbations $\bfh$ only affect the right hand side. Therefore, once all $s$-dependent LU-tuples are available, even solving a massive number of linear systems, which will be needed in order to assemble Jacobian-type matrices, becomes feasible.
The spatial implementation of both $\Es_{h,\tau}[\Gamma_{(\alpha_N,\beta_N)}]$ and $\Et'_{h,\tau}[\Gamma_{(\alpha_N,\beta_N)}]$ is done by using Raviart--Thomas elements as trial functions and their rotated counterparts as test functions on the boundary of a polyhedron within the \textit{bempp} framework. 
We note that these meshes are only Lipschitz and do not meet the high spatial regularity requirements that are needed to obtain the full spatial convergence of the domain derivative from Section~\ref{sec:FE} (see also \cite{BalBanSauVeit13} for corresponding studies regarding the scattered field). 
The following examples were performed on the Roihu supercomputer provided by CSC – IT Center for Science.
\subsection{Numerical verification of the domain derivative}\label{sec:numtests}
The purpose of this subsection is to verify the property \eqref{eq:FrechRep} as well as Theorem~\ref{thm:main} numerically.
For both cases we focus on a star-shaped scattering object with radius given as in \eqref{eq:rRep} with $N=3$ and coefficients given by 
\begin{align*}
\alpha_0^0 \, = \, 1.3\times \sqrt{4\pi}\, ,\quad \alpha_3^1 \, = \, 1.3\, ,\quad \alpha_3^3 \, = \, -0.65\, , \quad \alpha_n^m \, = \, 0 \; \text{for } (m,n) &\notin \{(0,0),\; (1,3),\; (3,3)\} \, , \\
\beta_2^1 \, = \, 0.3\, ,\quad \beta_3^2 \, = \, -0.2\, , \quad \beta_n^m \, = \, 0 \; \text{for } (m,n) &\notin \{(1,2),\; (2,3)\} \, . 
\end{align*}
As an incident wave we pick a modulated Gaussian plane wave as in \eqref{eq:EiG} with polarization $\bfA = [0,1,1]^\top$, incident direction $\bfd = [1,0,0]^\top$,  a pulse width of $\sigma = 20\mathrm{ns}$, a time delay $t_{\mathrm{lag}} = 200\mathrm{ns}$ and center frequency $f_0 = 40\mathrm{MHz}$.
We pick $M=5$ receivers located in the backscattering direction, to the left, right, top and bottom and 6$\mathrm{m}$ away from the origin and consider the final time $T=600\mathrm{ns}$.
A visualization of this scattering object together with the receiver position is found in the left plot of Figure~\ref{fig:geom}.
The $z$-component of the incident wave $\Ei$ as a function of the $x$-component at time $t=200\mathrm{ns}$ is found in the left plot of Figure \ref{fig:Eivis}. 
In the right plot of Figure~\ref{fig:Eivis} one finds the normalized absolute values of the Fourier transform of $\Ei_z$.

\begin{figure}[t!]
\centering 
\includegraphics[scale=.28]{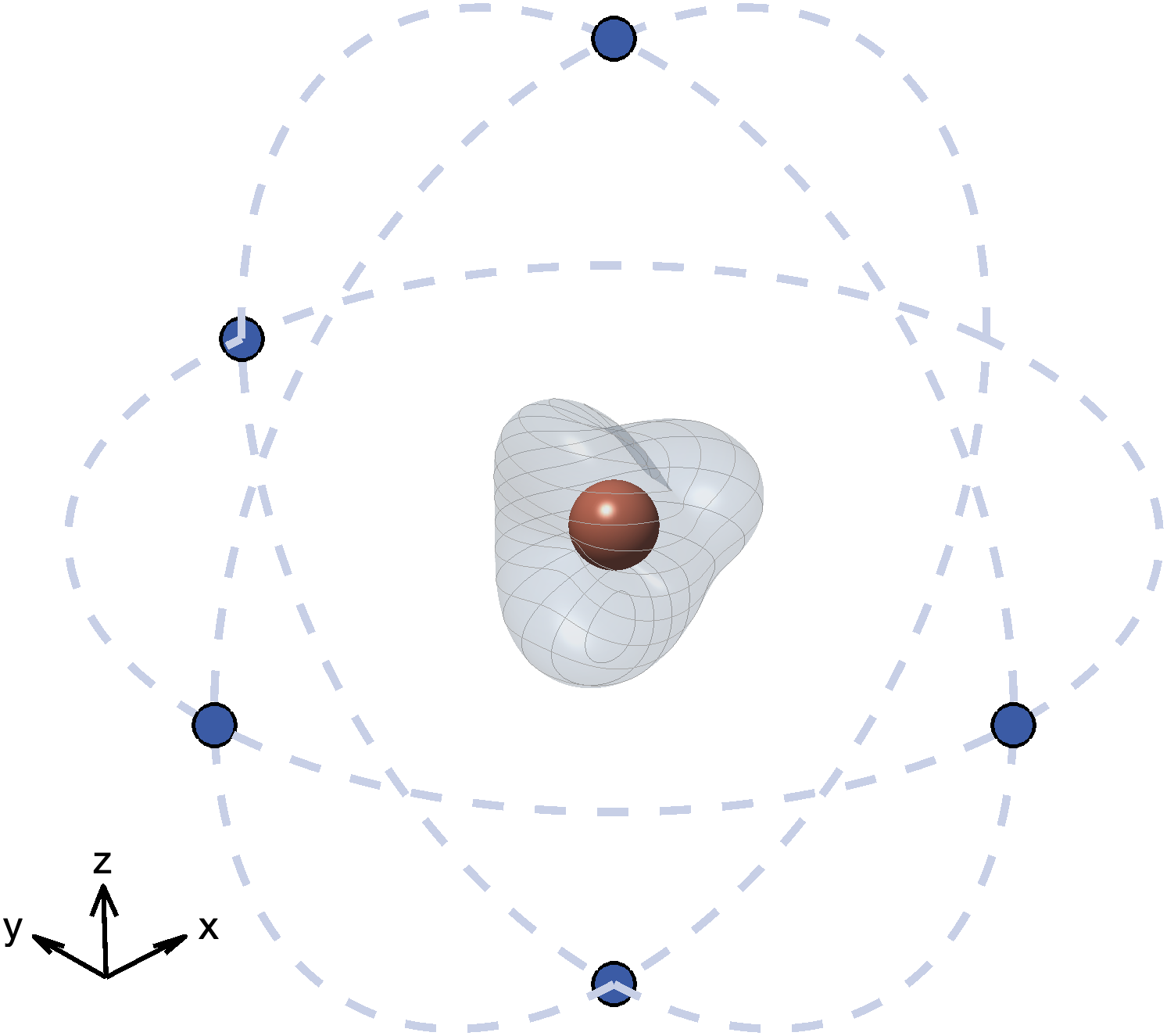} \hfill
\includegraphics[scale=.28]{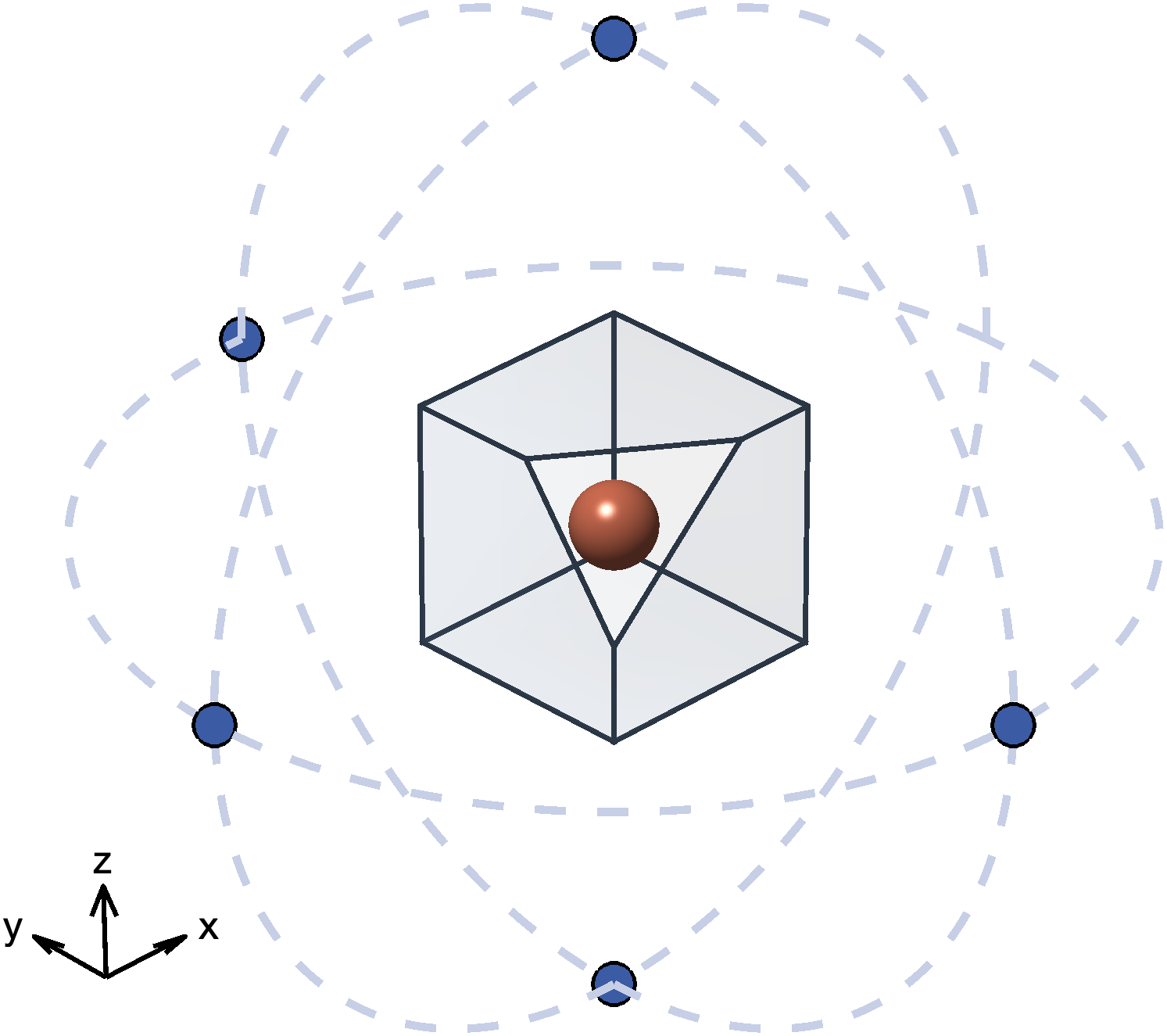}
\caption{Geometrical setups. The blue dots indicate the receiver positions, which are located $6\mathrm{m}$ away from the origin. The orange-colored initial guess of the Gau\ss--Newton scheme considered for the shape reconstruction of Section~\ref{sec:shaperec} is placed at the origin. Arrows of the coordinate crosses point towards the unit vectors in $\R^3$.}
\label{fig:geom}
\end{figure}

 \begin{figure}[t!]
\centering 
\includegraphics[scale=.28]{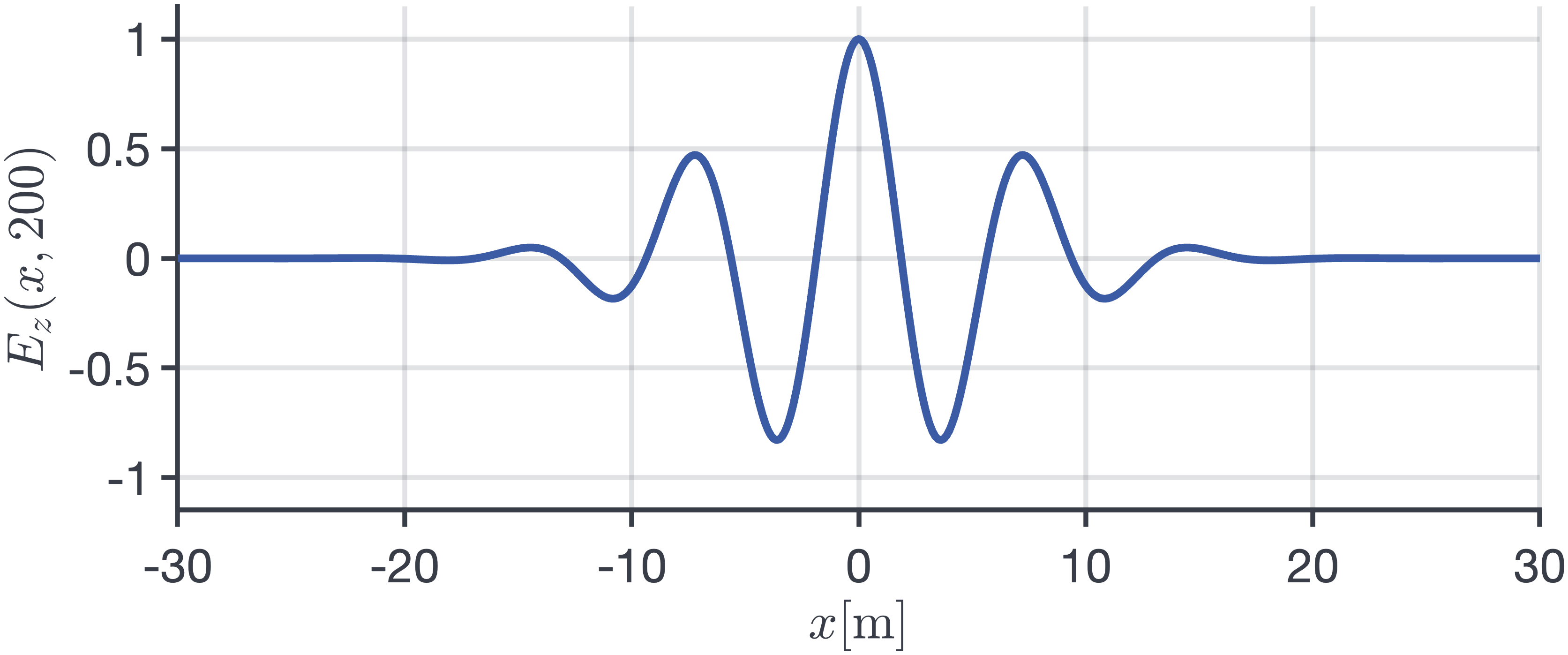}
\includegraphics[scale=.28]{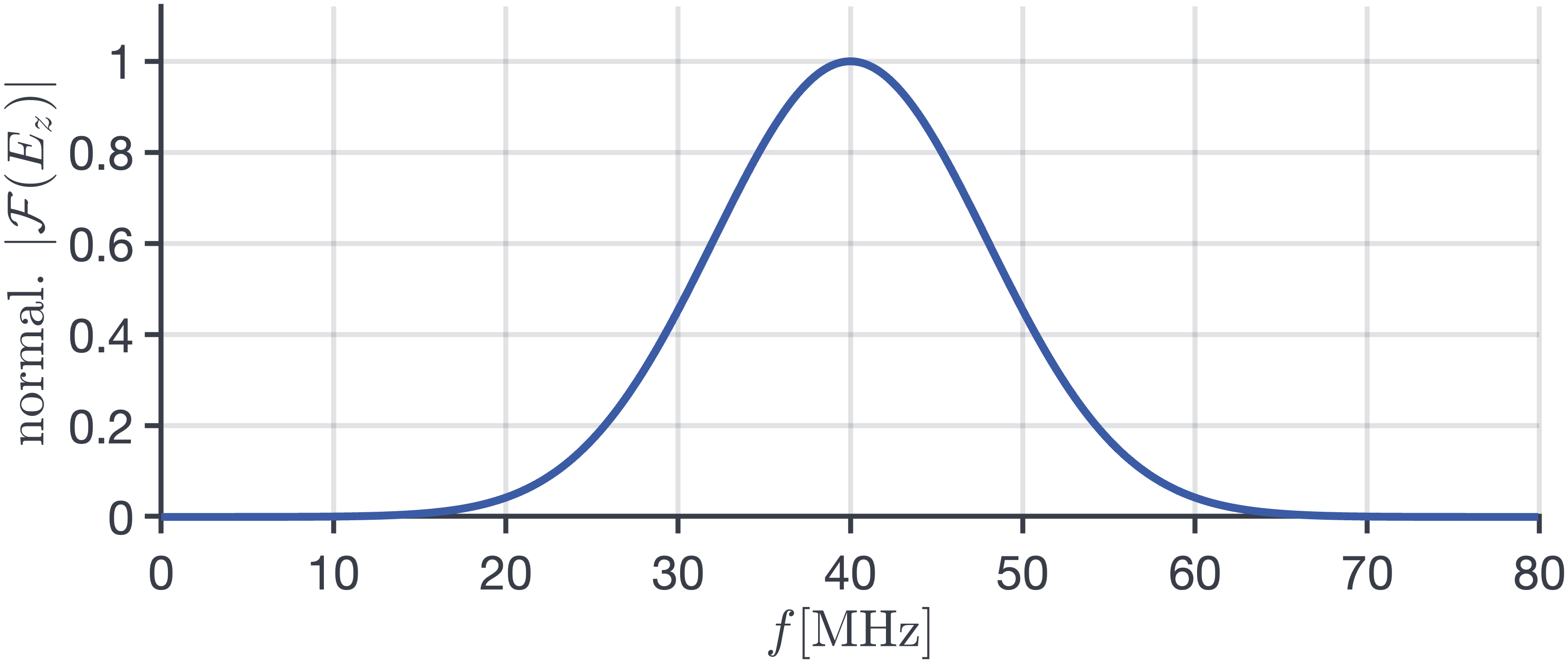}
\caption{The $z$-component of the incident wave $\Ei$ defined as in \eqref{eq:EiG} with coefficients as defined in Section~\ref{sec:numtests} at the time $t=200\mathrm{ns}$ as a function over the $x$-component. Left: The wave in physical space. Right: Normalized absolute values of the Fourier transform of $\Ei_z$. }
\label{fig:Eivis}
\end{figure}

In order to verify \eqref{eq:FrechRep} the aim is to check numerically that 
\begin{align*}
\left| (\Es_\bfh - \Es - \Et_\bfh')(\bfz,t) \right| \, = \, O(\Vert \bfh \Vert_{C^1}^2) \quad \text{as } \Vert \bfh \Vert_{C^1} \to 0\, .
\end{align*}
for some $\bfz \in \Omega$ and $t>0$. Here, we wrote $\Et_\bfh'$ for the domain derivative $\Et'$ from \eqref{eq:ddmwtd} to indicate its $\bfh$-dependence. Regarding the fact that neither $\Es_\bfh, \Es$ nor $\Et_\bfh'$ are explicitly available, what we compute numerically is the quantity
\begin{align}\label{eq:Deltaeps}
\Delta(\varepsilon) \, := \,  \left\Vert \Es_{h,\tau}[\Gamma_{(\alpha_3+\varepsilon \alpha_3',\beta_3+\varepsilon \beta_3')}] - \Es_{h,\tau}[\Gamma_{(\alpha_3,\beta_3)}] - \Et'_{h,\tau}[\Gamma_{(\alpha_3,\beta_3)}](\varepsilon \alpha_3',\varepsilon\beta_3') \right\Vert_{\ell_\tau^2(0,T; \R^{3M})}\, , 
\end{align}
for two random coefficients $\alpha_3'$ and $\beta_3'$ and several small values of $\varepsilon$.  The norm on the right hand side is defined in \eqref{def:l2norm}. The random coefficients are normalized in such a way that the corresponding perturbation $\bfh$ as in \eqref{eq:rRep} has a maximal deviation of $1\mathrm{m}$.
For the simulation of the quantity \eqref{eq:Deltaeps}, we use a spatial discretization of the scatterer involving 12288 degrees of freedom and $400$ time steps. The result is shown in the left plot of Figure~\ref{fig:sec2}. The black reference line has slope 2. We find that $\Delta(\varepsilon)$ decays like $\varepsilon^2$ as $\varepsilon \to 0$, which is exactly what we wanted to verify. 
\begin{figure}[t]
  \sbox\ordertable{\small
    \begin{tabular}{@{}cc@{}}
      \toprule
      $(n_1,n_2,n_3)$ & $p$ \\
      \midrule
      \midrule
      $(6,8,10)$   & $1.87$ \\
      $(8,10,12)$  & $1.90$ \\
      $(10,12,14)$ & $1.94$ \\
      $(12,14,16)$ & $1.96$ \\
      $(14,16,18)$ & $1.97$ \\
      $(16,18,20)$ & $1.98$ \\
      $(18,20,22)$ & $1.98$ \\
      $(20,22,24)$ & $1.98$ \\
      \bottomrule
    \end{tabular}}
  \setlength\panelheight{\dimexpr\ht\ordertable+\dp\ordertable\relax}
  \typeout{Panel height: \the\panelheight}
  \begin{minipage}[t]{0.37\textwidth}
    \centering
    \adjustbox{valign=t}{\includegraphics[width=\linewidth,height=\panelheight,keepaspectratio]{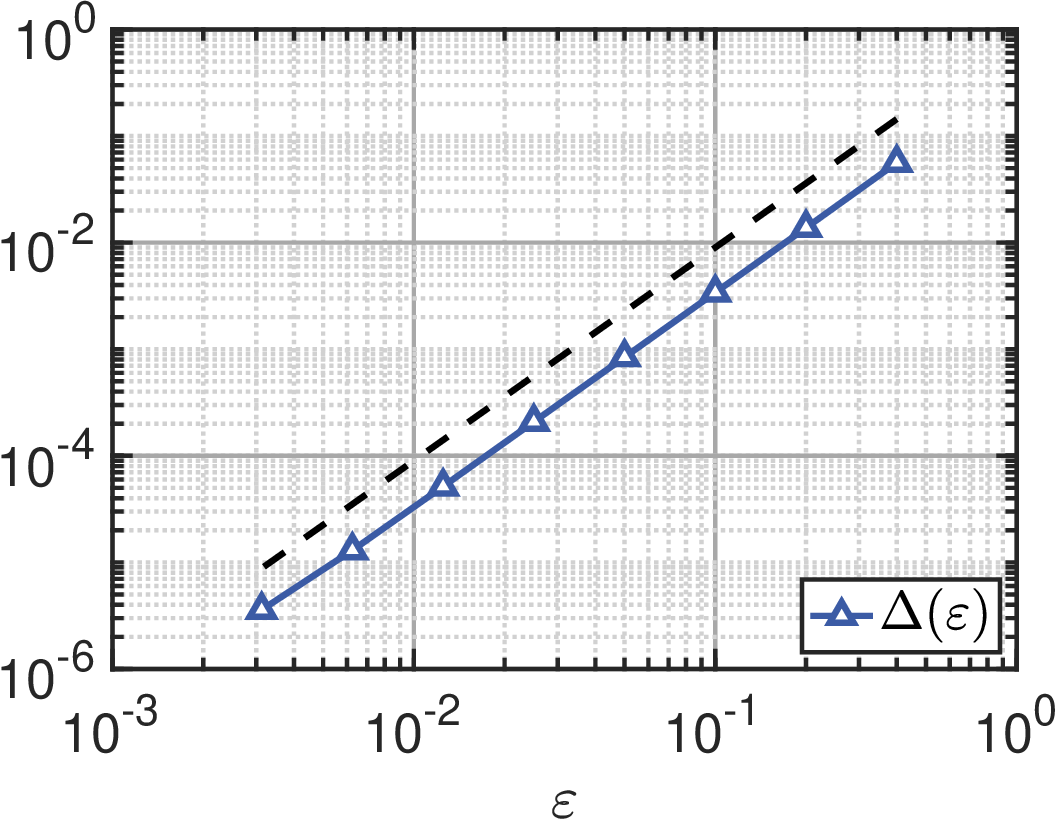}}
  \end{minipage}\hfill
  \begin{minipage}[t]{0.37\textwidth}
    \centering
    \adjustbox{valign=t}{\includegraphics[width=\linewidth,height=\panelheight,keepaspectratio]{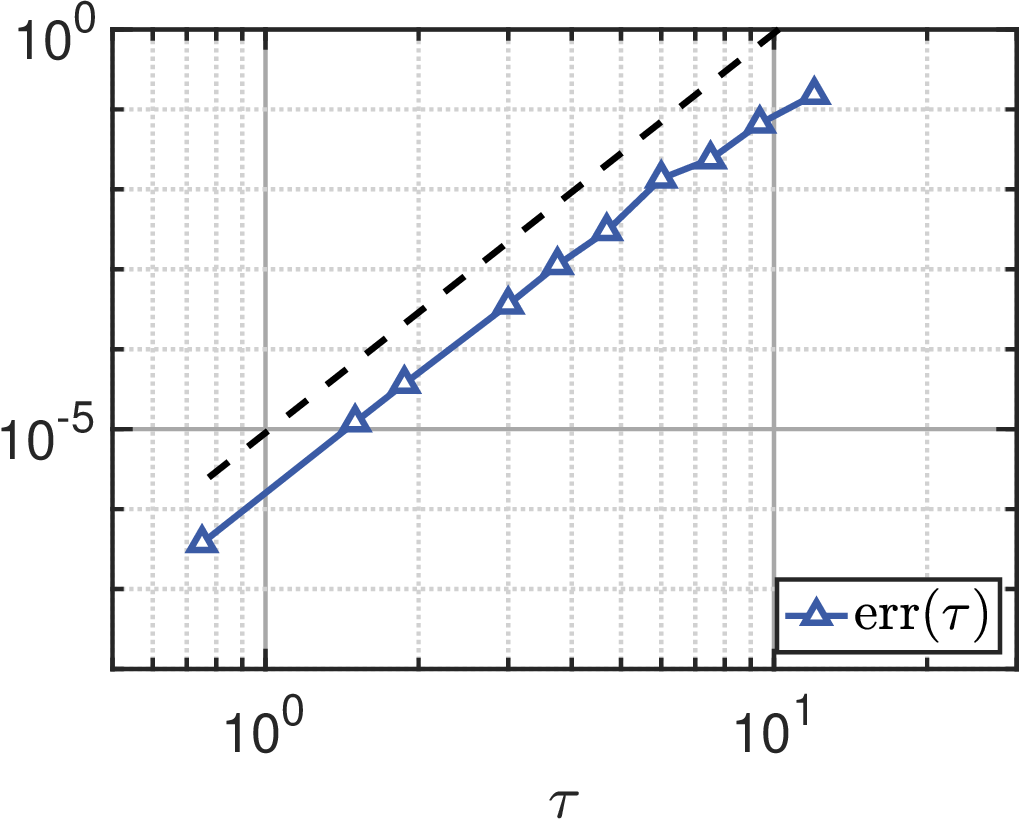}}
  \end{minipage}\hfill
  \begin{minipage}[t]{0.25\textwidth}
    \centering
    \adjustbox{valign=t}{\usebox\ordertable}
  \end{minipage}
\caption{Left: Decay of $\Delta$ from \eqref{eq:Deltaeps}. The dashed line has slope 2. Middle: Decay of $\mathrm{err}$ from \eqref{eq:errtau}. The dashed line has slope 5.
Right: Some mesh refinement triples with numerically estimated order $p$.}
\label{fig:sec2}
\end{figure}
To verify Theorem~\ref{thm:main} we consider $\Gamma_{(\alpha_3,\beta_3)}$, the same receiver positions and a perturbation generated via random coefficients as before. 
We first consider the time discretization. For this purpose, we generate a reference solution approximating $\Et'$ corresponding to 12288 degrees of freedom and $N_{t,\mathrm{ref}}=1600$ time steps. 
We denote this term  by $\Et'_{\mathrm{ref}}[\Gamma_{(\alpha_3,\beta_3)}](\alpha_3',\beta_3')$ (even though it is a numerical approximation itself). Now, we compare this to an approximation to $\Et'$ involving the same number of spatial degrees of freedom but with considerably less time steps $N_t<N_{t,\mathrm{ref}}$, namely those, that divide $N_{t,\mathrm{ref}}$ evenly and that are greater equal to 50. We denote this approximation by $\Et'_{h,\tau}[\Gamma_{(\alpha_3,\beta_3)}](\alpha_3',\beta_3')$. Now we compute
\begin{align}\label{eq:errtau}
\mathrm{err}(\tau) \, := \, \frac{\max\limits_{j=1,\dots,5}\max\limits_{n\leq N_t} \left|\Et'_{\mathrm{ref}}[\Gamma_{(\alpha_3,\beta_3)}](\alpha_3',\beta_3')(\bfz_j,t^{(n)}) - (\Et'_{h,\tau}[\Gamma_{(\alpha_3,\beta_3)}](\alpha_3',\beta_3')(\bfz_j))_n \right|}{\max\limits_{j=1,\dots,5}\max\limits_{n\leq N_\mathrm{t,ref}} \left|\Et'_{\mathrm{ref}}[\Gamma_{(\alpha_3,\beta_3)}](\alpha_3',\beta_3')(\bfz_j,t_n) \right|} \, , 
\end{align}
where $t^{(n)} = nT/N_t = t_n(N_\mathrm{t,ref}/N_t)$. 
 The result is shown in the middle plot of Figure~\ref{fig:sec2}. The black reference line has slope 5. We find that $\mathrm{err}(\tau)$ decays as $\tau^5$ in the asymptotic regime, which is what we expect  from Theorem~\ref{thm:main} for a 3-stage Runge--Kutta convolution quadrature method.
 
 Finally, we consider the spatial discretization. To generate decreasingly finer triangulations of $\Gamma_{(\alpha_3,\beta_3)}$ we start with a diamond shaped polyhedron with vertices $\pm \bfe_i$ for $i=1,2,3$. Each of the 8 triangles is refined into $n^2$ finer triangles for $n\in \N$ by using a uniform subdivision. Subsequently the newly obtained points are projected along the ray connecting the origin and the respective point onto $\Gamma_{(\alpha_3,\beta_3)}$. 
 Since the reference solution with 12288 degrees of freedom is only moderately finer than the meshes we aim to compare it against, its own error would distort orders computed against it. Therefore, we estimate the order from three consecutive meshes as follows.
 For $n_1<n_2<n_3$ with $n_k \in \N$, $k=1,2,3$, let $h_1>h_2>h_3$ denote the corresponding mesh size. 
 For 1600 time steps and the same randomly generated perturbation as considered before, we compute $(\Et'_{h_k,\tau}[\Gamma_{(\alpha_3,\beta_3)}](\alpha_3',\beta_3')(\bfz_j))_n$ for $k=1,2,3$, which we write as $\Et_{h_k}'$ for abbreviation. 
 For $h_k$ sufficiently small, convergence with order $p$ means that $\Et_{h_k}' = \mathbf{X} + \mathbf{C} h_k^p + o(h_k^p)$ with vector valued quantities $\mathbf{X}$, $\mathbf{C}$. Therefore,
 \begin{align}\label{eq:psearch}
 \frac{\Vert \Et_{h_3}' - \Et_{h_2}' \Vert}{\Vert \Et_{h_2}' - \Et_{h_1}' \Vert} \, = \, \frac{h_3^p-h_2^p}{h_2^p-h_1^p} \, 
 \end{align}
 for any norm $\Vert \cdot \Vert$. One can now apply a numerical method to solve \eqref{eq:psearch} for $p$. For different values for $n_1,n_2,n_3$ we report the approximate order in the right table of Figure~\ref{fig:sec2}. For sufficiently high values for $n_1,n_2,n_3$ the estimated order $p$ approaches 2, which is the rate from Theorem~\ref{thm:main}. 
We stress however again that the theorem does not cover the flat-triangle meshes.

\subsection{Shape reconstructions using time-dependent near field data}\label{sec:shaperec}
In this subsection we test the temporal domain derivative in a shape reconstruction algorithm.
For a star-shaped object we suppose that measurements of the scattered wave $\Es$ solving \eqref{eq:MWEs} with $\Ei$ as in \eqref{eq:EiG} at the receivers $\mathbf{z}_j$, $j=1,\dots,M$ are available and that they are given for $N_t \in \N$ equidistant time steps. We denote them by
$\mathbf{g} \in  \ell_\tau^2(0,T; \R^{3M})$.

We formulate the inverse problem as a minimization problem. To be precise, the aim is to compute
$(\alpha_N^*,\beta_N^*) \, = \, \argmin_{(\alpha_N,\beta_N)} f(\alpha_N,\beta_N)$, where
\begin{align*}
 f({\alpha_N,\beta_N}) \, = \, \frac{\left\Vert (\Es_{h,\tau}[\Gamma_{(\alpha_N,\beta_N)}])_{j=1,\dots ,M} - \mathbf{g} \right\Vert_{\ell_\tau^2(0,T; \R^{3M})}^2}{\left\Vert \mathbf{g} \right\Vert_{\ell_\tau^2(0,T; \R^{3M})}^2} \,  
\end{align*}
and where $(\Es_{h,\tau}[\Gamma_{(\alpha_N,\beta_N)}])_{j=1,\dots M}$ denotes the scattered electric field corresponding to the scattering object $\Gamma_{(\alpha_N,\beta_N)}$, which is approximated by a 3-stage Radau IIA Runge--Kutta convolution quadrature method in time and a Galerkin discretization in space and evaluated at $\bfz_j$, $j=1,\dots,M$.
To stabilize the optimization we introduce a regularization term $\psi$, that aims at penalizing high oscillations of the boundary of the object (see, e.g., \cite{Hag19,Kirsch93}). Finally, we minimize the regularized function
\begin{align}\label{eq:outputfun}
\gamma(\alpha_N,\beta_N) = f(\alpha_N,\beta_N) + \alpha_{\mathrm{reg}} \psi(\alpha_N,\beta_N)
\end{align}
for some regularization parameter $\alpha_{\mathrm{reg}}>0$ using the Gau\ss--Newton algorithm.
In our algorithm we pick $\alpha_{\mathrm{reg}}$ by trial and error. 
If the size of the Gau\ss--Newton update falls below a certain threshold and the functional in \eqref{eq:outputfun} is dominated by the regularization term we restart the algorithm with the updated regularization parameter $\alpha_{\mathrm{reg}}/100$. The algorithm terminates if the update size falls below a threshold and the functional \eqref{eq:outputfun} is dominated by $f$ or, when $f$ itself falls below 0.01.

\textbf{Example 1.}
In our first numerical example the true scattering object is a non-convex star-shaped object, whose radius is described through \eqref{eq:rRep} with coefficients described in Section~\ref{sec:numtests}.
Moreover, the parameters for the incident wave as well as the receiver positions are defined as in Section~\ref{sec:numtests}.
This object is found in the left hand side of Figure~\ref{fig:geom}.
For the simulation of the measurement data, we use a spatial discretization of the scatterer involving 12288 degrees of freedom and $400$ time steps.

For the inverse problem we start with a ball centered at the origin with the radius $0.5\mathrm{m}$. 
We use a spatial discretization of star-shaped scatterers with 3072 degrees of freedom and a temporal discretization involving $N_t = 400$ time steps.
Furthermore, we pick $N=10$ as the maximal degree of spherical harmonics in \eqref{eq:rRep}. Therefore, in every iteration 121 different perturbations are considered in order to assemble the Jacobian. 
The regularization parameter in \eqref{eq:outputfun} is chosen to be $\alpha_{\mathrm{reg}}=0.01$. 
 \begin{figure}[t!]
\centering 
\includegraphics[scale=.24]{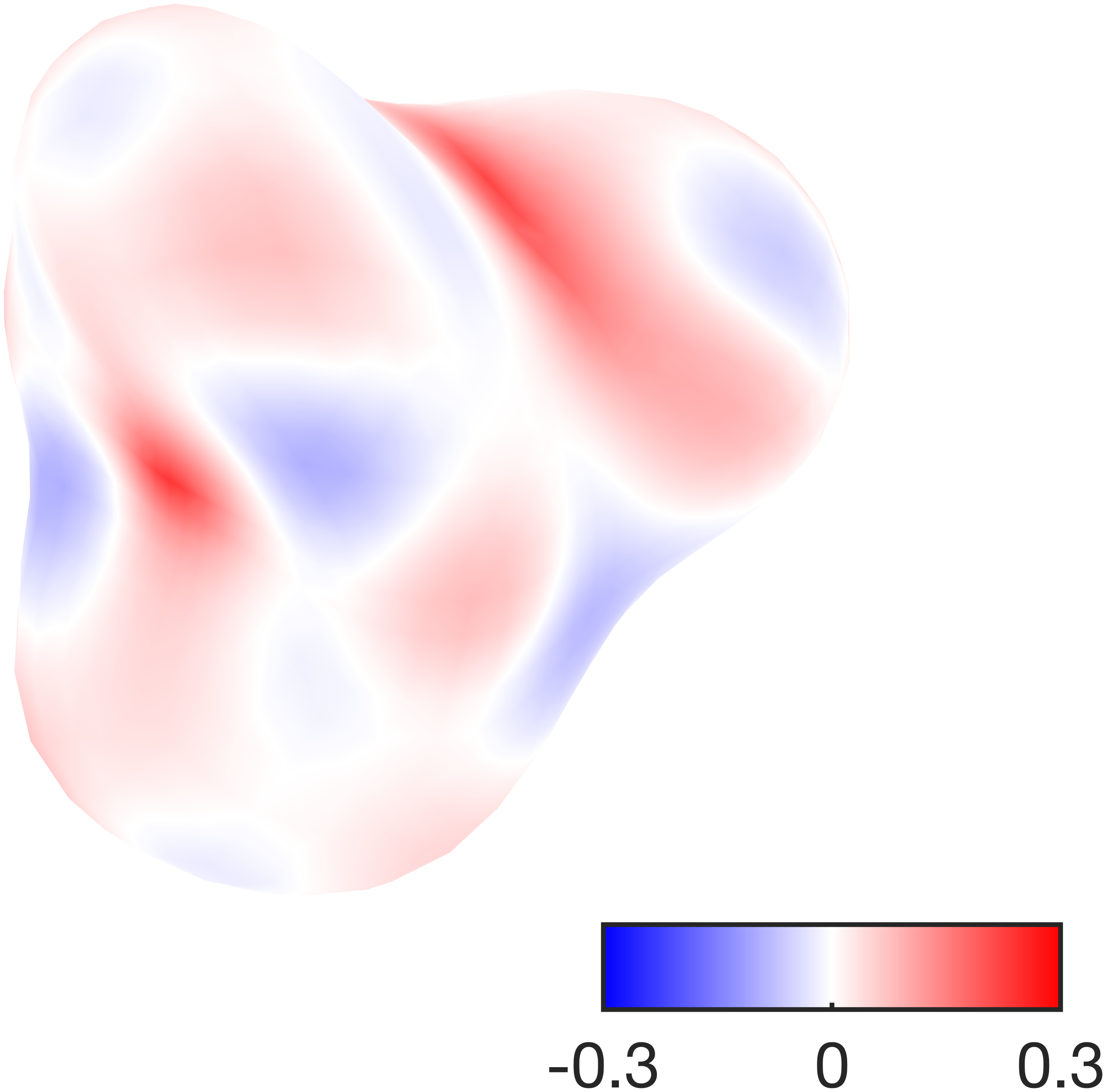} \hfill
\includegraphics[scale=.24]{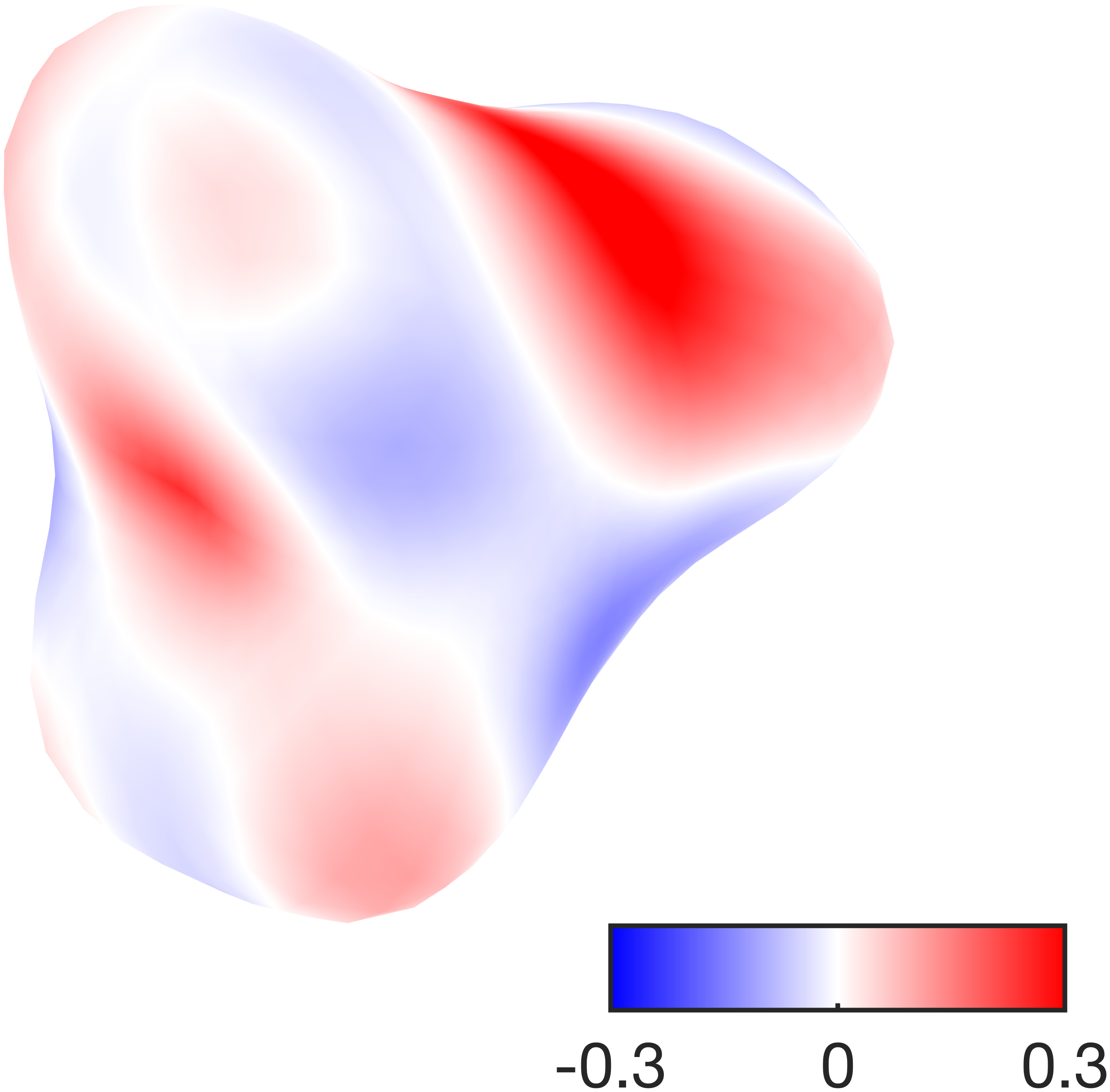} \hfill
\includegraphics[scale=.24]{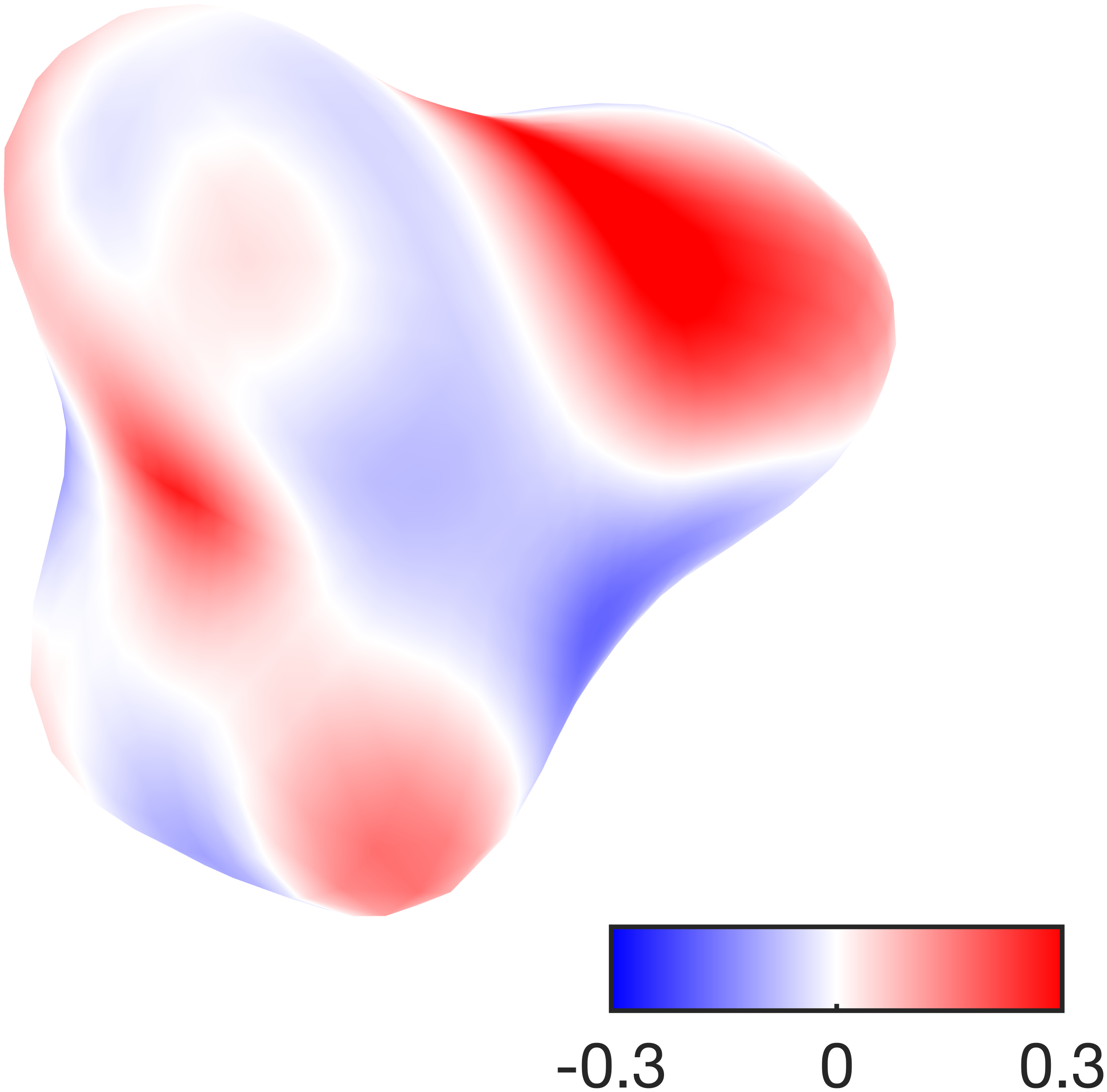}
\caption{Reconstruction of the scatterer from Example 1. The plots show the final iteration of the Newton scheme corresponding to noiseless data (left) and noisy data with $15\%$ additional noise (middle) and $30\%$ additional noise (right).}
\label{fig:rec1}
\end{figure}

In the left plot of Figure~\ref{fig:rec1} the quality of the reconstruction is visualized: 
The object shows the final iteration of the Gau\ss--Newton scheme that was obtained after 13 iterations.
The colors visualize the signed distance of the true object to the reconstruction in meters.
Red colored sections visualize segments of the reconstruction, which exceed the true scattering object's geometry, while blue sections show regions, where the reconstruction remains below the true geometry.
The reconstruction captures the overall shape of the scattering object nicely. It is off from the true object mostly in parts that lie in the shadow of the incoming wave.

Next, we perform shape reconstructions from noisy data.
For this purpose we artificially pollute the given data $\bfg$ by a randomly generated uniformly distributed noise.
The relative noise level is chosen to be $15\%$ and $30\%$, respectively.
The final reconstructions of the Newton scheme given these noisy data are found in the middle and in the right plot of Figure~\ref{fig:rec1}. The iterative scheme stops after 12 and 13 steps, respectively.
While the overall reconstruction is still a decent approximation to the unknown scatterer's shape, it is worse than in the noiseless case. The reconstruction parts that lie in the shadow of the object significantly worsened.

\textbf{Example 2.}
We study the situation, in which the exact scattering object is a cube with one sawn-off corner that is aligned with the coordinate axes and centered at the origin with side length 3$\mathrm{m}$. This object is found in the right hand side of Figure~\ref{fig:geom}.

We pick $M=5$ receivers located in the backscattering direction, to the left, right, top and bottom and 6$\mathrm{m}$ away from the origin and measure the scattered electric field up to the final time $T=600\mathrm{ns}$.
 \begin{figure}[t!]
\centering 
\includegraphics[scale=.2]{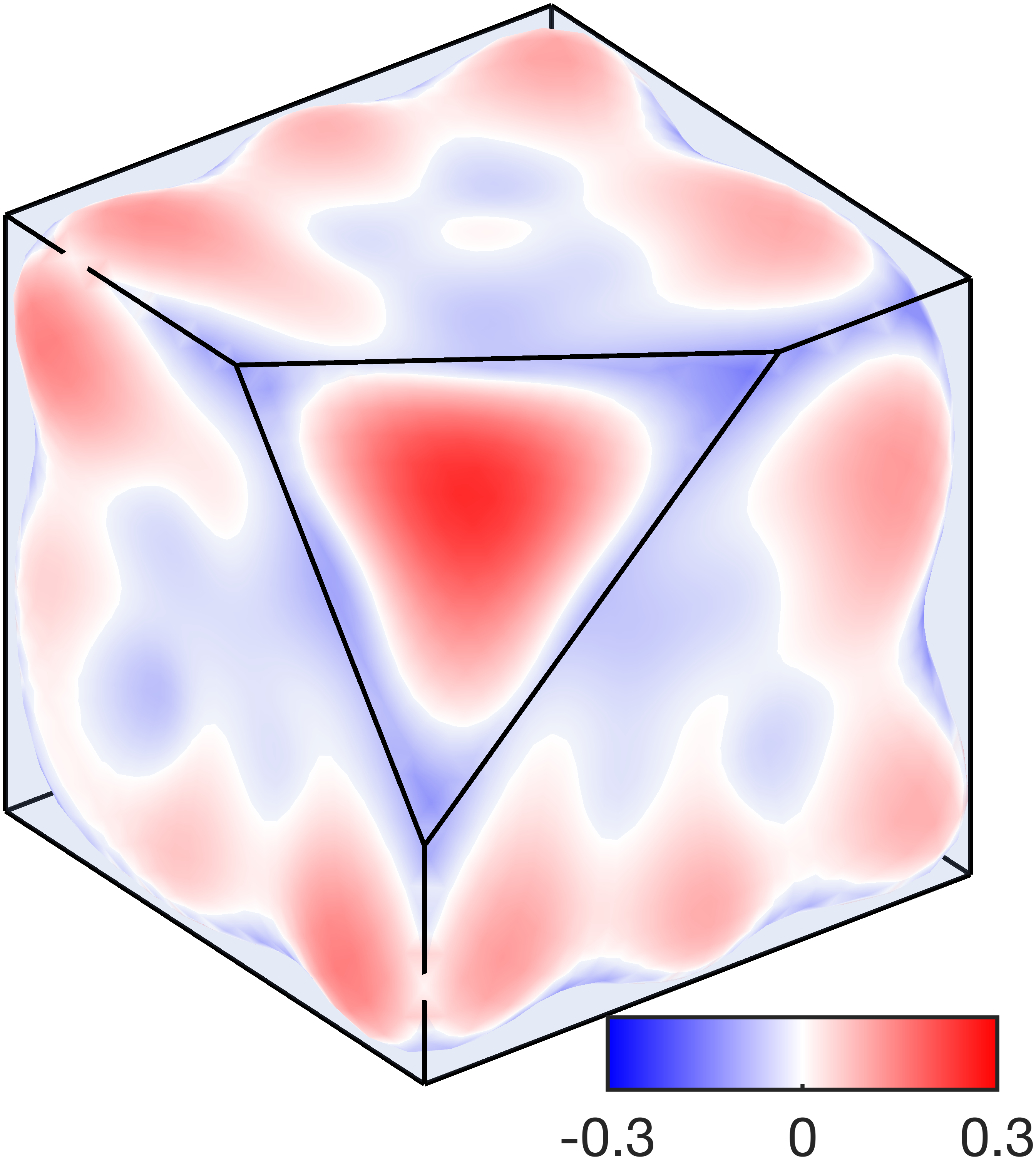} \hfill
\includegraphics[scale=.2]{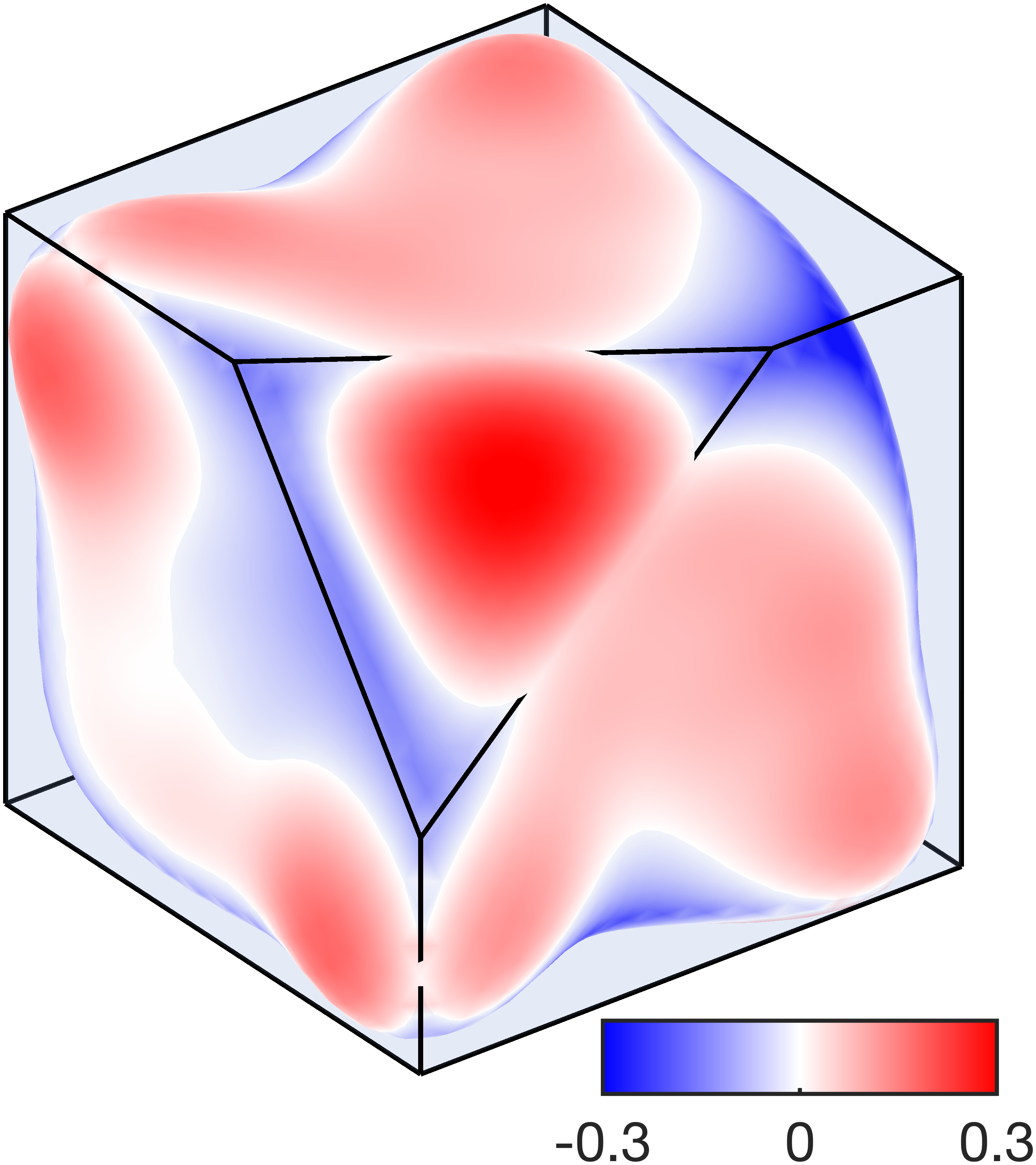} \hfill
\includegraphics[scale=.2]{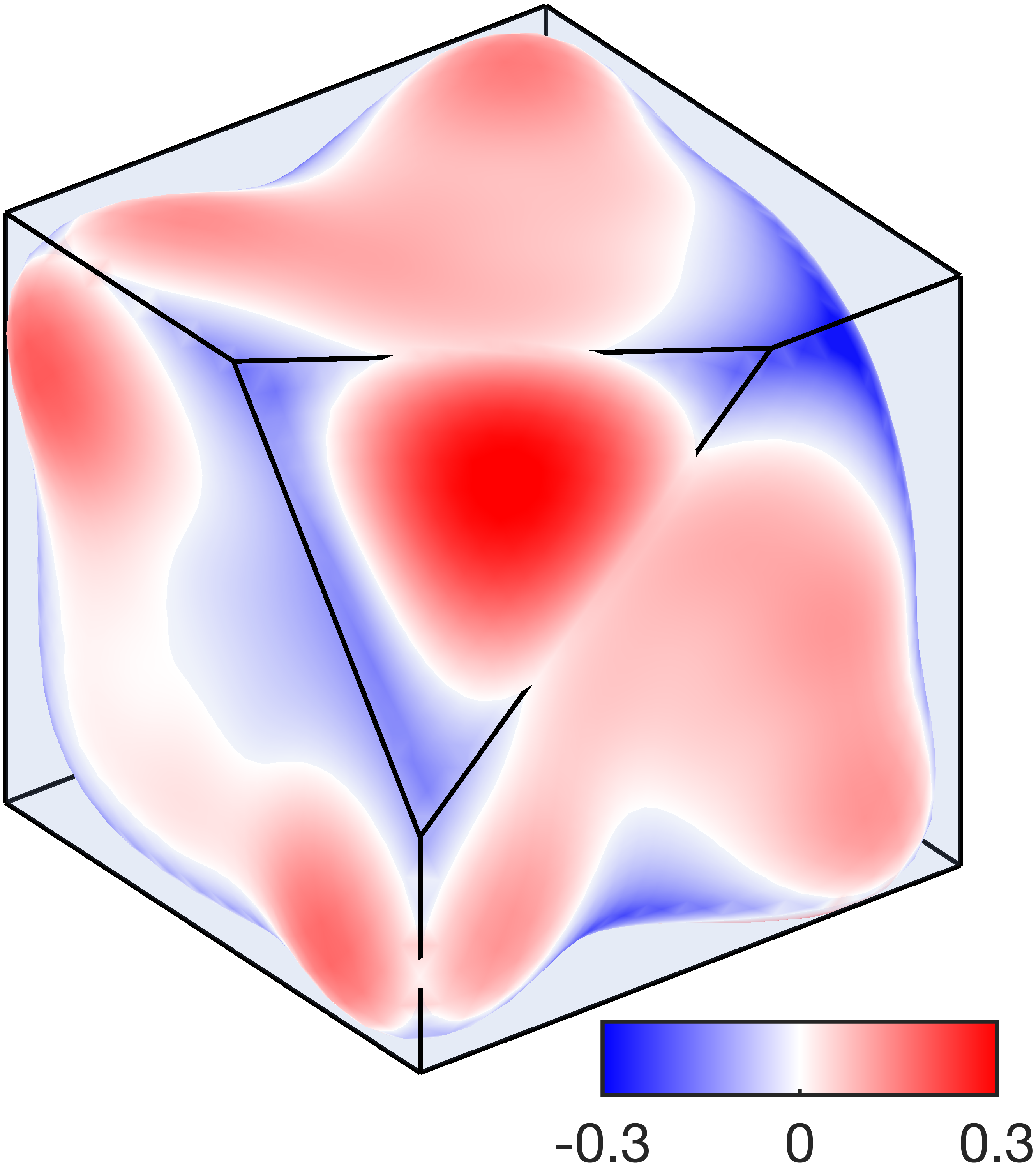}
\caption{Reconstruction of the scatterer from Example 2. The plots show the final iteration of the Newton scheme corresponding to noiseless data (left) and noisy data with $15\%$ additional noise (middle) and $30\%$ additional noise (right).}
\label{fig:rec2}
\end{figure}

In the left plot of Figure~\ref{fig:rec2} the quality of the reconstruction, obtained after 14 Newton steps, is visualized with the same convention on color coding as in Example 1.
The black-lined frame marks the outline of the true scattering object, while the colored object is the final iterate of the Gau\ss--Newton reconstruction. 
We find that the overall reconstruction is decent, however, in the region of the missing corner, the reconstruction is off.
This might be related to the combination of the incident direction and the receiver placement; a related study is conducted in Example 3.
Finally, we also perform shape reconstructions from noisy data. As before, we artificially pollute the data by a randomly generated uniformly distributed noise with a relative noise level of $15\%$ and $30\%$, respectively.
The final reconstructions of the Newton scheme given these noisy data, obtained after 12 and 15 steps, respectively, are found in the middle and in the right plot of Figure~\ref{fig:rec2}.
While the overall reconstruction is still a decent approximation to the unknown scatterer's shape, it is worse than in the noiseless case.
We stress that the automatic correction of the regularization parameter does not decrease the parameter as far as in the noiseless case. 
Improvements on the reconstruction might be obtained by decreasing the regularization parameter further, by choosing a more sophisticated rule for the regularization parameter, or, by using a mesh generation that is more flexible in general. 

\textbf{Example 3.} 
We consider the same scattering object as in Example 2 and treat the same measurement setup as before, i.e., the one depicted in the right hand side of Figure~\ref{fig:geom}. We pick the same parameters for the incident wave, except that we change the incident direction to be directed towards the cut-off corner, i.e., we pick $\bfd = 1/\sqrt{3}[1, 1, -1]^\top$. Note that still $\bfA \cdot \bfd = 0$, so $\Ei$ is still a valid choice for an incident wave for our setup.
For the reconstruction, we take the same parameters and overall setup as in Example 2.

The reconstruction with noiseless data and with noisy data with $15\%$ and $30\%$ relative noise is found in Figure~\ref{fig:rec3}. The Newton scheme stopped after 13, 11 and 12 steps, respectively.
In comparison to the results in Figure~\ref{fig:rec2} the reconstruction underestimates now on the section, where the corner is missing, while the overall shape reconstruction is still decent.
Different measurement setups might improve the reconstruction further. In particular, the reconstruction might profit from using backscattering data.

 \begin{figure}[t!]
\centering 
\includegraphics[scale=.2]{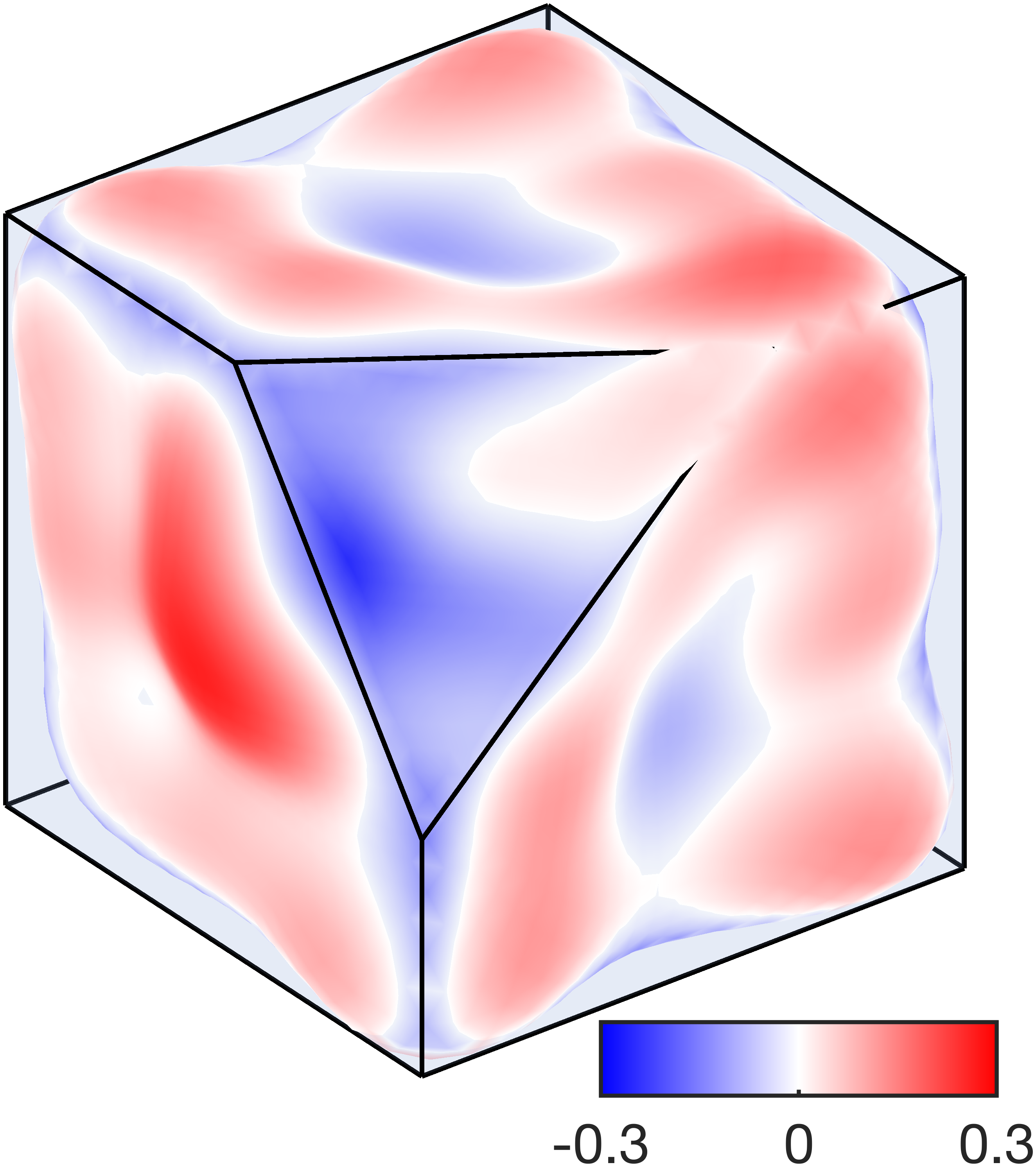} \hfill
\includegraphics[scale=.2]{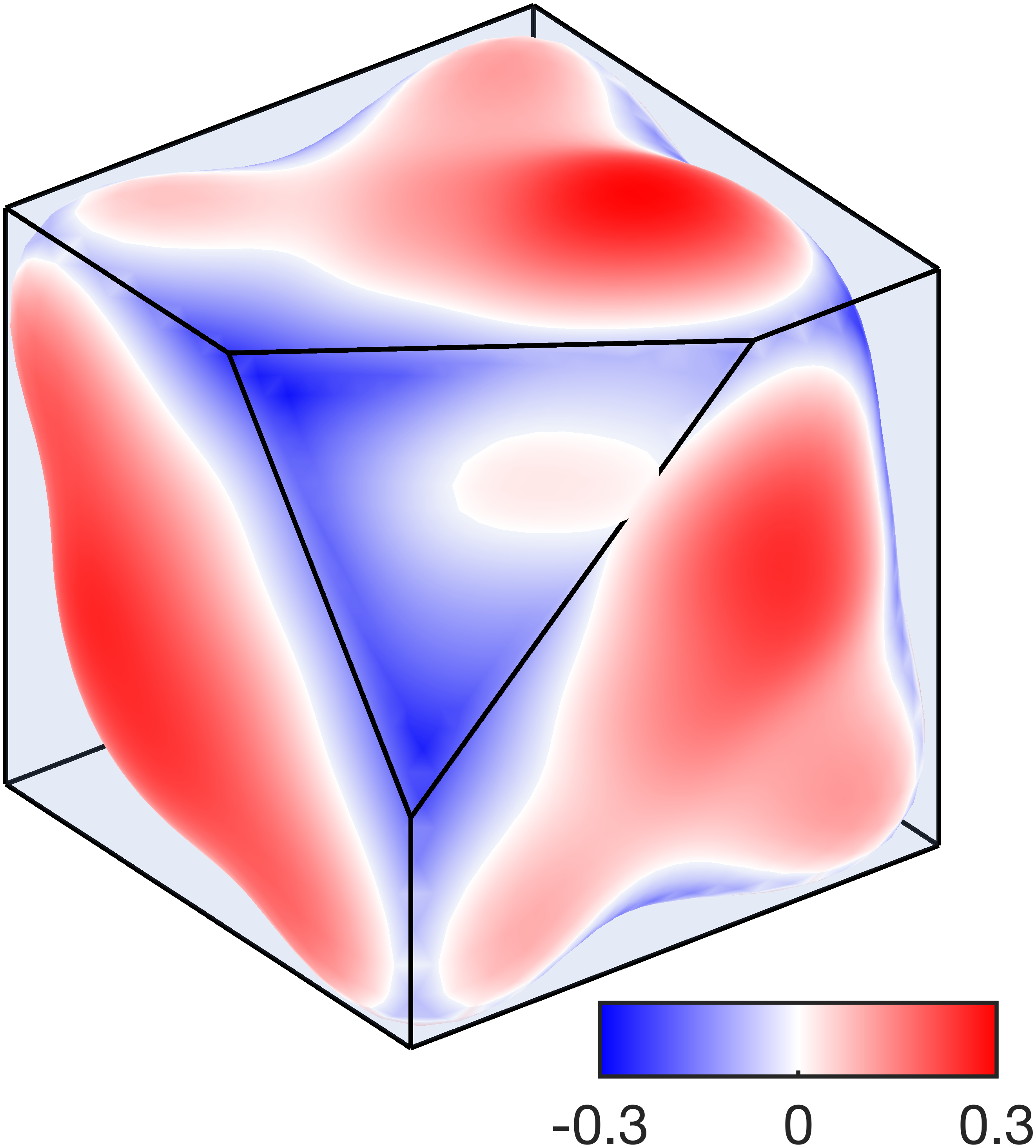} \hfill
\includegraphics[scale=.2]{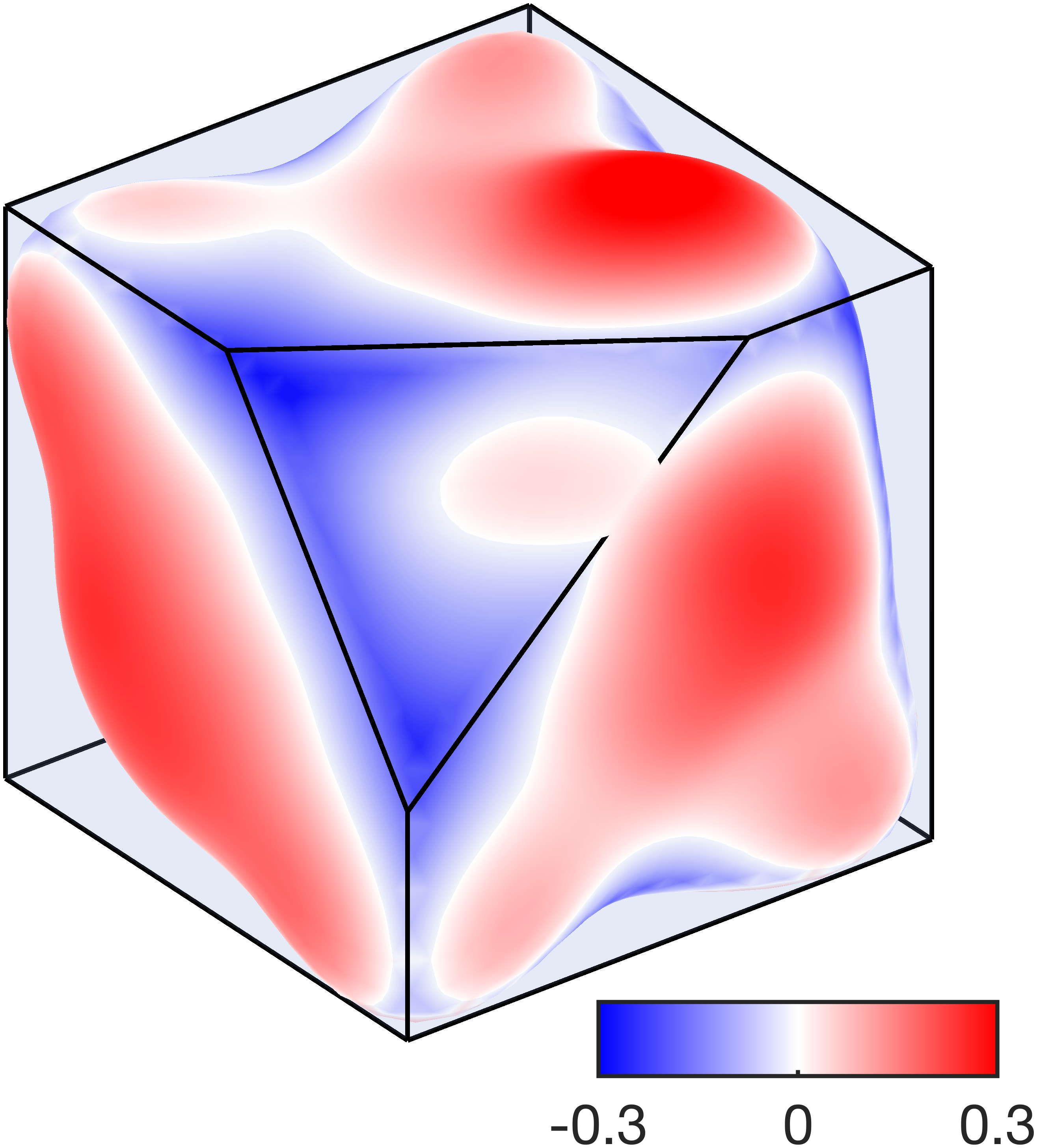}
\caption{Reconstruction of the scatterer from Example 3. The plots show the final iteration of the Newton scheme corresponding to noiseless data (left) and noisy data with $15\%$ additional noise (middle) and $30\%$ additional noise (right).}
\label{fig:rec3}
\end{figure}

\section*{Acknowledgments}
The author was supported by the Research Council of Finland 
(Flagship of Advanced Mathematics for Sensing, Imaging and Modelling 
grant~359182).
The author wishes to acknowledge CSC – IT Center for Science, Finland, for computational resources through the project Fast Algorithms for Inverse Learning (project number 2008159).

\section*{AI tool disclosure}
During the preparation of this manuscript, OpenAI's ChatGPT GPT-5.6 and Anthropic's Claude Opus 5 and 5.5 were used for language editing and literature search.
The proof of Lemma \ref{lem:TmTh} has been carried out with minor assistance from Claude.
Moreover, Claude assisted in the implementation of \eqref{def:T1h} and  \eqref{def:T2h} within the \textit{bempp} framework.
Claude created the Python scripts for verifying the numerical orders in Section~\ref{sec:numtests}.
Finally, Claude was used to identify errors in an earlier version of this manuscript.
After using these tools, their provided content was verified, reviewed and edited as needed. The author takes full responsibility for the content of this manuscript.

\small

\bibliographystyle{abbrvurl}

\bibliography{references}
\end{document}